\documentclass[10pt]{frank}

\usepackage{array}
\usepackage{enumitem}
\usepackage{scrtime} 
\usepackage{caption, microtype}

\usepackage[bottom]{footmisc}
\usepackage{datetime}
\usepackage{amsmath, amscd, amsthm}
\usepackage{mathrsfs, bbm, amsfonts, amssymb, cancel}
\usepackage{stmaryrd}
\usepackage{graphicx,array} 
\usepackage{rotating, lscape}
\usepackage[enableskew]{youngtab}
\usepackage[all]{xy}
\usepackage{color}
\usepackage{wrapfig}
\usepackage{colortab}
\usepackage{marvosym, enumitem}
\usepackage{hyperref, color, marvosym}
\definecolor{darkblue}{rgb}{0,0,1}
\hypersetup{colorlinks=true, breaklinks=true, linkcolor=darkblue, menucolor=darkblue, urlcolor=darkblue, citecolor=darkblue}

\numberwithin{equation}{section}

\newtheorem{proposition}{Proposition}
\newtheorem{theorem}[proposition]{Theorem}
\newtheorem{lemma}[proposition]{Lemma}

\theoremstyle{remark}
\newtheorem{corollary}[proposition]{Corollary}
\newtheorem{remark}[proposition]{Remark}
\newtheorem{example}[proposition]{Example}

\theoremstyle{definition}
\newtheorem{definition}[proposition]{Definition}
\newtheorem{notation}[proposition]{Notation}

\newcommand{\re}{{\rm Re}}
\newcommand{\im}{{\rm Im}}

\newcommand{\ZZZ}{\mathbbm Z}
\newcommand{\RRR}{\mathbbm R}

\newcommand{\CCC}{\mathbbm C}
\newcommand{\HHH}{\mathbbm H}

\newcommand{\X}{\mathfrak{X}}

\newcommand{\e}{\mathfrak{e}}
\renewcommand{\o}{\mathfrak{o}}
\renewcommand{\a}{{\rm a}}
\newcommand{\cl}{{\rm C\ell}}

\definecolor{grau}{rgb}{0.97,0.97,0.97}
\definecolor{hellgrau}{rgb}{0.90,0.90,0.90}
\newcommand{\grau}{\textcolor{grau}}
\newcommand{\hellgrau}{\textcolor{hellgrau}}

\usepackage{titletoc}
	
\titlecontents{section}[2.5em]{\addvspace{0.1pc}}{\contentslabel{1.25em}}{}{\titlerule*[0.5pc]{.}\contentspage}
\titlecontents{subsection}[4.75em]{\addvspace{-0.2pc}}{\contentslabel{2.2em}}{}{\titlerule*[0.5pc]{.}\contentspage}
\titlecontents{subsubsection}[7.75em]{\addvspace{-0.2pc}}{\contentslabel{4em}}{}{\titlerule*[0.5pc]{.}\contentspage}

\allowdisplaybreaks
\usepackage{comment}

\begin{document}

\sloppy \raggedbottom

\title{Supersymmetric Killing Structures}

\begin{start}
	\author{Frank Klinker\\}
	\\

\noindent
\makebox[0.8\textwidth]{%
\begin{minipage}{0.86\textwidth}
\begin{Abstract}
In this text we combine the notions of supergeometry and supersymmetry. 
We construct a special class of supermanifolds whose reduced manifolds are  (pseudo) Riemannian manifolds. 
These supermanifolds allow us to treat vector fields on the one hand and spinor fields on the other hand as equivalent geometric objects. 
This is the starting point of our definition of {supersymmetric Killing structures}. 
The latter combines subspaces of vector fields and spinor fields, provided they fulfill certain field equations. 
This naturally leads to a superalgebra that extends the supersymmetry algebra to the case of non-flat reduced space.  
We examine in  detail the additional terms that enter into this structure and  we give a lot of examples.
\\[1ex]
\\[1ex]
	\end{Abstract}
\end{minipage}}

\let\thefootnote\relax\footnote{
\hspace*{-1.4em}Faculty of Mathematics, TU Dortmund University,  44221 Dortmund, Germany\\[1ex]
\hspace*{1.em}E-mail address: \href{mailto:frank.klinker@math.tu-dortmund.de}{frank.klinker@math.tu-dortmund.de}\\[1ex]
\hspace*{1.em}This is the preprint version of {\em Comm.~Math.~Phys.}, {\bf 255} (2005), no.2, 419-467

\href{http://dx.doi.org/10.1007/s00220-004-1277-2}{DOI 10.1007/s00220-004-1277-2}
} 

\renewcommand{\dateseparator}{-}
\end{start}

\setcounter{tocdepth}{2}

\runningheads{%
\small\sf Preprint 
\hspace*{8.2cm}
\sf Supersymmetric Killing Structures
}{%
\small\sf Preprint
\hspace*{10.9cm}   
\sf Frank Klinker
}

\section*{Introduction}

One of the key ingredients of this work is the concept of supersymmetry. This topic became very famous in the 1970s in the context of the investigation of supergravity theories. In the early 1980s  supersymmetry was combined with string theory to gain superstring theory. There is a vast amount of literature on these topics of which we will cordially refer the reader to the standard books \cite{WessBagger} and  \cite{GreenSchwarzWitten}. 
From the mathematical point of view the investigation of supersymmetry is interesting because of at least two reasons. 
The first reason is that the notion of super Lie algebra naturally enters and therefore leads to an extension of the classical Lie theory. A classification of simple super Lie algebras is given in \cite{ScheunertNahm1} and \cite{ScheunertNahm2}, for example. 
From the very nature, supersymmetry is related to spin geometry. First constructions were given by Wess and Zumino in \cite{WessZumino2} and the universality of this construction was proven by Haag, {\L}opusza{\'n}ski, and Sohnius in \cite{HaagLopusSohn}.
The second reason is the occurrence of vector fields and spinor fields in this context, in particular the fact that both are treated on the same level, namely as infinitesimal transformations. 
If we consider  a supersymmetric field theory on a curved manifold, one part of the even transformations of the field configuration is  formed by vector fields on the manifold. 
One part of the odd transformations is formed by spinor fields. 
If we claim that these transformations are the generators of a certain superalgebra, we end up with vector fields  and  spinor fields that obey certain differential equations, the so called Killing equations. 
This discussion naturally leads to the question whether we can construct superalgebras by considering certain vector fields and spinor fields that solve such Killing equations. 
This is done by considering the subset of these Killing objects within the direct sum of both spaces and imposing an algebra structure. The superalgebra structure is usually achieved by defining the brackets between the certain objects in an appropriate but, more or less, ad hoc way. In \cite{Habermann1} and \cite{Fig1} the bracket  between two vector fields is the usual bracket, between two spinors it is defined to be the supersymmetry bracket and between a vector field and a spinor it is the Lie derivative of the spinor. This ad hoc definition forces us to check in each case whether or not the Jacobi Identity is fulfilled. This problem has been noticed in \cite{Habermann1} to give a negative answer to the question if the imaginary Killing spinors are the odd part of a superalgebra in general.
\\
There are several ways to take into account the concept of anticommutativity and put it in a geometric setting. 
One way is the notion of superspace. By this we mean a free graded module over a superalgebra rather than a vector space, i.e.\ the coefficients of the space are elements in a superalgebra rather than a field. 
There are some attempts to formulate gravity theories on such spaces (see \cite{WessZumino1} or \cite{BriGellRamSchw1}). 
Furthermore, flat superspace is the model space for the construction of  supermanifolds in the context of \cite{deWittBook}, \cite{Rogers1} 
or \cite{Leites1} and their investigation in e.g.\ \cite{Yagi2} or \cite{BruPes1}. 
A second main geometric way to introduce anticommuting elements  is the concept of graded manifold, see \cite{Kostant}. 
In contrast to the above construction -- where the coefficients of the model space are replaced -- in the case of graded manifolds the sheaf of functions on a manifold is replaced by a sheaf of commutative superalgebras. 
In fact, as proven in \cite{Rogers5}, the concept of deWitt manifolds is equivalent to the  concept of graded manifolds. In \cite{Rothstein} we find an axiomatic approach that covers the different kinds of supermanifolds. 
M.\ Batchelor showed in \cite{Bat1} that graded manifolds can be characterized by a manifold and a vector bundle over this manifold. 
Combining this with the description of  vector fields on the graded manifold we see that the latter are given by the vector fields on the base manifold and the sections in the vector bundle. 
This is another way to bring vector fields and  spinors together. To treat them in this way  has the advantage that we get a priori a graded Lie algebra structure on the direct sum of vector fields and spinors by the bracket of vector fields on the graded manifold.  
This structure is used in \cite{AlCorDevSem} to define  infinitesimal automorphisms  on a supermanifold that leads to  the proposition that the  twistor spinors on a Riemannian manifold are exactly the odd Killing fields on an associated supermanifold. 
One difference between the  result presented in our text to the result in \cite{AlCorDevSem} is that in our construction the bracket of two odd Killing fields does not always vanish, which means that we generate a nonvanishing motion on the base space. 

We present an ansatz for a geometric structure on a graded manifold leading to a superalgebra that recovers in lowest order the classical supersymmetry algebra. In the flat case we  recover the latter exactly. Moreover, the supersymmetric Killing structure admits an ideal, the center, whose dependence on the Killing fields is of purely algebraic nature. In general, this center turns out to be infinite. Nevertheless it is finite in the case of special  base manifolds. Although the center is infinite in the general case, we show that it admits a uniform shape. Roughly speaking, it consists of two types of elements.

\section{Special graded manifolds}

\subsection{Graded manifolds}

A {\sc graded manifold} is a locally  ringed space 
$\hat{M}=(M,\mathfrak{A})$ consisting  of a $D$ dimensional manifold $M$ and a sheaf $\mathfrak{A}=\mathfrak{A}_{0}\oplus \mathfrak{A}_{1}$  of commutative superalgebras with the following properties: 
\begin{itemize}[leftmargin=2em]
\item The reduced space  $M_{\rm red}$, which is the locally ringed space with sheaf $\mathcal{C}:=\mathfrak{A}/(\mathfrak{A}_1+(\mathfrak{A}_1)^2)$, is isomorphic to an ordinary manifold.
\item $\hat M$ is locally decomposable, i.e.\ for all  $x\in M$ there exists a neighborhood  $U$ of $x$  such that  $(U,\mathfrak{A}|_U)$ is isomorphic to $(M_{\rm red} ,\Lambda_{\mathcal{C}_{U}}\mathcal{E})$ where $\mathcal{E}$ is a rank-$r$ locally free sheaf  over $U$ of  $\mathcal{C} _{U}$ modules.
\end{itemize} 
A morphism between two locally ringed spaces is a pair of maps $(f,f^*):(M,\mathfrak{A})\to(M',\mathfrak{A}')$ with a map $f:M\to M'$ and a local morphism $f^*:\mathfrak{A}'\to f_*\mathfrak{A}$ of sheaves. The morphism is an isomorphism if $f$ is a diffeomorphism and $f^*$ an isomorphism. The pair $(D,r)$ is called the dimension of the graded manifold $\hat{M}$ and we identify the sheaf $\mathcal{C}$ with the sheaf of functions on $M$.
\begin{example}
Let $E\to M$ be a vector bundle. Denote the sheaf of sections of  the exterior bundle by $\Gamma \Lambda E$. Then $(M,\Gamma \Lambda E)$  is a graded manifold of  dimension $(\dim M, {\rm rank}E)$
\end{example}
This example is universal due to the following remarkable theorem.
\begin{theorem}[M.\ Batchelor \cite{Bat1}]
Every graded manifold over a manifold $M$ is isomorphic to $(M,\Gamma\Lambda E)$ for some vector bundle $E\to M$.
\end{theorem}
The sections in $\mathfrak{A}$ are called {\sc functions} on the graded manifold. If we choose local coordinates $(x_i)_{1\leq i\leq D}$ of $M$ and a local frame $(\theta_\alpha)_{1\leq\alpha\leq r}$ of $E$ we write the functions in the usual way
\begin{equation}
f(x_i,\theta_\alpha)=\sum_{a\in\ZZZ^r_2}f_{a}(x_i)\theta_1^{a_1}\cdots\theta_r^{a_r}.
\end{equation}
We call $(x_i,\theta_\alpha)$ a local coordinate system of $\hat M$.


On the graded manifold $\hat M$ we consider the sheaf of derivations, i.e.
\begin{equation}
\mathfrak{X}(\hat M)= \left\{X:\Gamma\Lambda E\to\Gamma\Lambda E\
 \left| \begin{array}{l}\text{X is }\RRR\text{-linear and }\\
  X(fg)=X(f)g+(-)^{Xf}fX(g)\\
  \text{for homogeneous }X,f,g\end{array} \right.\right\}\,.
\end{equation}
There exist canonical inclusions $\jmath:\X(M)\to \X(\hat M)$ and $\jmath:\Gamma E^*\to \X(\hat M)$ which lead to the isomorphism 
\begin{equation}
\X(\hat M) \cong \Gamma\Lambda E \otimes \mathfrak{X}(M) \oplus \Gamma \Lambda E  \otimes \Gamma E^*\label{isom}
\end{equation}
see \cite{Kostant}. 
In \eqref{isom} we denote the first and the second summand by $\X^v$ and $\X^s$ respectively and call the elements v-like and  s-like. The two summands we get by the splitting of $\X(\hat M)$ with respect to the natural $\ZZZ_2$-grading are denoted by $\X_0$ and $\X_1$. 
A local frame of $\hat M$ will  be a local frame of $\X(M)\oplus\Gamma E^*$, i.e.\ $\{e_i,\theta^\alpha\}$ where the two parts span the two summands.


\subsection{Special graded manifolds}

We consider a (pseudo) Riemannian spin manifold $M$ of dimension $D=s+t$ and signature $\sigma=t-s$ with spinor bundle $S$. The spinor bundle is an associated vector bundle\footnote{
We denote the bundles and the standard fibres by the same symbol.}
$S=P_{Spin(s,t)}\times_\rho S$ where the typical fibre $S$ is a spinor representation of $Spin(s,t)\subset\cl_{s,t}$. This can be a real or a complex space, see Appendix \ref{secSpingeo}. 
We sometimes add the label $\CCC$ or $\RRR$ if we want to stress the kind of fibre. $P_{Spin}$ denotes the spin  structure of $M$. This is a $Spin$-principal bundle over $M$ with a double cover $\hat\lambda$ of the $SO$-principal bundle $P_{SO}$ such that the diagram on the right  commutes.

\begin{wrapfigure}[9]{r}{185pt}
\vspace*{-10ex}
$\qquad {\xymatrix  @+0pc @ur{
&P_{Spin}\times Spin\ar[dl]_-{\hat\lambda\times\lambda}\ar[dr]^-{R}& \\
P_{SO}\times SO \ar[dr]^-{R}&&P_{Spin}\ar[dl]^-{\hat\lambda}\ar[d] \\
 &P_{SO} \ar[r]& M
}}
$
\end{wrapfigure}

The typical fibre comes with a $\mathfrak{spin}$-invariant bilinear form $C$ which may give rise to supersymmetry, depending on the dimension and the signature of $M$, see Appendix \ref{secSUSY}. The bilinear form on the standard fibre extends to a bilinear form on the bundle. 
The Levi-Civita connection $\nabla$ on $M$ induces on $S$ (and on $\Lambda S$) a connection, also denoted by $\nabla$, that respects  the Clifford multiplication 
\begin{equation}
\nabla(X\phi)=(\nabla X)\phi+X\nabla\phi.
\end{equation}
and the charge conjugation $C$
\begin{equation}
\nabla C(\phi,\psi)= C(\nabla\phi,\psi)+C(\phi,\nabla\psi).
\end{equation}
The above considerations and the constructions from the last subsection lead to the next definition which provides us with  the basic notion of manifold we need in our following constructions.

\begin{definition}
A {\sc special graded manifold}  (SGM) is a graded manifold $\hat M=(M,\mathfrak{A})$ together with a map $C$ where
\begin{itemize}[leftmargin=2em]
\item $M$ is a (pseudo) Riemannian spin manifold,
\item the sections in the sheaf $\mathfrak{A}$ are sections in the exterior bundle of a spinor bundle $S$, and
\item $C$ is the charge conjugation on $S$ and gives rise to supersymmetry.
\end{itemize}
We call the SGM {\sc real} or {\sc complex} if the structure sheaf $\mathfrak{A}$ comes from a real or  complex spinor bundle, respectively.
\end{definition}
If we consider real SGM the form of the typical fibre and the bilinear form $C$ depends on whether the representation of the $Spin$ group is real, complex or quaternionic. The next remark follows immediately from theorems \ref{whencomplex} and \ref{whenreal}.
\begin{remark}\label{whenSGM}
The base manifold of a  real special graded manifold has dimension and signature given in Theorem \ref{whenreal}, and 
the base of a complex special graded manifold is of dimension $D\leq 4$ mod 8, see Theorem \ref{whencomplex}.
\end{remark}
The charge conjugation yields the isomorphism
\begin{equation}
\mathfrak{X}(\hat M)\simeq 
\Gamma\Lambda S \otimes \mathfrak{X}(M) \oplus \Gamma\Lambda S \otimes \Gamma S \label{IsomorphFields}.
\end{equation}
In our ansatz we use an inclusion that has been proposed by \cite{AlCorDevSem}.  It is  given by
\begin{equation}\begin{aligned}
\jmath& :\X(M)\oplus\Gamma S\hookrightarrow\X(\hat M)\,,\\
\jmath& :\ X\ \mapsto\  \nabla_X\,, \\
\jmath& :\ \phi\ \mapsto\  \phi\rfloor \,,
\end{aligned}\label{iso2}\end{equation}
where $\nabla$ is the Levi-Civita connection on $M$ and the interior multiplication is defined via the charge conjugation, i.e.\ 
\begin{enumerate}[leftmargin=2.5em]
\item[(1)] $\phi\rfloor f_0=0$ for $f_0\in C^\infty(M)\subset\Gamma\Lambda S$.
\item[(2)] $\phi\rfloor\psi=C(\phi,\psi) $
for $\psi\in\Gamma S\subset\Gamma\Lambda S$.
\item[(3)] $\phi\rfloor\left(\omega\wedge\tau\right)
 =\left(\phi\rfloor\omega\right)\wedge\tau+(-)^k\omega\wedge\left(\phi\rfloor\tau\right)$
 for $\omega\in\Gamma\Lambda^kS\subset\Gamma\Lambda S,\tau\in\Gamma\Lambda S$.
\end{enumerate}

\begin{proposition}\label{basiccomm}
For $X,Y\in\X(M)$ and $\phi,\psi\in\Gamma S$ the brackets are given by
\begin{align}
\big[\jmath(X),\jmath(Y)\big]&= \nabla_X\circ\nabla_Y-\nabla_Y\circ\nabla_X
  =  R^{\Lambda S}(X,Y)+\jmath([X,Y]_0)\\
\big[\jmath(X),\jmath(\phi)\big]&=
  \nabla_X\circ\phi\rfloor-\phi\rfloor\circ\nabla_X
  =  \jmath(\nabla_X\phi)\\
\big[\jmath(\phi),\jmath(\psi)\big]&=\phi\rfloor\circ\psi\rfloor+\psi\rfloor\circ\phi\rfloor
  = 0
\end{align}
where $R$ is the curvature of $\nabla$ acting on the bundle added as a subscript, and $[.,.]_0$ denotes the usual commutator of vector fields on $M$.\footnote{We will drop the $\jmath$ in our notation if there is no danger of confusion.}
Furthermore we have the graded Jacobi identity
\begin{equation}
(-)^{XZ}\left[[X,Y],Z]\right]+(-)^{ZY}\left[[Z,X],Y]\right]+(-)^{XY}\left[[Y,Z],X]\right]=0
\end{equation}
for all $X,Y,Z\in\X(\hat M)$. 
\end{proposition}
\begin{proof}
The third bracket is a property of the interior derivation. The first bracket is the definition of the curvature, and the second bracket is a consequence of $C$ being invariant under the Levi-Civita connection.

The identity for the cyclic sum is valid for every vector field. But it should be mentioned that in our case the identity does not only follow from the same identity for the bracket $[.,.]_0$  on $M$, but that we also need the following two facts: 
\begin{itemize}[leftmargin=2em]
\item $\left[R^{\Lambda S}(X,Y),\phi\right]=R^S(X,Y)\phi$.
\item The $2^{\rm nd}$ Bianchi identity for the curvature $R$, i.e.\ $\nabla_{[i}R_{jk]}+R_{n[i}T^{n}_{jk]}=0$ where $T$ denotes the torsion of $\nabla$.
\end{itemize}
\end{proof}
The $\Gamma\Lambda S$-linear extension of this inclusion leads to the following splitting of vector fields and action on functions:
\begin{equation}
\left(f\otimes X+g\otimes\theta\right)(h) = f\wedge\nabla_Xh+g\wedge\theta\rfloor h \label{iso}
\end{equation}
for $f,g,h\in \Gamma\Lambda S$, $X\in\X(M)$ and $\theta\in\Gamma S$.

We describe the sheaf of functions $\mathfrak{A}$ of the real special graded manifold in terms of the complex spinors, see Appendix \ref{realspinorsapp}. The product structure on the real functions agrees with the one on the complex functions in the case of  real and complex representations, but differs in the case of quaternionic representation.
\begin{itemize}[leftmargin=2em]
\item  {\it The real case}.\ \

We have $S\subset S_\CCC$ and the product in $\Gamma\Lambda S$  is given by  the restriction of the wedge product in $\Gamma \Lambda S_\CCC$. The evaluation $\Gamma S\times \Gamma S\to C^\infty M$ is given by the bilinear form $C$.

A local basis  $\{\varphi_\alpha\}_{1\leq \alpha\leq r}$ of $S_\CCC$  gives rise to a local real basis of $S$ if we consider the set\footnote{In fact only one half of these $2r$ vectors are linearly independent.}  $\left\{\frac{1}{2}(\varphi_\alpha+\tau(\varphi_\alpha)),\frac{1}{2i}(\varphi_\beta-\tau(\varphi_\beta))\right\}_{1\leq\alpha,\beta\leq r}$.

\item  {\it The complex case}.\ \

The story is the same as in the real case. An additional feature of the complex case is the fact that we have the chiral structure not compatible with the the real spin structure but reversed by the conjugation.

For a complex basis $\left\{\varphi^+_\alpha\right\}_{1\leq\alpha\leq\frac{r}{2}}$ of $S_\CCC^+$ our preferred choice of local basis of $S$ is  $\left\{(\varphi_\alpha^+,\varphi_\alpha^-),(i\varphi_\beta^+,-i\varphi_\beta^-)\right\}_{1\leq\alpha,\beta\leq \frac{r}{2}}$ with $\varphi^-_\alpha:=\tau(\varphi_\alpha^+)\in S^-_\CCC$.

\item {\it The quaternionic case}.\ \

We have $S\subset S_\CCC\oplus S_\CCC $ with charge conjugation given by Remark \ref{bilinear}, i.e.\ $C({\rm H},\Xi)=\re(\bar\eta\xi)$ or $=\im(\bar\eta\xi)$, respectively. 

Let $\{\varphi_\alpha\}_{1\leq\alpha\leq r}$ be a local complex basis of  $S_\CCC$. A real basis of $S$ is then given by 
$\{ \Phi_\alpha=(\varphi_\alpha,\varphi_\alpha^C), \Phi_{\beta+r}=(i\varphi_\beta,-i\varphi_\beta^C)\}_{1\leq\alpha,\beta\leq r}$. We will also use the more common notation $\{\phi_\alpha=(\varphi_\alpha,0),\phi_{\bar \alpha}=(0,\varphi_\alpha^C)\}$  which is related to the additional $R$-symmetry. 
The relation between these notations is
\begin{equation}
\phi_\alpha=\frac{1}{2}(\Phi_\alpha-i\Phi_{\alpha+r}) \ \text{ and }\
\phi_{\bar\alpha}=\frac{1}{2}(\Phi_\alpha+i\Phi_{\alpha+r})\label{adaptedbasis}
\end{equation}
i.e.\ $\Psi= a^\alpha\Phi_\alpha+b^\alpha\Phi_{\alpha+r}
=z^\alpha\phi_\alpha+\bar z^\alpha\phi_{\bar\alpha}$ with $z^\alpha=a^\alpha+i b^\alpha$.
\end{itemize}


\subsection{Natural bilinear forms on special graded manifolds}

The two key ingredients for the vector fields on $\hat M$ are given by the vector fields on $M$ and the sections in $S$. They both come with bilinear maps given by the metric $g$ and the charge conjugation $C$. These two maps give rise to a bilinear form on the SGM. 

A {\sc bilinear form} on the graded manifold $(M,\mathfrak{A})$ is a map
\begin{equation}
b:\X(\hat M)\times\X(\hat M) \to \mathfrak{A}
\end{equation}
with
\begin{itemize}[leftmargin=2em]
\item $b$ is additive in both entries
\item $b$  is linear with respect to the multiplication by even elements in both arguments and with respect to odd elements in the first argument. In the second argument the latter comes with a sign.
\end{itemize}
We call a bilinear form {\sc homogeneous} if for all odd functions $f\in\mathfrak{A}_1$ we have
\begin{equation}
b(X,fY)= (-)^{b+X}f\,b(X,Y)
\end{equation}
with a fixed sign $(-)^b$. This provides the space of bilinear forms with a natural grading and turns it into a $\mathfrak{A}$ module. $b$ is called {even} or  {odd} if this sign is $+$ or $-$, respectively.

A homogeneous bilinear form is called
\begin{itemize}[leftmargin=2em]
\item\label{basic}
 {\sc basic} if for all $X,Y\in \X(M)\oplus\Gamma S$ we have $\pi b(X,Y)=b(X,Y)$
\item\label{susybilinear}
{\sc supersymmetric} if 
\begin{equation}
b(X,Y)=(-)^{b(X +Y)+XY}b(Y,X)
\end{equation}
and {\sc skew supersymmetric}  if
\begin{equation}
b(X,Y)=-(-)^{b(X +Y)+XY}b(Y,X)
\end{equation}
for all $X,Y\in\X$.
\end{itemize}

The metric and the charge conjugation give rise to an even  basic bilinear form $h$ on $\hat M$, called the {\sc metric},  by considering the direct sum $g\oplus C$ and its bilinear extension in the above sense, i.e.\ we define
\begin{equation}
h(\jmath X,\jmath Y)=g(X,Y),\quad
h(\jmath\eta,\jmath\xi)=C(\eta,\xi),\quad
h(\jmath X,\jmath\xi)=0
\end{equation}

\begin{remark}
$h$ is a supersymmetric bilinear form if the charge conjugation $C$ on the spinor bundle is skew-symmetric.
 If $C$ is symmetric $h$ is neither supersymmetric nor skew supersymmetric. 
If we denote by $D$ the  dimension of the base manifold $M$  and by  $r$ the rank of the spinor bundle,
the algebra that preserves this bilinear forms is 
\begin{itemize}[leftmargin=2em]
\item the orthosymplectic algebra $\mathfrak{osp}(D,r)\subset\mathfrak{gl}_{D|r}$ with even part $\mathfrak{osp}(D,r)_0\simeq\mathfrak{so}(D)\times\mathfrak{sp}(\frac{r}{2})$ in the first case, and 
\item the purely even sub algebra $\mathfrak{so}(D)\times\mathfrak{so}(r)$ of $\mathfrak{gl}_{D|r}$ in the second case (see \cite{ScheunertNahm1}).
\end{itemize}
\end{remark}

\begin{definition}
The Lie derivative $\mathcal{L}_X b$ of a bilinear form $b$ in direction $X\in\X$ is defined by
\begin{equation}
\mathcal{L}_Xb(Y,Z):=(-)^{XY}\big(X(b(Y,Z)) -b([X,Y],Z)-(-)^{X(b+Y)}b(Y,[X,Z])\big)\,.
\label{LieMetric}
\end{equation}
\end{definition}
\begin{remark}
The Lie derivative $\mathcal{L}_Xb$ of $b$ in direction $X$ is itself a bilinear form and its grading is given by $|\mathcal{L}_Xb|=|b|+|X|$.
\end{remark}

\section{Fields on special graded manifolds}

\subsection{Fields and spinors}

Some fields\footnote{
From now on we call the vector fields on the supermanifold $\hat M=(M,\mathfrak{A})$ fields. We do this to distinguish them from vector fields on the base manifold $M$.}
on the special graded manifold $\hat M$ arise naturally by considering sections in the spinor bundle $S$ or in the tensor product $S\otimes TM$.  More precisely, we consider  s-like fields of the form $\varphi\in\Gamma S\subset\X^s\cap\X_1$ and  v-like fields of the form $\Phi\in\Gamma S\otimes\X(M)\subset\X^v\cap\X_1$. Therein  the parallel spinors as well as the inclusion of $\Gamma S$ into $\Gamma S\otimes \X(M)$, see \eqref{spinordiagram},  turn out to be of particular interest for our construction.
\begin{equation}\label{spinordiagram}
\xymatrix{
&&&\Gamma S\ar@<0.3ex>[dl]^{\imath}\\
\Gamma S\ar[r]^-{\nabla}&\Gamma S\otimes\X(M)\ar[r]^-{\simeq}&\Gamma S\oplus\Gamma S_{\frac{3}{2}}\quad  \ar@<0.3ex>[ur]^-{\pi_1}\ar[dr]_-{\pi_2}&\\
&&&\Gamma S_{\frac{3}{2}}
}
\end{equation}
Introducing the orthonormal frame $\{e_i\}$ and their images $\{\gamma_i\}$ in the Clifford algebra the projections and the inclusion are respectively given by
\begin{align}
\pi_1(\Phi) & = \gamma^i\Phi(e_i)\,,\nonumber \\
\imath(\varphi) &=-\frac{1}{D} \gamma^k\varphi\otimes e_k\,, \nonumber \\ \intertext{and}
\pi_2(\Phi)&= \Phi-\imath\circ\pi_1(\Phi)=\Phi+\frac{1}{D}\gamma^j\gamma^k\Phi(e_k)\otimes e_j\,, \nonumber 
\end{align}
in particular $\imath(\varphi)(e_j)=-\frac{1}{D}\gamma_j\varphi$, $\pi_1(\varphi^k\otimes e_k)=\gamma^k\varphi_k$,  and $\pi_2(\phi^k\otimes e_k)(e_j)=\frac{D-1}{D}\phi_j + \frac{1}{D}\gamma_{jk}\phi^k $.

\subsection{Fields and endomorphisms}

One summand in the even part of $\X(\hat M)$ which is of special interest is
\begin{equation}
{\rm End}(\Gamma S)\ \simeq\  \Gamma S\otimes\Gamma S^*\ \stackrel{C}{\simeq}\ 
\Gamma S\otimes\Gamma S\ \subset\ \X(\hat M).
\end{equation}
In local coordinates $\{\theta_\alpha\}$ of $S$ the endomorphism $\Phi$ is expressed as
$\Phi\phi=\Phi^\alpha_\beta\phi^\beta\theta_\alpha$ and we get
\begin{equation}
\Phi=\Phi^\alpha_\beta C^{\beta\gamma}\theta_\alpha\otimes\theta_\gamma .\label{localendo}
\end{equation}
The fields in ${\rm End}(\Gamma S)$ act on functions via $\Phi(f)=0$ for $f\in C^\infty(M)\subset\Gamma\Lambda S$,  $\Phi(\vartheta)=  \Phi\vartheta$ for $\vartheta\in\Gamma S\subset\Gamma \Lambda S$,
and  on elements in $\Gamma\Lambda^k S$ for $k\geq 2$ via the natural extension as derivations.

The various brackets that  will be frequently used are given in the next proposition.

\begin{proposition}\label{unique}
Let $\Phi\in\Gamma({\rm End}(S))$.
If we consider $\Phi$ as an even field it has the following properties
\begin{equation}\begin{aligned}
\big[\Phi,\phi\big] 	&= -\Phi^{T_C}\phi &
\big[\Phi,X\big]		&= -\nabla_X\Phi   \\
\big[\Phi_1,\Phi_2\big]	&= \Phi_1\Phi_2-\Phi_2\Phi_1 &&
\end{aligned}\end{equation}
for $\phi\in\Gamma S\subset\X(\hat M)$ and $X\in\X(M)\subset\X(\hat M)$. 
\end{proposition}
The right hand side of the last bracket is given by the composition of endomorphisms and ${\ }^{T_C}$ denotes the  transpose  with respect to $C$, i.e.\  $C(\Phi^{T_C}\psi,\varphi)=C(\psi,\Phi\varphi)$ and is locally given by $(\Phi^{T_C})^\beta_\alpha= C_{\alpha\gamma}\Phi^\gamma_\sigma C^{\sigma\beta}$.

The next definition is dedicated to another  natural assignment.
\begin{definition}\label{E-endomorphism}
Let $\mathfrak{B}^{(k)}:\Gamma(\otimes^k T^*M)\otimes {\rm End}\Gamma S\otimes \Gamma (\otimes^kS)\to \X(\hat M)$ be given by
\begin{equation}
\mathfrak{B}^{(k)}(A;\xi_1,\ldots,\xi_k)=\gamma^{i_1}\xi_1\wedge\ldots\wedge\gamma^{i_k}\xi_k\owedge A_{i_1\cdots i_k}.
\end{equation}
with $\xi_\alpha\in\Gamma S$ and $A \in \Gamma (\otimes^kT^*M)\otimes {\rm End}(\Gamma S)$. For a function $g\in\Gamma\Lambda S$ this is 
\begin{equation*}
\mathfrak{B}^{(k)}(A;\xi_1,\ldots,\xi_k)(g) \ =\ \gamma^{i_1}\xi_1\wedge\ldots\wedge\gamma^{i_k}\xi_k\wedge A_{i_1\cdots i_k}(g) .
\end{equation*}
\end{definition}
The operation $\owedge$ is always used when we want to emphasize the fact that the field, that  a function  is  multiplied with, is not of order zero (in that case we used $\otimes$) but a field of higher order.

The map $\mathfrak{B}^{(k)}$ is even with respect to the $\ZZZ_2$ grading on the SGM in the sense that we have
\begin{equation}\begin{aligned}
|\mathfrak{B}^{(k)}(A;\xi_1,\ldots,\xi_k)| &= |\mathfrak{B}^{(k)}|+|A|+|\xi_1|+\cdots+|\xi_k| \\
 &= |\xi_1|+\cdots+|\xi_k|
\end{aligned}
\end{equation}
for all $A$ and $\xi_i$ like above. $\mathfrak{B}^{(k)}$ is one of the maps that  will play in important r\^ole in Section \ref{secCenter}. 

A special class of  fields from endomorphisms is given by the endomorphisms which are skew-symmetric with respect to $C$. In this case the bracket with odd zero-order fields and the action  on  functions of order one coincide.
The elements of the form  $\frac{1}{4}A_{ij}\gamma^{ij}$ are of this special type because of the $\mathfrak{spin}$-invariance of $C$. 
For example, the curvature of the Levi-Civita connection is of this kind.

In the next step we associate a field in $\X(\hat M)$ to an endomorphism of the tangent bundle of $M$ --- more precisely, to an element of ${\rm End}\X(M)$.

Therefore let $A\in{\rm End}(\X(M))$ which has  the local form ${A_k}^j=A_{ki}g^{ij}$ with respect to an orthonormal frame.
Consider the assignment
\begin{equation}
A\ \mapsto\ \frac{1}{4}A_{ij}\gamma^i\gamma^j \label{ass}
\end{equation}
Let us denote  by $A_\a$ the endomorphism we get by skew-symmetrizing $A$ with respect to the Riemannian metric on $M$.
In local coordinates we will write 
$A_\a^{ij}=A^{[ij]}=\frac{1}{2}(A^{ij}-A^{ji})$ with $A^{ij}=g^{ik}{A_k}^j$.
This skew-symmetric endomorphism $A_\a$  is mapped in ${\rm End}(\Gamma S)$.
The image is called $A_\a=\frac{1}{4}A_{[ij]}\gamma^{ij}$ too. $A_\a$ is skew-symmetric with respect to $C$.

We see that \eqref{ass} can be written as
\begin{align}
\frac{1}{4}A_{ij}\gamma^i\gamma^j
 & =  \frac{1}{8}\big(A_{ij}\gamma^i\gamma^j+A_{ij}\gamma^i\gamma^j\big)
  =   \frac{1}{8}\big(A_{ij}-A_{ji}\big)\gamma^i\gamma^j-\frac{1}{4}A_{ij}g^{ij} \nonumber\\
 & = \frac{1}{4}A_{[ij]}\gamma^{ij}-\frac{1}{4}A_{ij}g^{ij}
  =  A_\a - \frac{1}{4}{\rm tr}(A)\,.
\label{ass2}
\end{align}

In this section we restricted ourself to some natural constructions of  fields that involve endomorphisms. One can, of course, think of  other and more complicated  assignments. We will come back to this later.

\subsection{Even Killing fields of special graded manifolds}

We define even fields on the special graded manifold from vector fields on the base manifold. This is done by extending the inclusion $\jmath$. A natural definition leads to the conclusion that  the even Killing fields of the SGM are given by the Kiling vector fields on the base manifold. Furthermore we add some comments on conformal vector fields.

Consider the fields given by
\begin{equation}
\X(M)\ni X\ \mapsto\ \mathfrak{G}(X) = X+\Phi(X)\in\X_0
\end{equation}
where $\Phi(X)\in \Gamma S\otimes\Gamma S\subset\X^s\cap \X_0$ for all $X$ .
\begin{proposition}\label{satz20}
For the vector field $\mathfrak{G}(X)=X+\Phi(X)$ we have
\begin{equation}
\big[\mathfrak{G}(X),\mathfrak{G}(Y)\big]= \mathfrak{G}([X,Y]_0)+R(X,Y)+\nabla\Phi(X,Y)+\Phi\wedge\Phi(X,Y)
\end{equation}
with
\[
\nabla\Phi(X,Y):=\nabla_X(\Phi(Y))-\nabla_Y(\Phi(X))-\Phi([X,Y]_0)
\]
and
\[
\Phi\wedge\Phi(X,Y):=\Phi(X)\Phi(Y)-\Phi(Y)\Phi(X).
\]
\end{proposition}
\begin{remark}
For $\Phi$ being a 1-form on  $M$ with values in ${\rm End}S$ the abbreviations given in the proposition are just the usual formulas for bundle-valued forms and its exterior covariant derivative.
\end{remark}
\begin{proof}
We use Proposition \ref{basiccomm} to calculate
\begin{align*}
\big[\mathfrak{G}(X),\mathfrak{G}(Y)\big] 
&= [X,Y]+\big[X,\Phi(Y)\big]+\big[\Phi(X),Y\big]+\big[\Phi(X),\Phi(Y)\big]\\
&=  R(X,Y)+[X,Y]_0+\nabla_X(\Phi(Y))-\nabla_Y(\Phi(X)) +\Phi\wedge\Phi(X,Y)\\
&=  [X,Y]_0+\Phi([X,Y]_0)+ R(X,Y)+\nabla\Phi(X,Y)+\Phi\wedge\Phi(X,Y)\\
&= \mathfrak{G}([X,Y]_0)+ R(X,Y)+\nabla\Phi(X,Y)+\Phi\wedge\Phi(X,Y)
\end{align*}
\end{proof}
\begin{definition}
An even field $W\in\X(\hat M)$ is called an {\sc even Killing field of the SGM} if it satisfies
\begin{equation}
\mathcal{L}_{W}h = 0.
\end{equation}
\end{definition}
We recall the definition of the Lie-derivative $\mathcal{L}$ of the metric $h$:
\[
(-)^{WY}(\mathcal{L}_W h)(Y,Z) = W(h(Y,Z))-h([W,Y],Z)-(-)^{WY}h(Y,[W,Z])
\]
 for all $W,Y,Z\in\X$.
\begin{theorem}\label{7}
Let $X$ be a vector field on the base manifold $M$ and
\begin{equation}
\e_\alpha(X) = X-(\nabla X)_\a+\alpha\,{\rm tr}(\nabla X) \label{evenansatz}
\end{equation}
where $(\nabla X)_\a$ is the image of the skew symmetrized endomorphism $\nabla X$ of $\X(M)$ and ${\rm tr}$ denotes its trace, see Construction \eqref{ass2}. Then
$
\e_\alpha(X)$  is an even Killing field on the SGM $\hat M$ if and only if
$X$   is a  Killing vector field on $M$.

Furthermore, for two even Killing fields we have
\begin{equation*}
\big[\e_\alpha(X),\e_\alpha(Y)\big] = \e_\alpha([X,Y]_0).
\end{equation*}
\end{theorem}
\begin{remark}
Consider an orthonormal frame $\{e_i\}$ of $M$ and a frame $\{\theta_\kappa\}$ of $S$. This gives rise to a frame $\{\jmath(e_i),\jmath(\theta_\kappa)\}$ of the SGM $\hat M$, see \eqref{iso2}. Furthermore, consider the connection coefficients $\nabla_ie_j=\omega_{ij}{}^ke_k$ . With respect to this local structure and if we recall \eqref{localendo}  the field $\e_\alpha(e_i)$ is written as 
\[
\e_\alpha(e_i)  = 
\jmath(e_i) 
-\tfrac{1}{4}\omega_{jik}(\gamma^{jk}{})_{\mu}^{\lambda}C^{\mu\kappa}\theta_\lambda\otimes\jmath(\theta_\kappa  )
+\alpha\omega_{ji}{}^j   C^{\lambda\kappa}\theta_\lambda\otimes \jmath(\theta_\kappa)\, .
\] 
\end{remark}
For the proof of theorem \ref{7} we need some facts on (conformal) Killing vector fields summarized in the next lemma.
\begin{lemma}\label{lemma1}
\begin{enumerate}[leftmargin=2.5em]
\item[(1)] \label{1a}For all vector fields $X,Y$ the following identity holds.
\begin{equation}
R(X,Y) = \nabla_X(\nabla Y)-\nabla_Y(\nabla X)-\nabla[X,Y]_0- [\nabla X,\nabla Y]
\end{equation}
\item[(2)]\label{Kobayashi2} The endomorphism $\nabla X$ is skew symmetric, i.e.\ $(\nabla X)_\a=\nabla X$, if and only if $X$ is a Killing vector field.
\item[(3)]\label{3} For a Killing vector field $X$ and for all vector field $Y$ we have
\begin{equation}
\nabla_Y(\nabla X)=-R(X,Y).\label{CurvatureKilling}
\end{equation}
 Furthermore for two Killing vector fields $X,Y$ we have
\begin{equation}
[\nabla X,\nabla Y]=-\nabla[X,Y]_0+R(X,Y).
\end{equation}
\item[(4)] For two conformal fields we have
\begin{equation}
{\rm tr}(\nabla[X,Y]_0)=X({\rm tr}(\nabla Y))-Y({\rm tr}(\nabla X)).
\end{equation}
\end{enumerate}
\end{lemma}
\begin{remark}
The first equality in \eqref{3} holds if and only if $X$ is an  infinitesimal affine transformations, and  the second  if both, $X$ and $Y$, are. 
That means that  the natural lifts of the vector fields to the frame bundle preserve the connection form. In particular each Killing vector field is an infinitesimal affine transformation for the Levi-Civita connection.
\end{remark}
\begin{proof}\;[Lemma \ref{lemma1}.]\ \
(2) is just the definition of a Killing vector field and the calculations for (1) and (3) can for example be found in \cite{KobNom}.

To prove (4) we recall the definition of conformal vector fields. A vector field is a conformal vector field if  $\mathcal{L}^0_Xg=f_X\cdot g$ for a function $f_X$ that is given by a multiple of the divergence of the vector field $X$, $f_X=\frac{2}{D}{\rm div} X=\frac{2}{D}{\rm tr}(\nabla X)$. Now the statement is the change of the conformal factor under the bracket $f_{[X,Y]_0}=Xf_Y-Yf_X$.
\end{proof}

\begin{lemma}\label{lemma2}
For $\e_\alpha(X), \e_\alpha(Y)$ as above we have
\begin{equation}
\big[\e_\alpha(X),\e_\alpha(Y)\big]=\e_\alpha([X,Y]_0)-\Xi(X,Y)+\alpha\Theta(X,Y).
\end{equation}
With
\begin{align}
\Xi(X,Y)&=\frac{1}{4}
\left((\nabla X)_{\rm s}(\nabla Y)_{\rm s}-(\nabla Y)_{\rm s}(\nabla X)_{\rm s}\right)_{ij}
\gamma^{ij}\\
\Theta(X,Y)&= X({\rm tr}(\nabla Y))-Y({\rm tr}(\nabla X))-{\rm tr}(\nabla[X,Y]_0)
\end{align}
$\Xi$ vanishes if one of the entries is a conformal vector field and $\Theta$ if both entries are conformal vector fields.
\end{lemma}

\begin{proof}
We insert $\Phi(X)=-(\nabla X)_a+\alpha\,{\rm tr}(\nabla X)$ in Proposition \ref{satz20} and make use of the following facts:
\begin{itemize}[leftmargin=2em]
\item $\nabla_Y$ commutes with $A\mapsto A^t$ where the transposition ${}^t$ is taken  with respect to $g$.
\item $[A_\a,B_\a]+[A_{\rm s},B_{\rm s}]=[A,B]_\a$,
 i.e.\ $\frac{1}{4}[A-A^t,B-B^t]+\frac{1}{4}[A+A^t,B+B^t]
 =\frac{1}{2}([A,B]-[A,B]^t)$.
\item Part 1.\ of the previous lemma together with $R(X,Y)_\a=R(X,Y)$.
\end{itemize}
The vanishing of $\Xi$ for a conformal vector field  is obvious, because in this case $\mathcal{L}^0_Xg$ is a multiple of the identity. The  statement on $\Theta$  is part 4.\ of the previous lemma.
\end{proof}

\begin{proof}\;[Theorem \ref{7}]\ \
For the calculation of the Lie derivative we restrict ourself to entries in $\X(M)\oplus\Gamma S$, because of the $\Gamma\Lambda S$-linearity.
\begin{align*}
(\mathcal{L}_{\e(X)} h)(A,B) 
&= (\mathcal{L}_Xh)(A,B)-(\mathcal{L}_{(\nabla X)_\a}h)(A,B)
     +(\mathcal{L}_{\alpha{\rm tr}(\nabla X)}h)(A,B)\\
&= X(h(A,B))-h([X,A],B)-h(A,[X,B])\nonumber\\
&\quad -\underbrace{(\nabla X)_\a(h(A,B))}_{=0}
      +h([(\nabla X)_\a,A],B)+h(A,[(\nabla X)_\a,B])\\
&\quad  +\underbrace{\alpha\,{\rm tr}(\nabla X)(h(A,B))}_{=0}
      +h([\alpha\,{\rm tr}(\nabla X),A],B)
      +h(A,[\alpha\,{\rm tr}(\nabla X),B])
\end{align*}
For two v-like entries $A,B$ this reads as
\begin{align*}
(\mathcal{L}_{\e(X)}h)(A,B)
&= 		X(g(A,B))-g([X,A]_0,B)-g(A,[X,B]_0)-\underbrace{h(R(X,A),B)}^{=0}\\
&\quad -\underbrace{h(A,R(X,B))}_{=0}+\underbrace{h([(\nabla X)_\a,A],B)}_{=0}
		+\underbrace{h(A,[(\nabla X)_\a,B])}_{=0}\\
&\quad +\underbrace{h([\tfrac{\alpha}{4}{\rm tr}(\nabla X),A],B)}_{=0}
		+\underbrace{h(A,[\tfrac{\alpha}{4}{\rm tr}(\nabla X),B])}_{=0} \\
&= 		(\mathcal{L}^0_Xg)(A,B).
\end{align*}
For two s-like entries we get
\begin{align*}
(\mathcal{L}_{\e(X)}h)(A,B) 
&= 		\underbrace{XC( A,B)-C( \nabla_XA,B) -C( A,\nabla_XB)}^{=0}\\
&\quad  + \underbrace{C( (\nabla X)_\a A,B) 
		+C( A,(\nabla X)_\a B) }^{=0}
		-2\alpha\,{\rm tr}(\nabla X)\cdot C( A,B) \\
&=  	-2\alpha\,{\rm tr}(\nabla X)\cdot C( A,B)\,.
\end{align*}
Last but not least one v-like entry $A$ and one s-like entry $B$ give
\begin{align*}
(\mathcal{L}_{\e(X)}h)(A,B) 
&= X(\underbrace{h(A,B)}^{=0})-\underbrace{h([X,A]_0,B)}^{=0}
    -h( R(X,A),B)
   -\underbrace{h(A,\nabla_XB)}^{=0}\\
 &\quad -h(\nabla_A((\nabla X)_\a),B)+\underbrace{h(A,(\nabla X)_\a B)}_{=0}
    -\alpha\,A({\rm tr}(\nabla X))\cdot h( {\bf 1},B)\\
 &\quad    +\underbrace{h(A,[\alpha\,{\rm tr}(\nabla X),B])}_{=0}\\
&=  -h( R(X,A)+\nabla_A((\nabla X)_\a),B) -\alpha\,A({\rm tr}(\nabla X)) B,
\end{align*}
where we used $h(\Phi,\phi)=\Phi\phi\in\Lambda\Gamma S$ for fields $\Phi\in\Gamma S\otimes\Gamma S$ and $\phi\in\Gamma S$.

For $A=A^v+A^s,B=B^v+B^s\in\X(M)\oplus\Gamma S$ the last three terms together are
\begin{align*}
(\mathcal{L}_{\e(X)}h)(A, B)
& = (\mathcal{L}^0_Xg)(A^v,B^v)
     -2\alpha\,{\rm tr}(\nabla X)\cdot C(  A^s,B^s) \\
&\quad - \alpha\cdot h( A^v\big({\rm tr}(\nabla X)\big){\bf 1},B^s
    -\alpha\cdot h( A^s, B^v\big({\rm tr}(\nabla X)\big){\bf 1}) \\
&\quad - h( A^s,R(X,B^v)+(\nabla_{B^v}(\nabla X))_\a )
    - h( R(X,A^v)+(\nabla_{A^v}\nabla X))_\a,B^s)\,.
\end{align*}
Here the first summand vanishes if and only if $X$ is a Killing vector field. But then, because of part \ref{3}.\ of Lemma \ref{lemma1}, the last two summands are  identical zero, too. The same is true for the remaining parts, because $\nabla X$ is trace free, see (\ref{Kobayashi2}) of Lemma \ref{lemma1}.

Furthermore, for Killing vector fields $X,Y$  we have
\begin{equation}
[\e_\alpha(X),\e_\alpha(Y)] = \e_\alpha([X,Y]_0)
\end{equation}
because of Lemma \ref{lemma2}.
\end{proof}
In addition to part (\ref{1a}) of Lemma \ref{lemma1} we state the next lemma.
\begin{lemma}\label{conformlemma}
\begin{enumerate}[leftmargin=2.5em]
\item[(1)]For all vector fields on $M$ we have
\begin{equation}
(\nabla_V\mathcal{L}^0_Yg) = 2(\nabla_V(\nabla Y))_{\rm s}
\end{equation} and
\begin{equation}\begin{aligned}
g(R(Y,Z)V,W)&= (\nabla_V\mathcal{L}^0_Yg)(W,Z)-(\nabla_W\mathcal{L}^0_Yg)(V,Z) 
          +g(\nabla_W(\nabla Y)(Z),V) \\&\quad-g(\nabla_V(\nabla Y)(Z),W)
\end{aligned}\end{equation}
\item[(2)] For a conformal vector field $Y$ with divergence function $f$ and for all vector fields $Z$ we have
\begin{equation}
R(Y,Z)=\frac{1}{2}df\wedge Z-\nabla_Z(\nabla Y)_\a
\end{equation}
where we write $df\wedge Z\,(V,W)=V(f)g(Z,W)-W(f)g(Z,V)$.
\end{enumerate}
\end{lemma}
\begin{proof}
We have 
\begin{align*}
(\nabla_V\mathcal{L}^0_Yg) (W,Z)
&= V\big((\mathcal{L}^0_Yg)(W,Z)\big) -(\mathcal{L}^0_Yg)(\nabla_VW,Z)-(\mathcal{L}^0_Yg)(W,\nabla_VZ)\\
&= V(g(\nabla_WY,Z)+g(W,\nabla_ZY))-g(\nabla_{\nabla_VW}Y,Z)-g(\nabla_VW,\nabla_ZY)\\&\quad 
 -g(\nabla_WY,\nabla_VZ)-g(W,\nabla_{\nabla_VZ}Y)\\
&= g\big(\nabla_V\nabla_WY-\nabla_{\nabla_VW}Y,Z\big)
    +g\big(W,\nabla_V\nabla_ZY-\nabla_{\nabla_VZ}Y\big)\\
&= g\big(\nabla_V(\nabla Y)(W),Z\big)+g\big(W,\nabla_V(\nabla Y)(Z)\big)\\&
= g(2(\nabla_V(\nabla Y)_{\rm s}(W),Z\big)
\end{align*}
which is the first part of (1). Together with $A^T=2A_{\rm s}-A$ this  yields
\begin{align*}
(\nabla_V\mathcal{L}^0_Yg)&(W,Z)-(\nabla_W\mathcal{L}^0_Yg)(V,Z) +g(\nabla_W(\nabla Y)(Z),V) 
-g(\nabla_V(\nabla Y)(Z),W)\\
&= g(\nabla_V(\nabla Y)(W),Z)-g(\nabla_W(\nabla Y)(V),Z)\\ 
&= g(R(V,W)Y,Z)\\& =\ g(R(Y,Z)V,W)\,.
\end{align*}
(2) Let $Y$ be conformal, i.e.\ $\mathcal{L}^0_Yg=f\cdot g$, which because of $\nabla g=0$ implies $\nabla_V\mathcal{L}^0_Yg=V(f)\cdot g$. This, together with (1) and a frequent use of the $1^{\rm st}$ Bianchi identity for $R$, yields
\begin{align*}
&g(R( Y,Z)V,W) \\
&=  V(f)\cdot g(W,Z)-W(f)\cdot g(V,Z)+g(\nabla_W(\nabla Y)(Z),V)
  -g(\nabla_V(\nabla Y)(Z),W)\\
 &= df\wedge Z\,(V,W)+g(\nabla_Z(\nabla Y)(W),V)-g(\nabla_Z(\nabla Y)(V),W)
  +\,g(R(W,Z)Y,V)\\&\quad-g(R(V,Z)Y,W)\\
 &= df\wedge Z\,(V,W)-2g\big(\nabla_Z(\nabla Y)_\a(V),W\big)
  -g(R(Y,Z)V,W)
\end{align*}
\end{proof}
\begin{corollary}
Let $\e_\alpha(X)$ like before with $\alpha=\mp\frac{1}{2D}$. Furthermore let $X$ be a conformal vector field on the base manifold $M$ and $f=\frac{2}{D}{\rm tr}(\nabla X)$ its divergence function. Then we get for the Lie derivative of the bilinear form $h$
\begin{align}
\big(\mathcal{L}_{\e(X)}h\big)(V,W) =\ & f\cdot g(V,W)\nonumber \\
\big(\mathcal{L}_{\e(X)}h\big)(\phi,\psi) =\ &
\begin{cases}
\frac{1}{2}f\cdot C( \phi,\psi) \\
-\frac{1}{2}f\cdot C(\phi,\psi)
\end{cases} \\
\big(\mathcal{L}_{\e(X)}h\big)(V,\phi)
   =\ & 
\begin{cases}
-(df\cdot V^T)\cdot\phi\\
(V\cdot df^T)\cdot\phi
\end{cases} \nonumber
\end{align}
where the action of $(df\cdot V^T)$ on the spinors is according to \eqref{ass}.
\end{corollary}
\begin{proof}
Let $X$, $\alpha$ and $f$ as stated. From the proof of Theorem \ref{7} we know
\begin{align*}
\big(\mathcal{L}_{\e(X)}h\big)(V,W)& = (\mathcal{L}^0_Xg)(V,W) = f\cdot g(V,W)\\
\big(\mathcal{L}_{\e(X)}h\big)(\phi,\psi)& = -2\alpha\,{\rm tr}(\nabla X)\cdot C(\phi,\psi)
 =  \pm\frac{1}{2}f\cdot C(\phi,\psi)
\end{align*}
We recall ${\rm tr}(df\cdot V^T) = V(f)$ and 
\begin{equation}
\frac{1}{2}df\wedge V =  \frac{1}{8}\big((df\cdot V^T)_{ij}-(df\cdot V^T)_{ji}\big)\gamma^{ij}
 =(df\cdot V^T)_\a=-(V\cdot df^T)_\a.
\nonumber
\end{equation}
Using part (2) of Lemma \ref{conformlemma} this yields
\begin{align*}
\big(\mathcal{L}_{\e(X)}h\big)(V,\phi) =\ & -h( R(X,V)+(\nabla_{V}(\nabla X))_\a
            +\alpha\,V\big({\rm tr}(\nabla X)\big){\bf 1},\phi )\\
=\ &-h( \frac{1}{2}df\wedge V
            +\frac{D\alpha}{2}\,V(f){\bf 1},\phi)\\
=\ & 
\left.\begin{cases}
   - h(  (df\cdot V^T)_\a-\frac{1}{4}\,V(f){\bf 1},\phi ) \\
      h( (V\cdot df^T)_\a-\frac{1}{4}\,V(f){\bf 1},\phi )
\end{cases}\negthickspace\negthickspace\negthickspace\right\}
 =  \begin{cases}
-(df\cdot V^T)\cdot\phi  \\
 (V\cdot df^T)\cdot\phi
\end{cases}
\end{align*}
\end{proof}

\section{Supersymmetric Killing structures}

\subsection{Definition of the SSKS}

We saw that if we take $S$ to be the (real or complex) spinor bundle over the (pseudo) Riemann\-ian spin manifold $M$ the vector fields and the spinor fields are on equal footing as (complexified) vector fields on the special graded manifold $\hat M=(M,\mathfrak{A})$, denoted by $\X(\hat M)$. We recall the natural inclusions $\X(M)\subset\X(\hat M)_0$ and $\Gamma S\subset\X(\hat M)_1$.

Inspired by the construction of the supersymmetry generators in the flat case we consider maps
\begin{equation}
\begin{array}{rcl}
\X(M)\ni X &\mapsto& \e(X)\in\X(\hat M)_0\\
\Gamma S \ni\eta&\mapsto&\o(\eta)\in\X(\hat M)_1
\end{array}\label{assi}
\end{equation}
and define
\begin{definition}\label{defiSSKS}
A {\sc supersymmetric Killing structure} (SSKS) on $M$ is a tupel $(\mathcal{K},\mathfrak{Z},\e,\o)$, where $\mathcal{K}=\mathcal{K}_0\oplus\mathcal{K}_1$  is a sub-superspace of $\X(M)\oplus \Gamma S$, $\mathfrak{Z}\subset\X(\hat M)$  and  $\e,\o$ are given by \eqref{assi} such that for all $X,Y\in\mathcal{K}_0$ and  $\eta,\xi\in\mathcal{K}_1$ we have
\begin{equation}
\begin{aligned}
 \left[\e(X),\e(Y)\right] &= \e([X,Y]) \\
 \left[\e(X),\o(\eta)\right] &\in \o(\mathcal{K}_1)  \\
 \left[\o(\eta),\o(\xi)\right] &\in \e(\mathcal{K}_0) +\mathfrak{Z} \\
 \left[\o(\eta),\mathfrak{Z}\right] &\subset\mathfrak{Z},\qquad
 \left[\e(X),\mathfrak{Z}\right]\subset\mathfrak{Z},\qquad
 \left[\mathfrak{Z},\mathfrak{Z}\right]\subset\mathfrak{Z}
\end{aligned}\end{equation}
 i.e.\ $\e(\mathcal{K}_0)\oplus\mathfrak{Z}\oplus \o(\mathcal{K}_1)$ is a sub super Lie algebra of $\X(\hat M)$. The ideal  $\mathfrak{Z}$ which occurs in this construction is called {\sc center} and depends multilineary on $\mathcal{K}_1$. The fields in $\e(\mathcal{K}_0)$ and  $\o(\mathcal{K}_1)$ are called {\sc even} and  {\sc odd Killing fields of the SSKS}, respectively.
\end{definition}

The central part $\mathfrak{Z}$ of the supersymmetric Killing structure seems at first glance to be very arbitrary. The first thing, which is obvious, is that it is  not central in the usual sense, but in the sense that it is stable under the action of $\mathcal{K}$, i.e.\ it is a $\mathcal{K}$-module. Nevertheless, we call this part the center of the SSKS. The elements in $\mathfrak{Z}$ shall only depend on the odd part of $\mathcal{K}$ and, of course, on geometrical data of the base manifold, in particular on its curvature. The dependence on the elements in  $\mathcal{K}_1$ is multilinearly, i.e.\ an element  $\mathcal{Z} \in\mathfrak{Z}$ can be written as
\begin{equation}
\mathcal{Z}=\mathcal{Z}(\varphi_1\otimes\cdots \otimes \varphi_m).
\end{equation}
We call the supersymmetric Killing structure {\sc finite}  if  $\mathfrak{Z}$ is finite as a module.
We will see that finiteness  is a very restrictive property. Nevertheless, this will be the case we will focus on first.

\subsection{Odd and even fields}

Motivated by the splitting of $S\otimes TM$ from \eqref{spinordiagram} and by using the inclusion of the vector fields on $M$ into $\X(\hat M)$ we construct even and odd vector fields on the special graded manifold. These will give rise to the odd and even Killing fields of the SSKS on $\hat M$. In the flat case the ansatz reduces to the usual realization of the SUSY algebra by vector fields on flat superspace.

The even fields are given by the assignment
\begin{equation}
\X(M) \ni X\mapsto \e_0(X)=\e(X)\in\X_0(\hat M)\,.
\end{equation}
see \eqref{evenansatz}. The Killing vector fields on the base manifold will turn out to play an important role for the construction of the Killing fields of the SSKS as they did for the construction of the even Killing fields of the SGM.

Turning to the odd fields we
consider the vector fields $\o_\pm(\phi)\in\X_1$ which are given by
\begin{equation}
\o_\pm(\phi)\ =\ \phi\pm\imath(\phi)\ \in\ \X^s \oplus \X^v\label{oddansatz1}
\end{equation}
with $\phi\in\Gamma S$. $ \imath(\phi)\in\ \Gamma S\otimes\X(M)$ is given by $\imath(\phi)(Y) = Y\phi$ for all $Y\in\X(M)$. 
This is the inclusion 
$\Gamma S\hookrightarrow \Gamma S\otimes\X(M)$ from \eqref{spinordiagram}
up to the factor $-D^{-1}$.

\begin{remark}
Consider an orthonormal frame $\{e_i\}$ of $M$ and a frame  $\{\theta_\kappa\}$ of $S$. If we write the image $\gamma_i\in{\rm End}(\Gamma S)$ of $e_i$ in this frame, the local form of $\o_\pm(\theta_\kappa)$ with respect to the frame $\{\jmath(e_i),\jmath(\theta_\kappa)\}$ is given by
\[
\o_\pm(\theta_\kappa)=\jmath(\theta_\kappa)\pm(\gamma^k)_\kappa^\lambda\theta_\lambda\otimes \jmath(e_k)
\,.
\]
\end{remark}
Although we restrict ourself later to the positive sign in the above construction, we calculate the result here for both signs. In the flat case the two signs are related to whether we define the odd generator via right or left multiplication on superspace (see \cite{WessBagger}). In particular in flat space the two generators commute for different sign, because left and right multiplication commute. This fact is reflected in Theorem \ref{oddoddcomm} and  \eqref{plusminus} therein.

For the sake of completeness we state the result the Lie derivative of the metric $h$ with respect to the field $\o(\phi)$.
\begin{proposition}
The Lie derivative of the metric $h$ with respect to the odd field $\o_+(\phi)$ is given by 
\begin{align*}
(\mathcal{L}_{\o_+(\phi)}h)(X,Y) & =  X\nabla_Y\phi+Y\nabla_X\phi \\
(\mathcal{L}_{\o_+(\phi)}h)(\eta,\xi) & =  0 \\
(\mathcal{L}_{\o_+(\phi)}h) (X,\xi) & = \gamma^k\phi\wedge R(X,e_k)\xi - C(\xi,X\phi) +C(\nabla_X\phi,\xi).
\end{align*}
If we in particular consider the spinor $\phi$ to be parallel, we are left with 
\begin{equation}\begin{aligned}
(\mathcal{L}_{\o_+(\phi)}h)(X,Y) 	&= (\mathcal{L}_{\o_+(\phi)}h)(\eta,\xi)  =  0 \\
(\mathcal{L}_{\o_+(\phi)}h) (X,\xi) &= \gamma^k\phi\wedge R(X,e_k)\xi - C(\xi,X\phi) .
\end{aligned}\end{equation}
\end{proposition}
\begin{proof}
Recalling \eqref{LieMetric} and  Proposition \ref{basiccomm} gives the result by calculations similar to those we will perform next. 
\end{proof}

\subsection{Odd-odd and even-odd commutators}

\begin{lemma}\label{9}
Let $Y\in\X(M)$, $\phi,\psi\in\Gamma S$.
Then we have the following commutation relations:
\begin{align}
\big[Y,\imath(\phi)\big]
 &= \imath(\nabla_Y\phi)+\gamma^k\phi\owedge R(Y,e_k)-\gamma^k\phi\otimes\nabla_kY\nonumber\\
\big[\phi,\imath(\psi)\big]
 &=\tfrac{1}{2}\big\{\psi,\phi\big\} +\gamma^k\psi\otimes\nabla_k\phi\\
\big[\imath(\phi),\imath(\psi)\big]
 &= \mathfrak{B}^{(2)}(R;\phi,\psi)+\gamma^k\phi\owedge\imath(\nabla_k\psi)+\gamma^k\psi\owedge\imath(\nabla_k\phi)\nonumber
\end{align}
\end{lemma}

\begin{proof}
The proof needs the following observation. Let $\eta$ be an arbitrary spinor and  $\omega^k_l$  the local (skew-symmetric) connection form with respect to  the orthonormal frame $\{e_i\}$.
\begin{equation}
\gamma^k\eta\otimes \nabla e_k=\omega_k^l\gamma^k\eta\otimes e_l=-(-\omega^l_k\gamma^k)\eta\otimes e_l=-(\nabla\gamma^k)\eta\otimes e_k.
\end{equation}
This is the same as to say that the projection $\pi_1$ in diagram \eqref{spinordiagram} is parallel. We make frequent use of Proposition \ref{basiccomm} to get 
\begin{align*}
\big[\phi,\imath(\psi)\big]
& = \big[\phi,\gamma^k\psi\otimes e_k\big]
  = C(\phi,\gamma^k\psi) e_k-\gamma^k\psi\otimes\big[\phi,e_k\big]\\
& =\tfrac{1}{2}\big\{\phi,\psi\big\}+\gamma^k\psi\otimes\nabla_k\phi.
\end{align*}
In the same way we calculate
\begin{align*}
\big[Y,\imath(\psi)\big]&= \big[Y,\gamma^k\psi\otimes e_k\big] \\
&=  \nabla_Y(\gamma^k\psi)\otimes e_k+\gamma^k\psi\otimes\big[Y,e_k\big]\\
&=  \gamma^k\nabla_Y\psi\otimes e_k+(\nabla_Y\gamma^k)\psi\otimes e_k
  +\gamma^k\psi\otimes\big[Y,e_k\big]_0
  +\gamma^k\psi\owedge R(Y,e_k)\\
&=  \imath(\nabla_Y\psi)+(\nabla_Y\gamma^k)\psi\otimes e_k
  +\gamma^k\psi\otimes\nabla_Ye_k-\gamma^k\psi\otimes\nabla_{k}Y 
 +\gamma^k\psi\owedge R(Y,e_k)\\
&=  \imath(\nabla_Y\psi)-\gamma^k\psi\otimes\nabla_{k}Y
  +\gamma^k\psi\owedge R(Y,e_k).
\end{align*}
These two brackets together with Definition \ref{E-endomorphism} give the third one
\begin{align*}
\big[\imath(\phi), \imath(\psi)\big]
  &= \big[\gamma^k\phi\otimes e_k,\gamma^l\psi\otimes e_l\big]\\
 &=  \gamma^k\phi\wedge\nabla_k(\gamma^l\psi)\otimes e_l
   +\gamma^l\psi\owedge\big[e_l,\gamma^k\phi\otimes e_k\big]\nonumber\\
&= \gamma^k\phi\wedge\nabla_k(\gamma^l\psi)\otimes e_l
   +\gamma^l\psi\owedge\imath(\nabla_l\phi)
   -\gamma^l\psi\wedge\gamma^k\phi\otimes\nabla_{e_k}e_l 
 +\gamma^l\psi\wedge\gamma^k\phi\owedge R_{lk}\nonumber\\
&= \gamma^k\phi\wedge\gamma^l\nabla_k\psi\otimes e_l
   +\gamma^l\psi\owedge\imath(\nabla_l\phi)
   + \mathfrak{B}^{(2)}(R;\phi,\psi)\nonumber\\
&= \mathfrak{B}^{(2)}(R;\phi,\psi)
   +\gamma^k\phi\owedge\imath(\nabla_k\psi)+\gamma^l\psi\owedge\imath(\nabla_l\phi).
\end{align*}
\end{proof}

\begin{theorem}\label{oddoddcomm}
For the fields $\o_\pm(\phi)$ we have
\begin{equation}\begin{aligned}
\big[\o_\pm(\phi),\o_\pm(\psi)\big] =\ &   
\pm\big\{\phi,\psi\big\}+ \mathfrak{B}^{(2)}(R;\phi,\psi) 
   \pm\gamma^k\phi\owedge\o_\pm(\nabla_k\psi)
   \pm\gamma^k\psi\owedge\o_\pm(\nabla_k\phi)
\end{aligned}\label{plusplus}\end{equation}
and
\begin{equation}\begin{aligned}
\big[\o_+(\phi) ,\o_-(\psi)\big] =\ &   -\mathfrak{B}^{(2)}(R;\phi,\psi) 
  -\gamma^k\psi\owedge\o_+(\nabla_k\phi)+\gamma^k\phi\owedge\o_-(\nabla_k\psi)
\end{aligned}\label{plusminus}\end{equation}
\end{theorem}

\begin{proof} By Lemma \ref{9} we get with $[\phi,\psi]=0$
\begin{align*}
\big[\o_{\pm}(\phi),\o_{\pm\varepsilon}(\psi)\big] =\ &
\big[\phi\pm\imath(\phi),\psi\pm\varepsilon\imath(\psi)\big]\\
=\ & \pm\,\varepsilon\big[\phi,\imath(\psi)\big]\pm\big[\imath(\phi),\psi\big]
   +\varepsilon\big[\imath(\phi),\imath(\psi)\big]\\
=\ & \pm\tfrac{1+\varepsilon}{2}\{\phi,\psi\}
    +\varepsilon\mathfrak{B}^{(2)}(R;\phi,\psi)
     \pm\gamma^k\phi\owedge\big(\nabla_k\psi\pm\varepsilon\imath(\nabla_k\psi)\big)
	\\&\quad \pm\varepsilon\gamma^k\psi\owedge\big(\nabla_k\phi\pm\imath(\nabla_k\phi)\big)\\
=\ & \pm\tfrac{1+\varepsilon}{2}\{\phi,\psi\}
    +\varepsilon\mathfrak{B}^{(2)}(R;\phi,\psi)
    \pm\,\gamma^k\phi\owedge\o_{\pm\varepsilon}(\nabla_k\psi) 
  \pm\,\varepsilon\gamma^k\psi\owedge\o_{\pm}(\nabla_k\phi),
\end{align*}
which for $\varepsilon\in\{\pm1\}$ proves the proposition.
\end{proof}

If we restrict ourself to parallel spinors certain terms in \eqref{plusplus} and \eqref{plusminus} vanish and we get

\begin{corollary}
Let $\phi$ and $\psi$ be parallel spinors on the manifold $M$, then the commutators in Theorem \ref{oddoddcomm} reduce to
\begin{align}
\big[\o_\pm(\phi),\o_\pm(\psi)\big] &=
   \pm\big\{\phi,\psi\big\}+\mathfrak{B}^{(2)}(R;\phi,\psi)
   \nonumber \\
\big[\o_+(\phi),\o_-(\psi)\big] &=-\mathfrak{B}^{(2)}(R;\phi,\psi)
\end{align}
\end{corollary}

The next step is to combine the even Killing fields and the odd fields defined above. First we calculate the commutator for arbritrary vector fields and spinors before we restrict ourself to Killing vector fields and parallel spinors.

\begin{lemma}
For all $X+\phi\in\X(M)\oplus\Gamma S$ we have with $\e(X)=X-(\nabla X)_\a$ and $\o_\pm(\phi)=\phi\pm\imath(\phi)$
\begin{equation}\begin{aligned}
\big[\e(X),\o_\pm(\phi)\big] =\ & \o_\pm(\e(X)\phi)\pm \gamma^k\phi\owedge \left(R(X,e_k)+\nabla_k(\nabla X)_\a\right)\\
& \mp\big( (\nabla X)_{\rm s}(e^k)\big)\phi\otimes e_k 
\end{aligned}\label{hier}\end{equation}
\end{lemma}
\begin{proof}
With Proposition \ref{unique}\ and Lemma \ref{9}\ we have 
\begin{align*}
\big[\e(X),\o_\pm(\phi)\big]&= \big[X-(\nabla X)_\a,\phi\pm\imath(\phi)\big] \\
&=  \big[\e(X),\phi\big]
   \pm\,\big[X,\imath(\phi)]\mp[(\nabla X)_\a,\imath(\phi)\big]\\
&=  \e(X)\phi
  \pm\imath(\nabla_X\phi)\,\pm\,\gamma^k\phi \owedge R(X,e_k)
  \mp\gamma^k\phi\otimes\nabla_kX
 \mp(\nabla X)_\a\gamma^k\phi\otimes e_k \\&\quad
  \mp\gamma^k\phi\owedge \big[(\nabla X)_\a,e_k\big]\\
&= \e(X)\phi\,\pm\, \imath(\nabla_X\phi)
    \,\mp\,\gamma^k(\nabla X)_\a\phi\otimes e_k 
\pm\,\gamma^k\phi\owedge \left(R(X,e_k) \nabla_k(\nabla X)_\a\right) \\&\quad
  \mp\left(\gamma^k\phi\otimes\nabla_kX
  +\big[(\nabla X)_\a,\gamma^k\big]\phi\otimes e_k\right)\\
&= \e(X)\phi\,\pm\, \imath(\e(X)\phi)
  \pm\,\gamma^k\phi\owedge\left(R(X,e_k)+ \nabla_k(\nabla X)_\a\right)\\
&\quad \mp\,\left(\gamma^k\phi\otimes\nabla_kX
  +\big[(\nabla X)_\a,\gamma^k\big]\phi\otimes e_k\right).
\end{align*}
We recall
$[A,v]=Av$ for all $A\in\mathfrak{so}(V)$ and $v\in V$ considered as subsets of  $\cl(V)$.
So the last term can be rewritten as
\begin{align*}
\gamma^k  \phi\otimes \nabla_kX+\big[(\nabla X)_\a \gamma^k\big]\phi\otimes e_k 
& = \gamma^k\phi\otimes\nabla_kX+g^{kn}\big((\nabla X)_\a(e_n)\big)\phi\otimes e_k \\
&= g^{ij}g(\nabla_kX,e_i)\gamma^k\phi\otimes e_j
 +\tfrac{1}{2}\big(g^{kj}g(\nabla_jX,e_i)\gamma^i\phi\otimes e_k \\
 &\quad -g^{kj}g(e_j,\nabla_iX)\gamma^i\phi\otimes e_k\big) \\
&= \tfrac{1}{2}g^{kj}g(\nabla_iX,e_j)\gamma^i\phi\otimes e_k
 +\tfrac{1}{2}g^{kj}g(\nabla_jX,e_i)\gamma^i\phi\otimes e_k\\
&= \tfrac{1}{2}g^{kj}\big(g(\nabla_iX,e_j)+g(\nabla_jX,e_i)\big)\gamma^i\phi\otimes e_k\\
&= g^{kj}\big((\nabla X)_{\rm s}(e_j)\big)\phi\otimes e_k 
\end{align*}
which proves the statement.
\end{proof}

\begin{theorem}
Let $\mathcal{K}_0$ be the set of Killing vector fields  and $\mathcal{K}_1$ be the set of parallel spinors on $M$.  For $X\in\mathcal{K}_0$ and $\varphi\in\mathcal{K}_1$ we have
\begin{equation}
\big[\e(X),\o_\pm(\phi)\big]=\o_\pm(\e(X)\phi) \label{evenodd}
\end{equation}
and the  result  is in $\o_\pm(\mathcal{K}_1)$.
\end{theorem}
\begin{proof}
The vanishing of the second and third summand in \eqref{hier} follows  from Lemma \ref{lemma1}, because $X\in\mathcal{K}_0$.

For the second part we have to show $\nabla_Y(\e(X)\phi)=0$ for all vector fields $Y$ and $\phi\in\mathcal{K}_1$.
The proof needs the following fact which is a supplement to Lemma \ref{lemma1} and the key ingredient for the proof of Lemma \ref{lemma1} part \ref{3}.
\begin{lemma}
For all  Killing vector fields $X$, vector fields $Y$ and  spinors $\psi$ we have
\begin{equation}
\big[\nabla_X-\nabla X,\nabla_Y\big]\psi=\nabla_{[X,Y]_0}\psi\, . \label{LieCov}
\end{equation}
\end{lemma}
For a vector field $Y$ we get
\begin{equation}
\nabla_Y(\e(X)\phi)=-\nabla_Y(\overbrace{\nabla_X\phi}^{=0}-\nabla X\phi)=\overbrace{\nabla_{[X,Y]_0}\phi}^{=0}-(\nabla_X-\nabla X)(\overbrace{\nabla_Y\phi}^{=0})=0.
\end{equation}
This shows that the commutator is an element in $\mathcal{K}_1$.
\end{proof}


\subsection{Further commutator identities}\label{Secthirdorder}

In this section we  calculate commutators which connect spinors and vector fields with the third order field $\mathfrak{B}^{(2)}(R;\ldotp,\ldotp)$.

\begin{proposition}\label{thirdordercomm}
The following relations hold for a Killing vector field $X$ and parallel spinors $\varphi,\psi,\eta$.
\begin{align}
\left[\e(X),\mathfrak{B}^{(2)}(R;\varphi,\psi)\right]&=\mathfrak{B}^{(2)}(R;\e(X)\varphi,\psi))
                                                                                           +\mathfrak{B}^{(2)}(R;\e(X)\psi,\varphi))
\label{thirdorderandeven} \\
\big[\o(\varphi),\mathfrak{B}^{(2)}(R;\eta,\xi)\big]& =  \mathfrak{B}^{(3)}(\nabla R;\varphi,\eta,\xi)
-R^m{}_{jkl}\gamma^j\varphi\wedge\gamma^k\eta\wedge\gamma^l\xi\otimes\nabla_m
\label{thirdorderandodd}
\end{align}
\end{proposition}
\begin{lemma} 
For  spinors $\varphi,\xi$ and $\eta$ we have the following commutators.
\begin{align}&\begin{aligned} 
\big[\varphi,\mathfrak{B}^{(2)}(R;\eta,\xi)\big]&=
 \gamma^l\eta\owedge R(\{\varphi,\xi\},e_l)+\gamma^l\xi\owedge R(\{\varphi,\eta\},e_l)\\ 
&\quad  -\gamma^k\eta\wedge\gamma^l\xi\otimes R_{kl}\varphi \label{thirdorderandodd1} 
\end{aligned}\\
&\begin{aligned} 
\big[\imath(\varphi),\mathfrak{B}^{(2)}(R;\eta,\xi)\big] &= 
\mathfrak{B}^{(3)}(\nabla R;\varphi,\eta,\xi) 
-R^m{}_{jkl}\gamma^j\varphi\wedge\gamma^k\eta\wedge\gamma^l\xi\otimes\nabla_m\\
&\quad + \big(\gamma^j\varphi\wedge\gamma^k\nabla_j\eta\wedge\gamma^l\xi
+\gamma^j\varphi\wedge\gamma^k\eta\wedge\gamma^l\nabla_j\xi\big)\owedge R_{kl}\\
&\quad -\gamma^mR_{kl}\varphi\wedge\gamma^k\eta\wedge\gamma^l\xi\otimes\nabla_m\label{thirdorderandodd2}
\end{aligned} \end{align}
\end{lemma}
\begin{proof}
For $\varphi,\eta,\xi\in\Gamma S$ we have
\begin{align*}
\big[\varphi, \mathfrak{B}^{(2)}(R;\eta,\xi)\big] &=
\left[\varphi,\gamma^k\eta\wedge\gamma^l\xi\owedge R_{kl}\right]\\
&= \{\varphi,\eta\}^k\gamma^l\xi\owedge R_{kl}
-\{\varphi,\xi\}^l\gamma^k\eta\owedge R_{kl}
-\gamma^k\eta\wedge\gamma^l\xi\otimes R_{kl}\varphi\\
&=\ \gamma^m\xi\owedge R(\{\varphi,\eta\},e_m)+\gamma^m\eta\owedge R(\{\varphi,\xi\},e_m)
-\gamma^k\eta\wedge\gamma^l\xi\otimes R_{kl}\varphi
\end{align*}
and 
\begin{align*}
\big[\imath(\varphi),\mathfrak{B}^{(2)}(R;\eta,\xi)\big]& =
\left[\gamma^m\varphi\otimes\nabla_m,\gamma^k\eta\wedge\gamma^l\xi\owedge R_{kl}\right]\\
&=\gamma^m\varphi\wedge\gamma^k\eta\wedge\gamma^l\xi \owedge (\nabla_mR)_{kl}
+\gamma^m\varphi\wedge\gamma^k\nabla_m\eta\wedge\gamma^l\xi \owedge R_{kl}\\
&\quad+\gamma^m\varphi\wedge\gamma^k\eta\wedge\gamma^l\nabla_m\xi \owedge R_{kl}
-\gamma^k\eta\wedge\gamma^l\xi\wedge R_{kl}(\gamma^m\varphi)\otimes\nabla_m\displaybreak[0]\\
&=\gamma^m\varphi\wedge\gamma^k\eta\wedge\gamma^l\xi \owedge (\nabla_mR)_{kl}
+\gamma^m\varphi\wedge\gamma^k\nabla_m\eta\wedge\gamma^l\xi \owedge R_{kl}\\
&\quad+\gamma^m\varphi\wedge\gamma^k\eta\wedge\gamma^l\nabla_m\xi \owedge R_{kl}
 -\gamma^k\eta\wedge\gamma^l\xi\wedge \gamma^m R_{kl}\varphi\otimes\nabla_m\\
&\quad-\gamma^k\eta\wedge\gamma^l\xi\wedge R^m{}_{nkl}\gamma^n\varphi\otimes\nabla_m
\end{align*}
which is the result if we recall Definition \ref{E-endomorphism}.
\end{proof}

\begin{proof}\,[Proposition \ref{thirdordercomm}]\ 
We emphasize the following identity which is a consequence of \eqref{CurvatureKilling} and valid for a Killing vector field $X$ and vector fields $Y,Z$.
\begin{equation}
\big[\nabla X, R(Y,Z)\big]= (\nabla_XR)(Y,Z)+R(\nabla_YX,Z)+R(Y,\nabla_ZX).
\end{equation}
The covariance of the expression $\mathfrak{B}^{(2)}(R;\varphi,\psi)$ yields
\begin{align*}
\big[\nabla_Z,\mathfrak{B}^{(2)}(R;\varphi,\psi)\big]
&=\mathfrak{B}^{(2)}(\nabla_ZR;\varphi,\psi)+ \mathfrak{B}^{(2)}(R;\nabla_Z\varphi,\psi) 
	+\mathfrak{B}^{(2)}(R;\varphi,\nabla_Z\psi)
\end{align*}
These formulas yield
\begin{align*}
\big[\e(X)\,,\mathfrak{B}^{(2)}(R;\varphi,\psi)\big] 
&= \big[\nabla_X,\mathfrak{B}^{(2)}(R;\varphi,\psi)\big]
	-(\nabla X)\gamma^k\varphi\wedge\gamma^l\psi\owedge R_{kl} \\
&\quad 	-\gamma^k\varphi\wedge(\nabla X)\gamma^l\psi\owedge R_{kl} 
-\gamma^k\varphi\wedge\gamma^l\psi\owedge[\nabla X,R_{kl} ]\\
&=  \mathfrak{B}^{(2)}(\nabla_XR;\varphi,\psi)
	+\mathfrak{B}^{(2)}(R;\nabla_X \varphi,\psi)
	+\mathfrak{B}^{(2)}(R;\varphi,\nabla_X\psi) \\
&\quad -[\nabla X,\gamma^k]\varphi\wedge\gamma^l\psi\owedge R_{kl}
	-\gamma^k(\nabla X)\varphi\wedge\gamma^l\psi\owedge R_{kl} \\
&\quad 	-\gamma^k\varphi\wedge[\nabla X,\gamma^l]\psi\owedge R_{kl}
	-\gamma^k\varphi\wedge\gamma^l(\nabla X)\psi\owedge R_{kl} \\
&\quad	-\gamma^k\varphi\wedge\gamma^l\psi\owedge((\nabla_XR)_{kl}+R(\nabla_kX,e_l)+R(e_k,\nabla_lX))\\
&= \mathfrak{B}^{(2)}(R;\e(X)\varphi,\psi) +\mathfrak{B}^{(2)}(R;\varphi,\e(X)\psi)
\end{align*}
because 
$[\nabla X,\gamma^l]\psi\owedge R_{kl} 
 =(\nabla_lX)_m\gamma^m\psi\owedge R_{kl}
= -\gamma^m\psi\owedge R(e_k,\nabla_mX)$.

If we restrict ourself to parallel spinors $\varphi,\xi$ and $\eta$  the identities \eqref{thirdorderandodd1} and \eqref{thirdorderandodd2} yield \eqref{thirdorderandodd}, 
if we  recall \eqref{CurvatureKilling} and notice $R_{kl}\varphi=0$ which makes \eqref{thirdorderandodd1} vanish.
\end{proof}

As a consequence of the Bianchi identities $\nabla_{[j}R_{kl]}=R^m{}_{[jkl]}=0$ we get the following corollary.
\begin{corollary}\label{easy}
For a  parallel spinor $\varphi$ the commutator $\left[\o(\varphi),\mathfrak{B}^{(2)}(R;\varphi,\varphi)\right]$ vanishes.
\end{corollary}

Next we turn to the commutators between two third order terms. An easy calculation which makes use of $R(X,Y)\phi=0$ for a parallel spinor yields
\begin{proposition}
For parallel spinors $ \varphi,\psi,\xi,\eta$ we have
\begin{equation}
\left[\mathfrak{B}^{(2)}(R;\varphi,\psi),\mathfrak{B}^{(2)}(R;\xi,\eta)\right]
= \mathfrak{B}^{(4)}(\widetilde R; \varphi,\psi,\xi,\eta) 
\label{26}
\end{equation}
with
\begin{equation}
\begin{aligned}
\widetilde{R}(X,Y,U,V)=\ &\left[R(X,Y),R(U,V)\right]-R(R(X,Y)U,V)-R(U,R(X,Y)V)\\
 &+ R(R(U,V)X,Y)+R(X,R(U,V)Y) \label{27}
\end{aligned}
\end{equation}
or $ \widetilde{R}_{klmn}
  =\big[R_{kl},R_{mn}\big]-2R_{kl}{}^p{}_{[m}R_{n]p}+2R_{mn}{}^p{}_{[k}R_{l]p}$.
\end{proposition}
\begin{remark}
The symmetries of $\widetilde R$ do not force expression \eqref{26} to vanish - neither in general nor in the case of two or three equal entries. The case of four equal entries is obvious not only by representation theoretic considerations but it is clear from the skew symmetry of the bracket.
\end{remark}

\subsection{A first example}

In the above calculations of the commutators we always emphasized the case when the entries are parallel spinors. 
This is due to the fact that in Theorem \ref{oddoddcomm} we have terms which depend on the covariant derivatives of the spinorial entries. 
But going back to the Definition \ref{defiSSKS}  we see that these terms are not allowed to enter into the center $\mathfrak{Z}$. So it is natural to consider the following subsets of $\X(M)$ and $\Gamma S$.
\begin{align} 
\mathcal{K}_0&=\left\{\text{Killing vector fields on } M \right\}&\mathcal{K}_1&=\left\{\text{Parallel spinors on } M\right\}
\end{align}
When we are  going to dicuss supersymmetric Killing structures on (pseudo) Riemannian manifolds of this particular kind  we are  mainly left at the investigation of higher order fields contained in  the center $\mathfrak{Z}$.. Nevertheless we also have to make sure that the  manifolds with which we deal admit these certain kind of vector fields and in particular the special kind of spinors to get non-trivial  algebras.

For example, the set $\mathcal{K}_0$ is empty when we consider a compact simply connected Riemannian manifold with nonempty $\mathcal{K}_1$. This is due to the fact that the presence of a parallel spinor forces the manifold to be Ricci-flat which  can be seen by contracting the necessary condition $R_{kl}\varphi=0$ for the parallel spinor by  $\gamma^k$ and using the Bianchi identity  (see \cite{BFGK}). 
 That $\mathcal{K}_0$ is empty in this case follows from the integral formula (see \cite{KobNom})
$\int_M Ric(X,X)+{\rm tr}(\nabla X)^2-({\rm tr}\nabla X)^2=0$  with $Ric=0$. This formula forces the non trivial vector field $X$ to be parallel which splits one factor $S^1$. This is the contradiction to the simply connectness.
So we see that simply connected Riemannian examples with nontrivial $\mathcal{K}_0$ and $\mathcal{K}_1$  will always be non compact. 

Our definition of the supersymmetric Killing structure and the definition of the Killing fields are strongly motivated by the flat case so that this turns out to be our first example:

We consider the flat space $\RRR^D$ with metric of signature $\sigma$ such that supersymmetry is possible, see Theorems \ref{whencomplex} and  \ref{whenreal}. In flat space the curvature vanishes and so do all  third order terms which arise by calculating the commutator of two odd fields.
The parallel spinors in the flat case are the constant spinors and the Killing vector fields are represented by infinitesimal rotations and infinitesimal translations.
In this first example we recover the classical supersymmetry algebra from \eqref{classicalSUSY}. 

\section{SSKS on manifolds of special holonomy}

\subsection{K\"ahler manifolds}\label{SecSUn}

The manifolds we will investigate in this section are manifolds with Holonomy  $SU(n)$. By this we mean simply connected  irreducible
manifolds of dimension $D$ whose restricted holonomy group is  $SU(n)\subset {GL}_D\RRR$.
We recall that a Riemannian manifold which admits a parallel spinor is Ricci flat. 
Furthermore its holonomy group is one of the list in the next proposition which is due to \cite{Hitchin} and \cite{WangMK1}.
\begin{proposition}
Let $M$ be a complete, simply connected, irreducible, non flat Riemannian manifold.
If $M$ admits a parallel spinor, then its dimension $D$, its holonomy $G$, and the number $N$ of linearly
independent complex spinors is given by either
$(D=2n,G=SU(n),N=2)$,
$(D=4m,G=Sp(m),N=m+1)$,
$(D=8,Spin(7),N=1)$ or
$(D=7,G=G_2,N=1)$.
\end{proposition}

From \cite{GaoTian} or \cite{LawMich} the fact, that a manifold has holonomy $SU(n)$ if there exists exactly two parallel spinors, can reformulated as follows:  $M$ has holonomy  $SU(n)$  if there exists a parallel pure spinor and 
a spinor $\eta$ is called pure if the kernel of the map
\begin{equation}
TM \otimes \CCC \ni v\mapsto \rho(v)\eta=0
\end{equation}
has maximal complex dimension  $n$. We choose a local ON-basis $\{e_i\}$ of $TM$ and the associated complex basis $\{e_k:=\frac{1}{2}(e_k-ie_{k+n}),e_{\bar k}:=\frac{1}{2}(e_k+ie_{k+n})\}$ in such a way that the $\gamma$-matrices\footnote{For the
$\gamma$-matrices associated to the hermitian basis we have $\{\gamma^{\bar\imath},\gamma^j\}=-2g^{\bar\imath j}$, $\{\gamma^{\bar\imath},\gamma^{\bar\jmath}\}=\{\gamma^i,\gamma^j\}=0$ and $(\gamma^{\bar\jmath})^\dagger=-\gamma^i$ }
 obey
\begin{equation}
\gamma^{\bar\imath}\eta=0.\label{pure}
\end{equation}
Due to the symmetries of the charge conjugation we have 
$(\gamma^{\bar\imath}\eta)^C\sim \gamma^i\eta^C$ up to a sign so that  \eqref{pure} is equivalent to
\begin{equation}
\gamma^i\eta^C=0.\label{pure2}
\end{equation}
So we  always have pairs of pure spinors related by the charge conjugation, and because the charge conjugation is parallel, either both are parallel or none. We draw attention to the fact that in the case $D=0\mod4$ the two are of the same chirality whereas in the case $D=2\mod4$ they are of opposite chirality.

We recall the standard embedding $\mathfrak{su}(n)\subset\mathfrak{so}(2n)$ by
\begin{equation}
\mathfrak{su}(n)\ni \tilde R_{i\bar \jmath}\simeq \left(\begin{array}{cc}R_1&R_2\\-R_2&R_1\end{array}\right)=R_{ij}\in\mathfrak{so}(2n)\label{standardEmb}
\end{equation}
with $\tilde R=R_1+iR_2$, i.e.\ $R_{i\bar \jmath}=R_{ij}+iR_{i ,j+n}=\bar R_{\bar\imath j}$
\begin{remark}
We recall some identities used for calculations in complex coordinates.

The Bianchi identities $R^m{}_{[ijk]}=0$ and $\nabla_{[i}R_{jk]}=0$ are written as
\begin{equation}
R^m{}_{[ij]\bar n}=R^{\bar m}{}_{[\bar\imath\bar\jmath]n}=0\ \text{ and }\
\nabla_{[i}R_{j]\bar m}=\nabla_{[\bar\imath}R_{\bar\jmath]m}=0.
\end{equation}
Useful $\gamma$-matrix identities are
\begin{equation}\begin{aligned}
\big[\gamma^{i\bar\jmath},\gamma^n\big]&=-2g^{\bar\jmath n}\gamma^i\,,\qquad
\big[\gamma^{i\bar\jmath},\gamma^{\bar n}\big] =2g^{i\bar n}\gamma^{\bar\jmath}\,,\\
\big[\gamma^{o\bar p},\gamma^{i\bar\jmath}\big]&=2g^{\bar\jmath o}\gamma^{i\bar p}
-2g^{i\bar p}\gamma^{o\bar\jmath}.
\end{aligned}\end{equation}
\end{remark}
\begin{lemma}\label{purezero}
\begin{enumerate}[leftmargin=2.5em]
\item
For all (complex) spinors $\eta,\xi$ we have
\begin{equation}
\mathfrak{B}^{(2)}(R;\eta,\xi) 
	=2\big(\gamma^i\eta\wedge\gamma^{\bar \jmath}\xi\owedge R_{i\bar \jmath}
	+\gamma^{\bar \imath}\eta\wedge\gamma^{j}\xi \owedge R_{\bar\imath j }\big).
\end{equation}
\item A pure spinor $\eta$ satisfies
$ \mathfrak{B}^{(2)}(R;\eta,\eta)=0 $.
\end{enumerate}
\end{lemma}
\begin{proof}
With the definition of the complex $\gamma$-matrices we have
\begin{align*}
\begin{aligned}
\gamma^i\eta\wedge\gamma^{\bar \jmath}\xi\owedge R_{i\bar \jmath}  
&= \frac{1}{4}(\gamma^i+i\gamma^{i+n})\eta\wedge(\gamma^j-i\gamma^{j+n})\xi\owedge(R_{ij}+iR_{i,j+n})\\
&= \frac{1}{4}\gamma^i\eta\wedge\gamma^j\xi\owedge R_{ij}
	+\frac{i}{4}\big(\gamma^i\eta\wedge\gamma^j\xi\owedge R_{i,j+n}
	+\gamma^{i+n}\eta\wedge\gamma^{j+n}\xi\owedge R_{i,j+n}\\
&\quad  +\gamma^{i+n}\eta\wedge\gamma^j\xi\owedge R_{ij}
	-\gamma^i\eta\wedge\gamma^{j+n}\xi\owedge R_{ij}\big)
\end{aligned}
\end{align*}
and 
\begin{align*}
\gamma^{\bar \imath}\eta\wedge\gamma^{j}\xi\owedge R_{\bar\imath j }
&= 	\frac{1}{4}\gamma^i\eta\wedge\gamma^j\xi\owedge R_{ij} 
 -\frac{i}{4}\Big(\gamma^i\eta\wedge\gamma^j\xi\owedge R_{i,j+n} \\
&\quad	+\gamma^{i+n}\eta\wedge\gamma^{j+n}\xi\owedge R_{i,j+n} 
	+\gamma^{i+n}\eta\wedge\gamma^j\xi\owedge R_{ij}
	-\gamma^i\eta\wedge\gamma^{j+n}\xi\owedge R_{ij}\Big)
\end{align*}
which proves the first part. The second part  is a  consequence of  \eqref{pure} and \eqref{pure2}.
\end{proof}


\subsubsection{Real SSKS in the $SU(n)$ case}

We turn to the investigation of real supersymmetric Killing structures and distinguish three cases

1.\ {\it The real case:} $D\equiv 0 \mod8$

We are in the case where the charge conjugation is chirality preserving. All $\gamma$-matrices are chosen to be  symmetric (resp.\  skew-symmetric). So we have
\[
\gamma^{\bar\imath}\eta=\gamma^i\eta^C=0
\]
 for the pure spinors $\eta$ and $\eta^C=C\eta^*$. 
Let $\eta_1,\eta_2$ be the two parallel real spinors on $M$ of the same chirality given by
\begin{equation}
\eta_1=\frac{1}{2}(\eta+\eta^C)\quad\text{and }\ \eta_2=\frac{1}{2i}(\eta-\eta^C).
\end{equation}
In fact, the pure spinor can not be real,  i.e.\ $\eta\neq\eta^C$, otherwise its kernel would be the whole space.
With this notation the third order terms in the commutator of the generators of $\mathcal{K}_1={\rm span}_\RRR\{\eta_1,\eta_2\}$ are given by
\begin{equation}\begin{split}
\mathfrak{B}^{(2)}(R;\eta_1,\eta_1)& = \frac{1}{4}\mathfrak{B}^{(2)}(R;\eta,\eta)
   +\frac{1}{4}\mathfrak{B}^{(2)}(R;\eta^C,\eta^C)
   +\frac{1}{2}\mathfrak{B}^{(2)}(R;\eta,\eta^C)\\
\mathfrak{B}^{(2)}(R;\eta_2,\eta_2)& = -\frac{1}{4}\mathfrak{B}^{(2)}(R;\eta,\eta)
   -\frac{1}{4}\mathfrak{B}^{(2)}(R;\eta^C,\eta^C)
   +\frac{1}{2}\mathfrak{B}^{(2)}(R;\eta,\eta^C)\\
\mathfrak{B}^{(2)}(R;\eta_1,\eta_2)& = \frac{1}{4i}\mathfrak{B}^{(2)}(R;\eta,\eta)
   -\frac{1}{4i}\mathfrak{B}^{(2)}(R;\eta^C,\eta^C)
\end{split}\end{equation}
which with part (2) of Lemma \ref{purezero} leads to
\begin{equation}\begin{split}
\mathfrak{B}^{(2)}(R;\eta_1,\eta_2) &= 0\\
\mathfrak{B}^{(2)}(R;\eta_1,\eta_1) &= \mathfrak{B}^{(2)}(R;\eta_2,\eta_2) = \frac{1}{2}\mathfrak{B}^{(2)}(R;\eta,\eta^C).
\end{split}\end{equation}

2.\ {\it The complex case:} $D=2\mod 8$

The pure parallel spinor in $S^+$ is denoted by  $\eta_+=(\xi,0)$  and we know that its charge conjugated is given by $\eta_-=\tau(\eta_+)\in S^-$. The definition of the real spinor bundle as $S=\left\{ (\xi,\tau(\xi))\,|\  \xi\in S^+\right\}\subset S^+\oplus S^-$ gives two real parallel spinors
\begin{equation}
\eta_1=\frac{1}{2}(\eta_+ +\eta_-)\quad  \text{and}\  \eta_2=\frac{1}{2i}(\eta_+ -\eta_-),
\end{equation}
because $\tau(i\eta_+)=-i\tau(\eta_+)=-i\eta_-$. The same calculation as above shows that for these spinors we have
\begin{equation}\begin{split}
\mathfrak{B}^{(2)}(R;\eta_1,\eta_2)&= 0\\
\mathfrak{B}^{(2)}(R;\eta_1,\eta_1)&= \mathfrak{B}^{(2)}(R;\eta_2,\eta_2)
= \frac{1}{2}\mathfrak{B}^{(2)}(R;\eta_+,\eta_-).\label{auswertung}
\end{split}\end{equation}

3.\  {\it The quaternionic case:} $D=4\mod8$

The construction in this case differs from the one in the two cases before.  
We recall that the real spinor bundle is given by \[S=\{\Phi=(\varphi^1,\varphi^2) |\tau(\Phi)=\Phi\Leftrightarrow \varphi^I_C=\Omega^{IJ}\varphi^J\}.\]

The pure (complex) spinor is denoted by $\eta$ and its (linear independent) charge conjugated by $\eta^C$  like above. We define a real basis by
\begin{equation}
\Phi_1=\frac{1}{2}{\binom{\eta}{\eta^C}},       \quad
\Phi_2=\frac{1}{2}\binom{\eta^C}{-\eta},    \quad
\Phi_3=\frac{1}{2}\binom{i\eta}{-i\eta^C},    \quad
\Phi_4=\frac{1}{2}\binom{i\eta^C}{i \eta}.\label{realquatbase}
\end{equation}
Furthermore we define the following elements which turn out to be well adapted to our problem, see \eqref{adaptedbasis}:
\begin{equation}\begin{split}
\eta_1 &=\frac{1}{2}(\Phi_1-i\Phi_3)=\frac{1}{2}\binom{\eta}{0},\qquad 
\eta_{\bar 1}=\frac{1}{2}(\Phi_1+i\Phi_3)=\frac{1}{2}\binom{0}{\eta^C},\\
\eta_2&=\frac{1}{2}(\Phi_2-i\Phi_4)=\frac{1}{2}\binom{\eta^C}{0},\qquad
\eta_{\bar 2}=\frac{1}{2}(\Phi_2+i\Phi_4)=\frac{1}{2}\binom{0}{-\eta}.
\end{split}\label{complquatbase}\end{equation}
They do not form a real basis, but obey $\tau(\eta_\alpha)=\eta_{\bar\alpha}$.

\begin{lemma}
The evaluation of $\mathfrak{B}^{(2)}(R;\ldotp,\ldotp)$ on pairs of elements in $\{\eta_1,\eta_{\bar 1},\eta_2,\eta_{\bar 2}\}$ is given by:
\begin{equation}\begin{aligned}
\mathfrak{B}^{(2)}(R;\eta_1,\eta_1)&= \mathfrak{B}^{(2)}(R;\eta_{\bar 1},\eta_{\bar 1})
= \mathfrak{B}^{(2)}(R;\eta_2,\eta_2)
\\
&= \mathfrak{B}^{(2)}(R;\eta_{\bar 2},\eta_{\bar 2})
 =\mathfrak{B}^{(2)}(R;\eta_1,\eta_{\bar 2})
 = \mathfrak{B}^{(2)}(R;\eta_{\bar 1},\eta_2)=0 \\
\mathfrak{B}^{(2)}(R;\eta_1,\eta_{\bar 1})& =
  \frac{1}{2}\binom{\gamma^k\eta}{0}\wedge \binom{0}{\gamma^{\bar l}\eta^C }\owedge R_{k\bar l}
\\
\mathfrak{B}^{(2)}(R;\eta_2,\eta_{\bar 2})&=
  -\frac{1}{2}\binom{0}{\gamma^k\eta } \wedge\binom{\gamma^{\bar l}\eta^C}{0}\owedge R_{k\bar l }
\\
\mathfrak{B}^{(2)}(R;\eta_1,\eta_2)& =
  \frac{1}{2}\binom{\gamma^k\eta}{0}\wedge\binom{ \gamma^{\bar l}\eta^C}{0 }\owedge R_{k\bar l}
\\
\mathfrak{B}^{(2)}(R;\eta_{\bar 1},\eta_{\bar 2})& =
-\frac{1}{2}\binom{0}{\gamma^k\eta }\wedge\binom{0}{\gamma^{\bar l}\eta^C}\owedge R_{k \bar l }
\end{aligned}
\end{equation}
\end{lemma}
\begin{proof}
We will not prove all of the statements, but for example
\begin{align*}
\mathfrak{B}^{(2)}(R;\eta_1,\eta_1)&=
\binom{\gamma^k\eta}{0}\wedge\overbrace{\binom{\gamma^{\bar l}\eta}{0}}^{=0}\owedge R_{k\bar l}= 0\,,\\
\mathfrak{B}^{(2)}(R;\eta_1,\eta_{\bar 2})&=
\frac{1}{2}\binom{\gamma^k\eta}{0}\wedge\overbrace{\binom{0}{-\gamma^{\bar l}\eta}}^{=0}\owedge R_{k\bar l}
+\frac{1}{2}\overbrace{\binom{\gamma^{\bar k}\eta}{ 0}}^{=0}\wedge\binom{0}{-\gamma^l \eta}\owedge R_{\bar k l} 
=0\,,
\end{align*}
and one  non vanishing bracket
\begin{align*}
\mathfrak{B}^{(2)}(R;\eta_1,\eta_{\bar 1})& =
  \frac{1}{2}\binom{\gamma^k\eta}{ 0}\wedge\binom{0}{\gamma^{\bar l}\eta^C}\owedge R_{k\bar l}
  +\frac{1}{2}\overbrace{\binom{\gamma^{\bar l}\eta}{0}}^{=0}
              \wedge\overbrace{\binom{0}{\gamma^k\eta^C}}^{=0}\owedge R_{\bar l k}\\
 &=   \frac{1}{2}\binom{\gamma^k\eta}{0}\wedge\binom{0}{\gamma^{\bar l}\eta^C}\owedge R_{k\bar l}\;.
\end{align*}
\end{proof}

The previous lemma yields 
\begin{equation}\begin{aligned}
\mathfrak{B}^{(2)}(R;\Phi,\Psi)
&=(a\bar z+\bar az)\mathfrak{B}^{(2)}(R;\eta_1,\eta_{\bar1}) 
		+(\bar bw+b\bar w)\mathfrak{B}^{(2)}(R;\eta_2,\eta_{\bar2}) \\
&\quad  +(aw+bz)\mathfrak{B}^{(2)}(R;\eta_1,\eta_2) +(\bar a\bar w+\bar b\bar z)\mathfrak{B}^{(2)}(R;\eta_{\bar1},\eta_{\bar2})
\end{aligned}\end{equation}
for two real parallel spinors $\Phi=a\eta_1+\bar a\eta_{\bar 1}+b\eta_2+\bar b\eta_{\bar 2}$ and $\Psi =z\eta_1+\bar z\eta_{\bar 1}+w\eta_2+\bar w\eta_{\bar 2}$. 
For the sake of completeness we add the result for the real basis.

\begin{remark}\label{commoftherealbase}
\begin{align}
\mathfrak{B}^{(2)}(R;\Phi_1,\Phi_1)&=\mathfrak{B}^{(2)}(R;\Phi_3,\Phi_3)
  =\mathfrak{B}^{(2)}(R;\eta_1,\eta_{\bar1})\nonumber\\
\mathfrak{B}^{(2)}(R;\Phi_2,\Phi_2)&=\mathfrak{B}^{(2)}(R;\Phi_4,\Phi_4)
  =\mathfrak{B}^{(2)}(R;\eta_2,\eta_{\bar2})\nonumber\\
\mathfrak{B}^{(2)}(R;\Phi_1,\Phi_3)&=\mathfrak{B}^{(2)}(R;\Phi_2,\Phi_4)= 0  \nonumber\\
\mathfrak{B}^{(2)}(R;\Phi_1,\Phi_2)&=-\mathfrak{B}^{(2)}(R;\Phi_3,\Phi_4)
  =\mathfrak{B}^{(2)}(R;\eta_1,\eta_2)+\mathfrak{B}^{(2)}(R;\eta_{\bar1},\eta_{\bar2})
     \nonumber\\
\mathfrak{B}^{(2)}(R;\Phi_1,\Phi_4)&=\mathfrak{B}^{(2)}(R;\Phi_3,\Phi_2)
  =i\big(\mathfrak{B}^{(2)}(R;\eta_1,\eta_2)-\mathfrak{B}^{(2)}(R;\eta_{\bar1},\eta_{\bar2})\big)
     \nonumber
\end{align}
\end{remark}

\begin{definition}\label{defthirdorder}
Let $M$ be a simply connected, Ricci flat, non flat  K\"ahler manifold with real spinor bundle $S$. Its  two (complex) parallel pure spinors are given by $\eta$ and $\eta^C$. 
The third order even  fields which appear in the construction above will be denoted by
\begin{equation}
\mathcal{Z}:=\mathfrak{B}^{(2)}(R;\eta,\eta^C)\label{egal1}
\end{equation}
in the real and complex case, and by
\begin{equation}\begin{split}
\mathcal{Z}_1&:=\mathfrak{B}^{(2)}(R;\eta_1,\eta_{\bar1})\\
\mathcal{Z}_2&:=\mathfrak{B}^{(2)}(R;\eta_2,\eta_{\bar2})\\
\mathcal{Z}_3&:=\mathfrak{B}^{(2)}(R;\eta_1,\eta_2)
                                       +\mathfrak{B}^{(2)}(R;\eta_{\bar1},\eta_{\bar2})\\
\mathcal{Z}_4&:=i (\mathfrak{B}^{(2)}(R;\eta_1,\eta_2)
                                       -\mathfrak{B}^{(2)}(R;\eta_{\bar1},\eta_{\bar2}))
\end{split}\end{equation}
in the quaternionic case. Furthermore in both cases we define $\mathfrak{Z}$ to be the respective linear span, i.e.
\begin{equation}
 \mathfrak{Z}:=\RRR \mathcal{Z}\quad\text{or }\ 
 \mathfrak{Z}:={\rm span}_\RRR\{\mathcal{Z}_1,\ldots\mathcal,{Z}_4\} \label{egal2}
\end{equation}
and in the quaternionic case the two subspaces of $\mathfrak{Z}$ given by
\begin{equation}
 \mathfrak{Z}_1:= \RRR\mathcal{Z}_1,\quad\mathfrak{Z}_2:=\RRR\mathcal{Z}_2
\end{equation}
\end{definition}

\begin{theorem}\label{generalSUfall}
{\em (1)\ }\ Let $M$ be a simply connected, Ricci flat, non flat K\"ahler manifold of real  ($\dim M= 0\mod8$) or complex ($\dim M=2\mod8$) type, $S$ be the real spinor bundle and $\hat M$ be the associated special graded manifold. Denote  its parallel pure (complex) spinors by $\eta$ and $\eta^C$. Furthermore let
\begin{align}
\mathcal{K}_0&=\{\text{Killing vector fields}\}\subset\X(M) \nonumber\\
\mathcal{K}_1&=\{\text{parallel spinors}\}\subset\Gamma S\nonumber
\end{align}
and consider the maps 
\[
  \X(M)
\stackrel{\e}{\longrightarrow}
\X(\hat M)_0
\qquad\text{and}\qquad
 \Gamma S
\stackrel{\o}{\longrightarrow}
\X(\hat M)_1\;.
\]
The set $\mathfrak{Z}$ is defined in Definition \ref{defthirdorder}. Then
\begin{equation}
\e(\mathcal{K}_0)\oplus\mathfrak{Z}\oplus\o(\mathcal{K}_1)
\end{equation}
is a finite supersymmetric Killing structure on $\hat M$ with the following brackets
\begin{align*}
\left[\e(\mathcal{K}_0),\e(\mathcal{K}_0)\right]& \subset \e(\mathcal{K}_0)&
\left[\o(\mathcal{K}_1),\o(\mathcal{K}_1)\right]& \subset \e(\mathcal{K}_0)\oplus \mathfrak{Z} \\
\left[\e(\mathcal{K}_0),\mathfrak{Z}\right]& \subset \mathfrak{Z} &
\left[\e(\mathcal{K}_0),\o(\mathcal{K}_1)\right]& \subset \o(\mathcal{K}_1) \\
\left[\o(\mathcal{K}_1),\mathfrak{Z}\right]& = \left[\mathfrak{Z},\mathfrak{Z}\right] = 0&&
\end{align*}
{\em (2)\ }\ Let $M$ be as in part 1, but of quaternionic type ($\dim M=4\mod 8$). Define $\mathcal{K}_{1,k}\subset \Gamma S$ by
\begin{align*}
\mathcal{K}_{1,1}&={\rm span}_\RRR\{\Phi_1,\Phi_3\} \\
\mathcal{K}_{1,2}&={\rm span}_\RRR\{\Phi_2,\Phi_4\}
\end{align*}
and  consider $\mathfrak{Z}_1$ and $\mathfrak{Z}_2$ from Definition \ref{defthirdorder}. Then
\begin{equation}
\e(\mathcal{K}_0)\oplus\mathfrak{Z}_k \oplus\o(\mathcal{K}_{1,k})
\end{equation}
is a finite supersymmetric Killing structure on $\hat M$ for $k=1,2$ and the brackets are the same as in the previous case with the obvious replacements.
\end{theorem}

\begin{proof}
We consider first the real respective complex case and note that we only have to show the commutator relations with one entry being $\mathcal{Z}=\mathfrak{B}^{(2)}(R;\eta,\eta^C)$. For a Killing vector field $X$ we get from \eqref{thirdorderandeven}
\begin{equation}
\left[\e(X),\mathcal{Z}\right] =
\mathfrak{B}^{(2)}(R;\e(X)\eta,\eta^C)+\mathfrak{B}^{(2)}(R;\eta,\e(X)\eta^C)\;. \nonumber
\end{equation}
We know that  the fields $\e(X)\eta $ and $\e(X)\eta^C$ are parallel for a Killing vector field $X$ which is enough for the proof. But we have more informations:  we area able to write the resulting spinors as a linear combination of $\eta$ and $\eta^C $. In fact, for a Killing vector field $X$ we have the properties
\begin{equation}
\e(X)\eta=z\eta\,,\quad \e(X)\eta^C=\bar z\eta^C
\end{equation}
for $z\in\CCC$ because $\gamma^{\bar k}\e(X)\eta = [\nabla X,\gamma^{\bar k}]\eta=(\nabla_iX)_{\bar\jmath}[\gamma^{i\bar\jmath},\gamma^{\bar k}]\eta= 2g^{i\bar k}(\nabla_i X)_{\bar\jmath }\gamma^{\bar\jmath}\eta=0$. This yields the explicit form
\begin{equation}\begin{aligned}
\left[\e(X),\mathcal{Z}\right]&=\mathfrak{B}^{(2)}(R;\e(X)\eta,\eta^C)
+\mathfrak{B}^{(2)}(R;\eta,\e(X)\eta^C) \\
&=  (z+\bar z)\mathfrak{B}^{(2)}(R;\eta,\eta^C) 
\end{aligned}\end{equation}
The next step is the bracket with an element in $\o(\mathcal{K}_0)$. Because of the linearity, we can restrict to the images of $\eta_1$ and $\eta_2$.

Recalling \eqref{auswertung} and Corollary \ref{easy} yields
\begin{equation}
\big[\o(\eta_1),\mathcal{Z}\big]=\big[\o(\eta_1),\mathfrak{B}^{(2)}(R;\eta_1,\eta_1)\big]=0.
\end{equation}
The argument is the same for  $\eta_2$.

For the proof of the quaternionic case as stated we have nothing to change but the notation.
\end{proof}


\subsubsection{Complex  SSKS in the $SU(n)$ case}

In the case of complex special graded manifolds we do not have to take care about the doubling of the spinor bundle in dimension four modulo eight so that we end up at

\begin{theorem}\label{generalSUfallC}
Let $M$ be a simply connected, Ricci flat, non flat K\"ahler manifold  of dimension $ 0,2,$ or $4\mod8$, $S$ be the complex spinor bundle and $\hat M$ be the associated complex  special graded manifold. Denote  its parallel pure  spinors by $\eta$ and $\eta^C$. Furthermore let
\begin{align}
\mathcal{K}_0&=\{\text{complexified Killing vector fields}\}\subset\X_\CCC(M) \nonumber\\
\mathcal{K}_1&=\{\text{parallel spinors}\}\subset\Gamma S\nonumber
\end{align}
and consider the maps 
\[
  \X(M)
\stackrel{\e}{\longrightarrow}
\X(\hat M)_0
\qquad\text{and}\qquad
 \Gamma S
\stackrel{\o}{\longrightarrow}
\X(\hat M)_1
\]
The center  is defined  by   $\mathfrak{Z}=\CCC\cdot\mathcal{Z}$ with  \eqref{egal1}. Then 
\begin{equation}
\e(\mathcal{K}_0)\oplus\mathfrak{Z}\oplus\o(\mathcal{K}_1)
\end{equation}
is a finite supersymmetric Killing structure on $\hat M$, with the same brackets as in Theorem \ref{generalSUfall} part 1.
\end{theorem}

\begin{proof}
The poof is the same as for Theorem \ref{generalSUfall}.
\end{proof}

\subsubsection{K\"{a}hler 4-folds}

Taking into account the classification of \cite{WangMK1} we prove the following theorem.

\begin{theorem}
Let $M$ be a simply connected four dimensional Riemannian manifold with complex spinor bundle $S$.  Furthermore let $\mathcal{K}_0$ be the space of (complexified) Killing vector fields on $M$, $\mathcal{K}_1=\{\varphi\in\Gamma S;\ \nabla\varphi=0\}$ and  $\e,\o$ like before.
We have to distinguish three cases
\begin{itemize}[leftmargin=2em]
\item $\dim \mathcal{K}_1 =0$ and $(\mathcal{K},\mathfrak{e},\mathfrak{o})$ is an even supersymmetric Killing structure  (there are no odd generators).
\item $\dim\mathcal{K}_1>0$ and $\mathcal{K}_1\subset\Gamma S^{+}$ [or $\Gamma S^-$]. In this case    $M$ is a Ricci flat, non flat, half conformally flat K\"{a}hler manifold and $(\mathcal{K},\mathfrak{e},\mathfrak{o})$ gives rise to  a supersymmetric  Killing structure of odd dimension two where the central part of the complex SSKS  from  Theorem  \ref{generalSUfallC} vanishes whereas the central parts of the two real SSKS from Theorem \ref{generalSUfall}.2 do not. In the second case the central parts coincide. 

\item $\dim\mathcal{K}_1>0$ and $\mathcal{K}_1\cap \Gamma S^{\pm}\neq\emptyset$. Then M is flat and  $\mathcal{K}_1=S$.
\end{itemize}
\end{theorem}
\begin{proof}
The first statement is clear. 
Suppose that $M$ has one parallel spinor of negative chirality, say. The curvature of a four dimensional Riemannian manifold decomposes in the Weyl curvature and the (in our case vanishing) Ricci- and scalar curvature part. From the fact that the Ricci curvature and the scalar curvature vanish we are left with the following relation between the Weyl curvature $W$ of a Riemannian manifold and a parallel spinor $\phi$:
\[
W_{kl}\phi=0.
\]
An algebraic computation shows  that  a  parallel spinor $\phi\in\Gamma S^\pm$ leads to $W^\pm=0$, see \cite[Thm 1.13]{BFGK}. 

The assumption that $M$ is non-flat implies that all parallel spinors on $M$ are of the same chirality, e.g.\ of negative chirality, so that the curvature is anti self dual. Since the curvature restricted to the bundle $S^-$ is zero, and $M$ is simply connected we get that  $S^-$ is trivial. 
Furthermore we know from the classification of manifolds with parallel spinors that in the non-flat case the existence of a parallel spinor yields $\dim \mathcal{K}_1=2$. This is the first part of the second statement.

For $M$ admitting one parallel spinor of each chirality, the above yields $W=0$. So the manifold is flat which is the third statement.

We turn back to the second statement and assume the manifold to be anti self dual, i.e.  the curvature obeys
\[
R_{kl}=R_{kl}^-=\frac{1}{2}(R_{kl}-\frac{1}{2}\varepsilon_{klij}R^{ij})\;.
\]
We show that the center vanishes by using an explicit realization of the gamma matrices with respect to which the parallel spinors are given by $e_2$ and $e_4$, see Appendix \ref{SecDim4}. 
\begin{equation}\begin{aligned}
\mathfrak{B}^{(2)}(R;\phi,\psi) & = \gamma^k\phi\wedge\gamma^l\psi\,R_{kl} \\
& = \left(\gamma^1\phi\wedge\gamma^2\psi-\gamma^2\phi\wedge\gamma^1\psi-\gamma^3\phi\wedge\gamma^4\psi+\gamma^4\phi\wedge\gamma^3\psi\right)R_{12}\\
&\quad +\left(\gamma^1\phi\wedge\gamma^3\psi-\gamma^3\phi\wedge\gamma^1\psi+\gamma^2\phi\wedge\gamma^4\psi-\gamma^4\phi\wedge\gamma^2\psi\right)R_{13}\\
&\quad +\left(\gamma^1\phi\wedge\gamma^4\psi-\gamma^4\phi\wedge\gamma^1\psi-\gamma^2\phi\wedge\gamma^3\psi+\gamma^3\phi\wedge\gamma^2\psi\right)R_{14},
\end{aligned}\end{equation}
which reduces for $\psi=\phi$ to
\begin{equation}\begin{aligned}
\mathfrak{B}^{(2)}(R;\phi ,\psi) &= \gamma^k\phi\wedge\gamma^l\psi R_{kl}\\
&= 2\left(\gamma^1\phi\wedge\gamma^2\phi-\gamma^3\phi\wedge\gamma^4\phi\right)R_{12}
+2\left(\gamma^1\phi\wedge\gamma^3\phi+\gamma^2\phi\wedge\gamma^4\phi\right)R_{13}\\
&\quad +2\left(\gamma^1\phi\wedge\gamma^4\phi-\gamma^2\phi\wedge\gamma^3\phi\right)R_{14}.
\end{aligned}\end{equation}
This expression vanishes for all combinations of parallel spinors. This is easily computed  in our choice of gamma matrices.

The anti-chiral constant spinors $e_2$ and $e_4$ give a basis for $ S_\CCC$. This choice of basis also fulfills
\begin{equation}
e_2^C= -e_4,\quad e_4^C=e_2
\end{equation}
and $e_2$ is pure\footnote{
In our explicit description the kernel is given by $\gamma^1-i\gamma^3$ and $\gamma^2+i\gamma^4$}. 
Consequently the terms with equal entries vanish by Lemma \ref{purezero} and we are left with the mixed term.

For the real SSKS we choose the basis like in \eqref{realquatbase}. Then  the term with different entries vanishes because of Corollary \ref{commoftherealbase}. In this case we are left with the two terms with entries $(e_2,0),(0,-e_4)$ and $(e_4,0),(0,e_2)$ for the two SSKS given by $\mathcal{K}_{1,k}$ respective. 

First we consider the complex SSKS and calculate $\mathfrak{B}^{(2)}(R;e_2,e_4)$.
\begin{align}
&\big( \gamma^1e_2\wedge\gamma^2e_4-\gamma^2e_2\wedge\gamma^1e_4- \gamma^3e_2\wedge \gamma^4e_4+ \gamma^4e_2\wedge\gamma^3e_4\big) R_{12}\nonumber\\
&\ +\big(\gamma^1e_2\wedge\gamma^3e_4-\gamma^3e_2\wedge\gamma^1e_4+ \gamma^2e_2\wedge \gamma^4e_4- \gamma^4e_2\wedge\gamma^2e_4\big)R_{13}\nonumber\\
&\ +\big(\gamma^1e_2\wedge \gamma^4e_4-\gamma^4e_2\wedge \gamma^1e_4- \gamma^2e_2\wedge \gamma^3e_4+ \gamma^3e_2\wedge \gamma^2e_4 \big)R_{14}\nonumber\\
&=\big(-e_1\wedge(-e_1)- e_3\wedge(-e_3)- ie_1\wedge ie_1+ ie_3\wedge(-ie_3)\big) R_{12}\nonumber\\
&\quad+\big(-e_1\wedge(-ie_3)-ie_1\wedge(-e_3)+ e_3\wedge ie_1- ie_3\wedge(-e_1)\big)R_{13}\nonumber\\
&\quad+\big(-e_1\wedge ie_1-ie_3\wedge (-e_3)- e_3\wedge (-ie_3)+ ie_1\wedge (-e_1)\big)R_{14}\\
&= 0
\end{align}

Next we look at the real SSKS and the only two third order terms we have to calculate are $\mathfrak{B}^{(2)}(R;(e_2,0),(0,-e_4))$ and $\mathfrak{B}^{(2)}(R;(e_4,0),(0,e_2))$. The first one is 
\begin{align*}
&\textstyle 
	\Big(\gamma^1\binom{e_2}{0}\wedge\gamma^2\binom{0}{ -e_4}
	-\gamma^2\binom{e_2}{0}\wedge\gamma^1\binom{0}{-e_4}
	- \gamma^3\binom{e_2}{0}\wedge\gamma^4\binom{0}{-e_4}\\
&\ \textstyle
	+ \gamma^4\binom{e_2}{0}\wedge\gamma^3\binom{0}{-e_4}\Big) R_{12}
	+\Big(\gamma^1\binom{e_2}{0} \wedge\gamma^3\binom{0}{-e_4}
	- \gamma^3\binom{e_2}{0}\wedge\gamma^1\binom{0}{ -e_4}\\
&\ \textstyle
	+\gamma^2\binom{e_2}{0}\wedge \gamma^4\binom{e_4}{ 0}
	- \gamma^4\binom{e_2}{0}\wedge\gamma^2\binom{0}{-e_4}\Big)R_{13}
	+\Big(\gamma^1\binom{e_2}{0}\wedge \gamma^4\binom{0}{-e_4}\\
&\ \textstyle
	- \gamma^4\binom{e_2}{0}\wedge \gamma^1\binom{0}{ -e_4}
	-\gamma^2\binom{e_2}{ 0}\wedge \gamma^3\binom{0}{-e_4}
	+\gamma^3\binom{e_2}{ 0}\wedge \gamma^2\binom{0}{-e_4} \Big)R_{14}\displaybreak[0]\\
=&\textstyle
	\Big( -\binom{e_1}{0}\wedge\binom{0}{e_1} 
	- \binom{e_3}{0}\wedge\binom{0}{ e_3}
	+\binom{ie_1}{0}\wedge\binom {0}{ie_1}
	+\binom{ ie_3}{0}\wedge\binom{0}{ie_3}
	\Big) R_{12}\\
&\ \textstyle
	+\Big(-\binom{e_1}{0}\wedge\binom{0}{ ie_3}
	-\binom{ie_1}{0}\wedge\binom{0}{e_3}
	- \binom{e_3}{0}\wedge\binom {0}{ ie_1}
	-\binom {ie_3}{0}\wedge\binom{0}{e_1}\Big)R_{13}\\
&\ \textstyle
	+\Big(\binom{e_1}{0}\wedge \binom{0}{ ie_1}
	-\binom{ie_3}{ 0}\wedge\binom {0}{e_3}
	- \binom{e_3}{ 0}\wedge \binom{0}{ ie_3}
	+ \binom{ie_1}{0}\wedge \binom{0}{ e_1}\Big)R_{14}\displaybreak[0]\\
=& \textstyle
	-2\Big( \big(\binom{e_1}{ 0}\wedge\binom{0}{ e_1}
 	+\binom{e_3}{0}\wedge\binom{0}{e_3}\big) R_{12}
    -i\big(\binom{e_1}{0}\wedge\binom{0}{ e_3} 
	+\binom{e_3}{0}\wedge\binom{0}{e_1}\big)R_{13}\\
&\ \textstyle  
	+i\big(\binom{e_1}{0}\wedge \binom{0}{ e_1}
	-\binom{e_3}{0}\wedge \binom{0}{e_3}\big)R_{14}\Big)
\end{align*}
and a similar calculation shows that the second term gives the same value.
\end{proof}

\begin{remark}
The last proposition reflects a special property of dimension four.  
In this dimension we are able to choose the ON-basis of the base manifold  in such a way that $\gamma^{\bar k}\eta=\gamma^k\eta^C=0$. 
But moreover we are able to choose $\gamma^i\eta=\epsilon_{i\bar \jmath}\gamma^{\bar \jmath}\eta^C$ which induces the symmetry between the centers of the real supersymmetric Killing structures.
\end{remark}


\subsection{$Spin(7)$ and $G_2$ manifolds}\label{SecSpin7}

Recalling once more the classification result by Wang, an eight dimensional Riemannian manifold of holonomy ${Spin}(7)$ has exactly one parallel spinor. This yields that the SSKS for ${Spin}(7)$ manifolds has an one dimensional center if we define $\mathcal{K},\e$ and $\o$ as before. 

Although a seven dimensional Riemannian manifold of holonomy $G_2$ has exactly one parallel spinor as well,  we can not state such a result  because of Remark \ref{whenSGM}. Nevertheless, we will turn back to $G_2$ when we discuss Brinkmann spaces next.


\subsection{Brinkmann spaces with finite SSKS}

\subsubsection{PP-waves}

We will start this section on Brinkmann spaces by investigating pp-waves. They form a special class of  Brinkmann spaces. 
The importance of pp-waves  in physics, and there in particular in the analysis of supergravity models, grew very fast in the last two years. 
This is due to the fact  that these  manifolds have a  large number of Killing spinors. This is the reason why they yield so much profit: they preserve much of the supersymmetries of the gravity model (see e.g.\ \cite{BlauFigHullPapa1}, \cite{BlauFigHullPapa2}, \cite{BenaRoiban}, \cite{Michelson}, or  \cite{FigPapado1}). 

To fix notation we assume the manifold to be of dimension $D=n+2$. A Brinkmann space is a Lorentzian manifold with a light like parallel vector field. This forces the holonomy to be contained in $\mathfrak{so}(n)\ltimes\RRR^n$, because the holonomy representation has to possess a singlet which excludes the types I and III  in the classification of \cite{BergIkem11} (the part of the holonomy denoted by $\mathscr{A}$  has to vanish).

In contrast to the Riemannian case the existence of the parallel light-like vector field does not lead to a decomposition of the space, although the holonomy is reduced (see for example the discussion in \cite{BaumKath1} and references therein). 

The pp-waves are those Brinkmann spaces with parallel spinors and nontrivial abelian holonomy group. 
The existence of a parallel spinor reduces the holonomy to a subalgebra $\mathfrak{g}\ltimes\RRR^n$, where $\mathfrak{g}\subset\mathfrak{spin}(n)$ annihilates  at least one spinor.  
In \cite{Fig3} it is shown that  in low dimensions ($D\leq 5$) every Brinkmann space which admits a parallel spinor is a pp-wave, i.e $\mathfrak{g}=0$.
In \cite{Bryant1} several pseudo Riemannian metrics in dimension up to eleven are constructed. 
Of particular interest therein is the eleven dimensional Brinkmann space which admits exactly one parallel null spinor. 
The  space in this case has holonomy $(\mathfrak{spin}(7)\ltimes\RRR^8)\times\RRR$.

A $D$-dimensional pp-wave is locally given by a metric of the form 
\begin{equation}
g= 2 dx^-dx^+ +\sum_{i,k=1}^n h_{ik}(x^-)x^k dx^idx^- +H(x^-,x^i)(dx^-)^2+\sum_{i=1}^{n}(dx^i)^2 
\label{condition}\end{equation}
With respect to the local basis 
$e^+=dx^++\frac{1}{2}Hdx^- + \frac{1}{2}\sum_{i=1}^n h_idx^i,
e^-=dx^-,
e^k=dx^k$
 the metric  is written  $g=2e^-e^+ +\sum_{i=1}^n e^ie^i$ and the nonzero
 connection coefficients are 
\begin{equation}
\omega_{-j}=\frac{1}{2}(\partial_j H-\partial_-h_j) dx^-  -\frac{1}{2}\partial_{[i}h_{j]}dx^i \,,\quad
\omega_{ij}=-\frac{1}{2} \partial_{[i}h_{j]}dx^- .
\end{equation}
The surviving curvature and Ricci terms are 
\begin{align}
R_{-i}&=\Big( \frac{1}{4}\sum_{j=1}^n \partial_{[i}h_{j]}\partial_{[j}h_{k]} -\frac{1}{2}\partial_-\partial_{\{i}h_{k\}}  
+\frac{1}{2}\partial_i\partial_k H \Big)dx^k\wedge dx^-
\label{remaining}\\
Ric_{--}&=\frac{1}{2}\Delta H-\frac{1}{2} \partial_-{\rm div}\,h +\frac{1}{4}{\rm tr}\big[({\rm grad}\, h)^2\big]
                           -\frac{1}{4}{\rm tr}\big[({\rm grad}\, h)^T({\rm grad}\,h)\big]
\end{align} 
with ${\rm grad}\, h=\big(\partial_jh_k\big)_{jk}$.
The holonomy is explicitly given by  ${\rm span} \{\gamma^k\gamma^-\}_{1\leq k\leq n}$. The kernel of the holonomy algebra as a subalgebra of $\mathfrak{spin}(1,D-1)$ acting on the spinors is of dimension $d=2^{[\frac{D}{2}]-1}$ such that a pp-wave has exactly this amount of parallel spinors. This is seen as follows. First we note that the kernel of the holonomy is equal to the kernel of $\gamma^-$. Then we have  $(\gamma^-)^2=(\gamma^+)^2=0$, i.e.\ the minimal polyomial of the two matrices is $x^2$ which forces the kernels to be of dimension at least $d$. Furthermore we have  $\frac{1}{2}(\gamma^++\gamma^-)^2=\frac{1}{2}(\gamma^+\gamma^-+\gamma^-\gamma^+)=-1$ which gives dimension zero for the  intersection of the kernels, so that their dimension is at most  $d$.

For the special class of metrics with  $h=0$ the Killing vector fields of this metric are given by  
\begin{align*}
X_+		&=\partial_+\,,&
X_-		&=\partial_- \\
Y		&=f(x^-)\partial_k+g(x^-)x_k\partial_+\,,&
Z_{kl}	&= x^k\partial_l-x^l\partial_k
\end{align*}
with $\nabla X_+=0, (\nabla X_-)_{AB}=\delta_{k[A}\delta_{B]-}\partial^kH$ 
and $(\nabla Y)_{AB}=2\delta_{-[A}\delta_{B]k}f' $ if  $f'=-g, f\partial_kH=2x_kf''$ and 
$(\nabla Z_{kl})_{AB}=2\delta_{k[A}\delta_{B]l}$ if $x_k\partial_lH-x_l\partial_kH=0 $.

A further specialization to $H=\sum_{i=1}^n A_{ij}x^ix^j$ with $A_{ij}= -\delta_{ij}\lambda_i^2$ yields the celebrated solution 
\begin{align*}
Y_k &= \sin(\lambda_kx^-)\partial_k-\lambda_kx_k\cos(\lambda_kx^-)\partial_+,\\
Y_{k^*}&=\cos(\lambda_kx^-)\partial_k+\lambda_kx_k\sin(\lambda_kx^-)\partial_+
\end{align*}
with non vanishing commutators
\begin{align}
[Y_k,Y_{l^*}]&=2\delta_{kl}\lambda_kX_+,\quad
[X_-,Y_k]=\lambda_kY_{k^*},\quad [X_-,Y_{k^*}]=-\lambda_kY_{k}\,.
\nonumber
\end{align}
The rotational Killing vector fields are related to those rotations leaving the matrix $A_{ij}$ invariant. This invariance is only possible nontrivially in the diagonal case and  if some of the $\lambda_i$ coincide. For $A={\rm diag}\big(\lambda_1{\bf 1}_{n_1},\ldots,\lambda_r{\bf 1}_{n_r}\big)$ the full rotation algebra breaks down to $\mathfrak{so}(n_1)\oplus\cdots\oplus\mathfrak{so}(n_r)$. 

A consequence of the above discussion is the next proposition
\begin{proposition}
For two parallel spinors $\eta$ and $\xi$ on the pp-wave $M$ the commutator $[\o(\eta),\o(\xi)]$ is at most of order 2.
\end{proposition}
\begin{proof}
We have to show that $\mathfrak{B}^{(2)}(R;\eta,\xi)$ vanishes. Therefore, we write this term in the above coordinates and get
\begin{equation}\begin{aligned}
\mathfrak{B}^{(2)}(R;\eta,\xi) &=\gamma^A\eta\wedge\gamma^B\xi\owedge R_{AB}
=  \gamma^k\eta\wedge\underbrace{\gamma^-\xi}_{=0}\owedge R_{k-}
           +\underbrace{\gamma^-\eta}_{=0}\wedge\gamma^k\xi\owedge R_{-k}
 =0\,.
\end{aligned}\end{equation}
\end{proof}
This gives the following statement for the supersymmetric Killing structure.
\begin{corollary}\label{ppSSKS}
 The  supersymmetric Killing structure on a pp-wave $M$ is given by 
\begin{equation}
\e(\mathcal{K}_0)\oplus\o(\mathcal{K}_1)
\end{equation}
where $\mathcal{K}_0$ is the space of complexified Killing vector fields and $\mathcal{K}_1$  the space of parallel spinors on $M$.
\end{corollary}
\begin{remark}
Although pp-waves have a large number of odd Killing fields, Corollary \ref{ppSSKS} shows that the SSKS in this case is the same as in the flat case, i.e.\ the usual supersymmetry algebra. 
\end{remark}

\subsubsection{General Brinkmann spaces}

If we turn back to more general Brinkmann spaces which admit a parallel spinor we are lead to holonomy $\mathfrak{g}\ltimes\RRR^n\subset\mathfrak{spin}(n)\times\RRR^n\subset\mathfrak{spin}(1,n+1)$. The number of spinors which are mapped to zero by this algebra is half the number of the spinors mapped to zero by $\mathfrak{g}\subset\mathfrak{spin}(1,n+1)$. 
Obviously, we have $\ker(\mathfrak{g}\ltimes\RRR^n)=\ker(\mathfrak{g})\cap\ker(\RRR^n)$ and we show by an explicit discription of the $\gamma$-matrices that only one half of the kernel of $\mathfrak{g}\subset{\rm span}\{\gamma^{ij}\}$ is preserved by $\RRR^n={\rm span}\{\gamma^i\gamma^-\}$.

Consider the $\mathfrak{so}(n)$ $\gamma$-matrices $\sigma^k$ with $\{\sigma^k,\sigma^l\}=-2g^{kl}$ for $ 1\leq k\leq n$, from which we construct the $\mathfrak{so}(1,D-1)$ $\gamma$-matrices $\gamma^A$ for $0\leq A\leq n+1=D-1$.
\begin{equation}\begin{aligned}
\gamma^j 		&=\begin{pmatrix} &-i\sigma^j \\i\sigma^j& 0 \end{pmatrix} \text{ for }1\leq j\leq n,\qquad
\gamma^{n+1} 	 =\begin{pmatrix}  &i{\bf 1}\\i{\bf 1}&\end{pmatrix} \\
\gamma^0 		&=
  \begin{cases}
   \begin{pmatrix} &-i\hat \sigma \\i\hat\sigma& \end{pmatrix},
               \ \hat\sigma=(-)^{\frac{n}{4}}\sigma^1\cdots\sigma^n &\text{for $n$ even}
    \\
    \begin{pmatrix}  -{\bf 1}&\\&{\bf 1} \end{pmatrix}				&\text{for $n$ odd} 
  \end{cases}
\end{aligned}\end{equation}
Furthermore we define the light cone matrices  
\begin{equation}\begin{aligned}
\gamma^+&=\frac{1}{\sqrt{2}}(\gamma^{n+1}+\gamma^0)=
  \begin{cases}
     \frac{i}{\sqrt{2}}\begin{pmatrix} &{\bf 1}-\hat\sigma \\ {\bf 1} +\hat\sigma& \end{pmatrix} & \text{for $n$ even}
     \\
     \frac{1}{\sqrt{2}}\begin{pmatrix}-{\bf 1}&i{\bf 1}\\i{\bf 1}&{\bf 1} \end{pmatrix} &\text{for $n$ odd} 
 \end{cases} \\
\gamma^-& =\frac{1}{\sqrt{2}}(\gamma^{n+1}-\gamma^0)=
  \begin{cases}
     \frac{i}{\sqrt{2}}\begin{pmatrix} &{\bf 1}+\hat\sigma \\ {\bf 1} -\hat\sigma& \end{pmatrix} & \text{for $n$ even}
     \\
     \frac{1}{\sqrt{2}}\begin{pmatrix}  {\bf 1}&i{\bf 1}\\i{\bf 1}&-{\bf 1} \end{pmatrix} &\text{for $n$ odd} 
  \end{cases}
\end{aligned}\end{equation}
and in particular we have
\begin{equation}\begin{aligned}
\gamma^{ij} &= \begin{pmatrix} \sigma^{ij}&\\&\sigma^{ij}\end{pmatrix} \\
\gamma^{j-}	&=
\begin{cases}
	\frac{1}{\sqrt{2}}\begin{pmatrix} \sigma^j({\bf 1}-\hat\sigma)&\\&-\sigma^j({\bf 1}+\hat\sigma)\end{pmatrix} 
		&\text{for $n$ even}\\
	\frac{1}{\sqrt{2}}\begin{pmatrix} \sigma^j& i\sigma^j\\i\sigma^j&\sigma^j\end{pmatrix} 
		&\text{for $n$ odd}\\
\end{cases}
\end{aligned}\end{equation}
The $D$-spinors are constructed from the $n$-spinors by 
\begin{equation}
S_D=S_n \otimes S_2 =S_n \oplus S_n \stackrel{\text{$n$ even}}{=} (S_n^+\oplus S_n^-)\oplus(S_n^+\oplus S_n^-),
\label{double}
\end{equation}
(see e.g.\ \cite{Klinker1}). With this decomposition and recalling that ${\bf 1}\pm \hat\sigma$ is two times the projection on $S_n^\pm\subset S_n$ we get the following kernels for $\gamma^\pm$:
\[
\begin{array}{lr}
\displaystyle{\rm ker}\,\gamma^- =&\displaystyle  \left( S_n^+\oplus \{0\}\right)\oplus\left(\{0\}\oplus S_n^-\right)\\
\displaystyle{\rm ker}\,\gamma^+=&\displaystyle  \left(\{0\}\oplus S_n^-\right)\oplus\left(S_n^+\oplus \{0\}\right)
\end{array}
\qquad\text{for $n$ even}
\]
and
\[
\begin{array}{lr}
\displaystyle{\rm ker}\,\gamma^-&  = \displaystyle\left\{(\phi,-i\phi)\in S_n \oplus S_n\right\}\\
\displaystyle{\rm ker}\,\gamma^+& = \displaystyle \left\{(\phi,i\phi)\in S_n\oplus S_n\right\}
\end{array}
\qquad \text{for $n$ odd}
\]
This explicit discription of the kernel of $\gamma^-$ shows that only one half of the kernel of $\mathfrak{g}\subset \mathfrak{spin}(1,n+1)$ is annihilated. Because of \eqref{double} the latter is two times the kernel of $\mathfrak{g}\subset\mathfrak{spin}(n)$ and the null space of the holonomy algebra is isomorphic to one of them. 

Recalling the supersymmetric Killing structure for  manifolds with holonomy $SU(n)$ or  $Spin(7)$ from Sections \ref{SecSUn} and \ref{SecSpin7} and for pp waves from Corollary \ref{ppSSKS} we get the following result for more general Brinkmann spaces.
\begin{theorem}
For all $D$-dimensional Brinkmann spaces with holonomy from the list in Table \ref{brinkmanntable}
 the finite supersymmetric Killing structure is given by 
\[
\mathcal{K}_0\oplus\mathfrak{Z}\oplus \mathcal{K}_1.
\]
The central part is determined by the part $R_\mathfrak{g}$ of the curvature coming from $\mathfrak{g}\subset \mathfrak{spin}(n)\subset\mathfrak{spin}(1,D-1)$ to
\[
\mathfrak{B}^{(2)}(R_{\mathfrak{g}};\eta,\xi).
\]
\end{theorem}

\begin{table}[htb]\caption{Brinkmann spaces with finite SSKS}\label{brinkmanntable}
\renewcommand{\arraystretch}{1.9}
\[
\begin{array}{c|c|c}
D       & \mathfrak{hol} & \text{Type} \\\hline\hline
2m+1 & \mathfrak{su}(m)\ltimes\RRR^n & \rm II \\\hline
10      & \mathfrak{spin}(7)\ltimes\RRR^8 & \rm II \\\hline
9   & \mathfrak{g}_2\ltimes\RRR^7  & \rm II \\\hline
n+2   &  \RRR^n & \rm II \\\hline
11 & (\mathfrak{spin}(7)\ltimes\RRR^8)\times\RRR &  \rm IV \\\hline
10 & (\mathfrak{g}_2\ltimes\RRR^7)\times\RRR & \rm IV\\
\end{array}
\]
\end{table}


\section{The central part of the SSKS}\label{secCenter}

Up to now the supersymmetric Killing structures we discussed were finite --- in particular they were spanned by one element. As we saw in Section \ref{Secthirdorder} we can not expect finiteness in general. 
Therefore, we will  come back to the investigation of the central part $\mathfrak{Z}$ of the supersymmetric Killing structure in this section and we will show that  although the center is not finite it has a uniform shape. More precisely, it turns out that there are at most two types of elements occurring in the central part. These two types of elements depend, just as demanded, multilinearly on the spinorial entries and, furthermore, they depend on the curvature. We stress that there is no need to distinguish between real and complex SSKS with respect to the general structure of the center. Therefore we will not explicitly  mention if the vector fields on the reduced manifold have to be taken complexified.

Recalling the results of Section \ref{Secthirdorder} will help us to describe the central part $\mathfrak{Z}$ of the supersymmetric Killing structure in general. In the section mentioned above we calculated the commutators between  even and odd Killing fields and a third order term as well as the commutator between two third order terms. The results were 
\begin{align*}
\left[\e(X),\mathfrak{B}^{(2)}(R;\varphi,\psi)\right]
&= \mathfrak{B}^{(2)}(R;\e(X)\varphi,\psi))+\mathfrak{B}^{(2)}(R;\e(X)\psi,\varphi))
\tag{\ref{thirdorderandeven}}\\
\left[\o(\varphi),\mathfrak{B}^{(2)}(R;\eta,\xi)\right]
&= \mathfrak{B}^{(3)}(\nabla R;\varphi,\eta,\xi)
+R_{klj}{}^m\, \gamma^k\eta\wedge\gamma^l\xi\wedge\gamma^j\varphi\otimes\nabla_m 
\tag{\ref{thirdorderandodd}}\\
\left[\mathfrak{B}^{(2)}(R;\varphi,\psi),\mathfrak{B}^{(2)}(R;\xi,\eta)\right]
&= \mathfrak{B}^{(4)}(\widetilde R;\varphi,\psi,\xi,\eta)
\tag{\ref{26}}
\end{align*}
with $\tilde R$ defined by 
\begin{align*}
\widetilde{R}(X,Y,U,V)  =\ & \left[R(X,Y),R(U,V)\right] -  R(R(X,Y)U,V)  -R(U,R(X,Y)V)\\
& + R(R(U,V)X,Y)+R(X,R(U,V)Y). \tag{\ref{27}}
\end{align*}
and denoted by  $R \bullet R$, see Definition \ref{operations} below.

At this stage we see that at least  two different types of elements occur in the center $\mathfrak{Z}$. One is of the form  
\begin{multline}
\mathfrak{B}^{(k)}(A;\xi_1,\ldots,\xi_k)
=\ \gamma^{i_1}\xi_1\wedge\ldots\wedge\gamma^{i_k}\xi_k\owedge  A_{i_1\cdots i_k}
\in \Lambda^{k+1}\Gamma S\otimes \Gamma S\subset \X(\hat M)_{k\,{\rm mod}\,2} \label{Be}
\end{multline}
where  $A$ is section in $\otimes^k TM\otimes {\rm End}(TM)$ which, when evaluated on $k$ vectors, takes it values in the holonomy algebra of $M$, see Definition \ref{E-endomorphism}.

The second type of element, we will be confronted with, is of the form
\begin{multline}
\mathfrak{D}^{(k+1)}(A;\xi_1,\ldots,\xi_{k+1})
=A_{i_1\cdots i_{k+1} }{}^j \gamma^{i_1}\xi_1\wedge\ldots\wedge\gamma^{i_{k+1}}\xi_{k+1}\otimes\nabla_j\\
\in\Lambda^{k+1}S\otimes \X(M)\subset \X(\hat M)_{k+1\,{\rm mod}\,2} \label{De}
\end{multline}
where $A$ is of the same form as before. 

\begin{remark}
The second summand in \eqref{thirdorderandodd} is of this type: $\mathfrak{D}^{(3)}(R;\varphi,\eta,\xi)$
\end{remark}

We summarize the result for general supersymmetric Killing structures in the next proposition.
In its formulation we need various operations on and between endomorphism valued tensors. We will use intuitive notations that we list below.
\begin{definition}\label{operations}
\begin{align*}
\text{(a)}&&& \displaystyle   (\nabla A)_{i_0i_1\cdots i_k} = (\nabla_{i_0}A)_{i_1\cdots i_k} \\
\text{(b)} &&& \displaystyle(A\rfloor^{1} R)_{mi_1\cdots i_k}
   =  A_{i_1\cdots i_k}{}^nR_{nm}\\
\text{(c)} &&& \displaystyle(A\bullet B)_{i_1\cdots i_kj_1\cdots j_l}
  =  \big[A_{i_1\cdots i_k},B_{j_1\cdots j_l}\big] 
            - \sum_{\beta=1}^l A_{i_1\cdots i_kj_\beta}{}^m B_{j_1\cdots m\cdots  j_l} 
            +\sum_{\alpha=1}^kB_{j_1\cdots j_li_\alpha}{}^m A_{i_1\cdots m\cdots i_k}\\
\text{(d)}  &&&\displaystyle(A\rfloor^1B\rfloor^2 R)_{i_1\cdots i_kj_1\cdots j_l} 
  = A_{i_1\cdots i_k}{}^mB_{j_1\cdots j_l}{}^n R_{mn}\\
\text{(e)} &&&\displaystyle(\nabla_AB)_{i_1\cdots i_kj_1\cdots j_l}
 =  A_{i_1\cdots i_k}{}^m\nabla_mB_{j_1\cdots j_l}\\
\text{(f)} &&&\displaystyle(A\circ B)_{i_1\cdots i_kj_1\cdots j_l}{}^m 
  = \sum_{\alpha=1}^k B_{j_1\cdots j_li_\alpha}{}^n A_{i_1\cdots n\cdots i_k}{}^m
\end{align*}
\end{definition}

\begin{theorem} \label{evolution}
Let $\hat M$ be a special graded manifold over the (pseudo) Riemannian manifold $M$. 
Consider the maps $\e:\X(M)\to\X(\hat M)_0$ and $\o:\Gamma S\to \X(\hat M)$ from \eqref{evenansatz} and \eqref{oddansatz1} and their restrictions to $\mathcal{K}_0=\{\text{Killing vector fields on }M\}$ and $\mathcal{K}_1=\{\text{Parallel spinors on }M\}$ respective. 
These data give rise to a supersymmetric Killing structure with central part $\mathfrak{Z}$ which can be  described uniformly by two types of elements. 

The first one, denoted by $\mathfrak{B}$, is s-like  and of the form \eqref{Be}. The second one, denoted by $\mathfrak{D}$, is v-like and of the form \eqref{De}. The dependence of $\mathfrak{B}$ and $\mathfrak{D}$ on the spinorial entries  is multilinear. The further entry is an endomorphisms which is recursively given by the curvature endomorphism via the operations given  in Definition \ref{operations}. 

Beside the usual brackets the remainig ones are
\begin{align*}
{(1)}&&& 
\big[\e(X),\mathfrak{B}^{(k)}(A;\xi_1,\ldots,\xi_k)\big]
 = \sum_{\alpha=1}^k \mathfrak{B}^{(k)}(A;\xi_1,\ldots,\e(X)\xi_\alpha,\ldots,\xi_k)
\\
{(2)}&&&
\big[\e(X),\mathfrak{D}^{(k)}(A;\xi_1,\ldots,\xi_k)\big]
=  \sum_{\alpha=1}^{k}\mathfrak{D}^{(k)}(A;\xi_1,\ldots,\e(X)\xi_\alpha,\dots,\xi_k)\big]
\\
{(3)}&&& 
\big[\o(\varphi),\mathfrak{B}^{(k)}(A;\xi_1,\ldots,\xi_k)\big] 
	=    \mathfrak{B}^{(k+1)}(\nabla A;\varphi,\xi_1,\ldots,\xi_k) 
	+(-)^k\mathfrak{D}^{(k+1)}(A ;\xi_1,\ldots,\xi_k,\varphi)
\\
{(4)}&&&
\big[\o(\varphi),\mathfrak{D}^{(k)}(A;\xi_1,\ldots,\xi_k)\big] 
	=  \mathfrak{D}^{(k+1)}(\nabla A;\varphi,\xi_1,\ldots,\xi_k) 
	  +(-)^k\mathfrak{B}^{(k+1)}(A\rfloor^1 R;\xi_1,\ldots,\xi_k,\varphi)
\\
{(5)}&&& 
\big[\mathfrak{B}^{(k)}(A;\xi_1,\ldots,\xi_k),\mathfrak{B}^{(l)}(B;\eta_1,\ldots,\eta_l)\big]
 =  \mathfrak{B}^{(k+l)}(A\bullet B;\xi_1,\ldots,\xi_k,\eta_1,\ldots,\eta_l)
\\
{(6)}&&& 
\big[\mathfrak{D}^{(k)}(A;\xi_1,\ldots,\xi_k),\mathfrak{D}^{(l)}(B;\eta_1,\ldots,\eta_l)\big]   
 = \mathfrak{B}^{(k+l)}(A\rfloor^1 B\rfloor^2 R;\xi_1,\ldots,\xi_k,\eta_1,\ldots,\eta_l) \\
&&& \hspace*{18em}   +\mathfrak{D}^{(k+l)}(\nabla_AB-\nabla_BA;\xi_1,\ldots,\xi_k,\eta_1,\ldots,\eta_l)
\\
{(7)}&&&
\big[\mathfrak{D}^{(k)}(A;\xi_1,\ldots,\xi_k),\mathfrak{B}^{(l)}(B;\eta_1,\ldots,\eta_l)\big] 
 = \mathfrak{B}^{(k+l)}(\nabla_AB;\xi_1,\ldots,\xi_k,\eta_1,\ldots,\eta_l)\\
&&& \hspace*{18em}+ \mathfrak{D}^{(k+l)}(A\circ B;\xi_1,\ldots,\xi_k,\eta_1,\ldots,\eta_l)
\end{align*}
We will denote this symbolically by 
 \begin{align*}
 {(1)}\ \ &  [\e,\mathfrak{B}]=\mathfrak{B},  &
 {(2)}\ \  &  [\e,\mathfrak{D}]=\mathfrak{D},&&\\ 
 {(3)}\ \ &  [\o,\mathfrak{B}]=\mathfrak{B}+\mathfrak{D}, &
 {(4)}\ \ &  [\o,\mathfrak{D}]=\mathfrak{B}+\mathfrak{D},&&\\
 {(5)}\ \ &  [\mathfrak{B},\mathfrak{B}]=\mathfrak{B},& 
 {(6)}\ \ &  [\mathfrak{D},\mathfrak{D}]=\mathfrak{B}+\mathfrak{D},&
 {(7)}\ \ & [\mathfrak{B},\mathfrak{D}]=\mathfrak{B}+\mathfrak{D}. 
 \end{align*}
\end{theorem}
From the property that all involved endomorphisms are build up via the curvature endomorphism we immediately get the following remark. It is intensively used in the proof of Theorem \ref{evolution}.
\begin{remark}\label{holbem}
The endomorphisms which occur as entries in the elements of the center take their values in the holonomy subbundle of ${\rm End}TM$.
\end{remark}
Another property which is almost immediate is the next special example.
\begin{corollary}\label{symmetricspace}
Let $\hat M$ be a special graded manifold whose reduction $M$ is a Riemannian symmetric space. Then the brackets reduce to 
\begin{align*}
{(3)}\ \ &\big[\o,\mathfrak{B}\big]=\mathfrak{D},&
{(4)}\ \ &\big[\o,\mathfrak{D}\big]=\mathfrak{B},&&\\
{(5)}\ \ &\big[\mathfrak{B},\mathfrak{B}\big]=\mathfrak{B},&
{(6)}\ \ &\big[\mathfrak{D},\mathfrak{D}\big]=\mathfrak{B},&  
{(7)}\ \ &\big[\mathfrak{B},\mathfrak{D}\big]=\mathfrak{D}\,.
\end{align*}
From these relations we read that the central part admits a further $\ZZZ_2$-grading, beside the one provided by $\X(\hat M)$. It is given by the two different  types of elements and not compatible with the standard $\ZZZ_2$-grading.
\end{corollary}
\begin{proof}\,[Theorem \ref{evolution}]
The proof of \eqref{thirdorderandeven} made  use of the fact that the curvature endomorphism is invariant under infinitesimal isometries of the base manifold $M$. Similar calculations  show part {\em 1.} and {\em 2.} of Theorem \ref{evolution} by using the following lemma.

\begin{lemma}
Let  $A\in\Gamma\big(\otimes^k TM\,\otimes \,{\rm End} TM\big)$ 
and $B\in\Gamma(\otimes^l TM\,\otimes\,{\rm End}TM\big)$ be invariant under  infinitesimal isometries of the 
(pseudo) Riemannian manifold $M$. Then the endomorphisms from the list in Definition \ref{operations} are also  invariant 
under infinitesimal automorphisms. 
\end{lemma}

\begin{proof}
We only have to prove the statement for {(a)},  {(b)}, and for the first summand in {(c)}, because it turns out  that the property holds for each summand in {(c)} separately.  Therefore, let $X$ be a Killing vector field on the base manifold $M$. Then we get
\begin{align*}
(\mathcal{L}_X \nabla A)(Y)&=(\mathcal{L}_X\nabla_YA-(\nabla A)([X,Y]))=\\
   &=[\mathcal{L}_X,\nabla_Y]A+\nabla_Y\underbrace{\mathcal{L}_X A}_{=0}-\nabla_{[X,Y]}A 
   \stackrel{\eqref{LieCov}}{=} 0
\end{align*}
The result for {(b)} and the bracket in {(c)} is an easy  consequence of  the Leibniz rule and $A\rfloor^{1}B(X_1,\ldots,X_k,Y_1,\ldots,Y_l)=B(A(X_1,\ldots,X_k)Y_1,Y_2,\ldots,Y_l)$. 
\end{proof}

The remaining calculations  {(3)} to {(7)}\ in theorem \ref{evolution} are extended versions of the calculations for \eqref{thirdorderandodd} and \eqref{26} given in Section \ref{Secthirdorder}. A key property is given in Remark \ref{holbem}. This gives zero for the action of the endomorphisms on the parallel spinors.
We will perform these calculations  at two examples, namely part  {(6)} and {(7)}
\begin{align*}
&\big[\mathfrak{D}^{(k)}(A;\xi_1,\ldots,\xi_k),\mathfrak{D}^{(l)}(B;\eta_1,\ldots,\eta_l)\big] \\
=\ &  \big[A_{(i)}{}^n\gamma^{(i)}\xi\otimes\nabla_n, B_{(j)}{}^m\gamma^{(j)}\eta\otimes\nabla_m\big]\\
=\ &  \gamma^{(i)}\xi\wedge \gamma^{(j)}\eta\otimes A_{(i)}{}^n\nabla_nB_{(j)}{}^m\nabla_m 
  +\gamma^{(i)}\xi\wedge\gamma^{(j)}\eta\otimes A_{(i)}{}^nB(e_{j_1},\ldots, e_{j_l},\nabla_ne^m)\nabla_m \\
&  +\gamma^{(i)}\xi\wedge\gamma^{(j)}\eta\owedge A_{(i)}{}^nB_{(j)}{}^m(R_{nm}+\nabla_{[e_n,e_m]}) 
 -(-)^{k+l}\gamma^{(j)}\eta\wedge\gamma^{(i)}\xi \otimes B_{(j)}{}^m\nabla_mA_{(i)}{}^n\nabla_n \\
&  -(-)^{k+l}\gamma^{(j)}\eta\wedge\gamma^{(i)}\xi \otimes B_{(j)}{}^nA(e_{i_1},\ldots, e_{i_k},\nabla_me^n)\nabla_n \\
=\ &  \gamma^{(i)}\xi\wedge\gamma^{(j)}\eta\owedge A_{(i)}{}^nB_{(j)}{}^m R_{nm}
  +\big( A_{(i)}{}^n\nabla_nB_{(j)}{}^m-B_{(j)}{}^n\nabla_nA_{(i)}{}^m \big)
     \gamma^{(i)}\xi\wedge \gamma^{(j)}\eta\otimes \nabla_m \\
=\ &   \mathfrak{B}^{(k+l)}(A\rfloor^1 B\rfloor^2 R;\xi_1,\ldots,\xi_k,\eta_1,\ldots,\eta_l)
 +\mathfrak{D}^{(k+l)}(\nabla_AB-\nabla_BA;\xi_1,\ldots,\xi_k,\eta_1,\ldots,\eta_l)
\end{align*}
and 
\begin{align*}
&\big[\mathfrak{D}^{(k)}(A;\xi_1,\ldots,\xi_k),\mathfrak{B}^{(l)}(B;\eta_1,\ldots,\eta_l)\big]\\
=\ & \big[A_{(j)}{}^m\gamma^{(j)}\xi\otimes\nabla_m,\gamma^{(i)}\eta\otimes B_{(i)}\big]\\
=\ &
\sum\nolimits_{\alpha} \gamma^{(j)}\xi \wedge\gamma^{i_1}\eta_1\wedge\cdots
\wedge(\nabla_m\gamma^{i_\alpha})\eta_\alpha\wedge\cdots\wedge\gamma^{i_l}\eta_l\owedge A_{(j)}{}^m B_{(i)}\\
&+ \gamma^{(j)}\xi\wedge\gamma^{(i)}\eta\owedge A_{(j)}{}^m [\nabla_m,B_{(i)}]  \\
&-(-)^{k+l}{\textstyle\sum_{\beta}} \gamma^{(i)}\eta \wedge\gamma^{j_1}\xi_1\wedge\cdots\wedge[B_{(i)},\gamma^{j_\beta}]\xi_\beta
\wedge\cdots\wedge\gamma^{j_k}\xi_k\otimes A_{(j)}{}^m\nabla_m \\
=\ & 
\gamma^{(j)}\xi\wedge\gamma^{(i)}\eta\owedge A_{(j)}{}^m(\nabla_mB)_{(i)} \\
& \quad 
-(-)^{k+l}{\textstyle\sum_{\beta}} \gamma^{(i)}\eta \wedge\gamma^{j_1}\xi_1\wedge\cdots\wedge 
B_{(i)}{}^{j_\beta}{}_n\gamma^{n}\xi_\beta\wedge\cdots\wedge\gamma^{j_k}\xi_k\otimes A_{(j)}{}^m\nabla_m \\ 
&=
\mathfrak{B}^{(k+l)}(\nabla_AB;\xi_1,\ldots,\xi_k,\eta_1,\ldots,\eta_l) \\
&\quad  
+(-)^{k+l}\gamma^{(i)}\eta \wedge\gamma^{j_1}\xi_1\wedge\cdots\wedge 
\gamma^{j_k}\xi_k\otimes 
{\textstyle\sum_{\beta}} B_{(i)}{}_{j_\beta}{}^{n}A_{j_1\cdots n\cdots j_k}{}^m\nabla_m\displaybreak[0] \\ 
=\ &
\mathfrak{B}^{(k+l)}(\nabla_AB;\xi_1,\ldots,\xi_k,\eta_1,\ldots,\eta_l)  
 +(-)^{k+l}\gamma^{(i)}\eta \wedge\gamma^{(j)}\xi\otimes  (A\circ B)_{(j)(i)}{}^m \nabla_m \\ 
=\ & 
\mathfrak{B}^{(k+l)}(\nabla_AB;\xi_1,\ldots,\xi_k,\eta_1,\ldots,\eta_l)
+\mathfrak{D}^{(k+l)}( A\circ B;\xi_1,\ldots,\xi_k,\eta_1\ldots,\eta_l) 
\end{align*}
We used the short abbreviation $\gamma^{(i)}\xi:=\gamma^{i_1}\xi_1\wedge\cdots\wedge\gamma^{i_k}\xi_k$ as well as the notation $(i)=(i_1,\ldots,i_k)$  for multi indices.
\end{proof}
\begin{proof}\,[Corollary \ref{symmetricspace}]
For symmetric spaces we have $\nabla R=0$. Therefore the statement is a consequence of the evolution of the involved endomorphism, see Theorem \ref{evolution}.
\end{proof}

\section{Further remarks}

\begin{itemize}[leftmargin=2em]
\item  {\it Products of pseudo Riemannian manifolds.}\  
Given  a product of two Riemannian manifolds $M=K\times L$ and assume one factor to be non flat. The space of parallel spinors $\mathcal{K}_1(M)$ is given by the tensor product of the two spaces of parallel spinors of the respective factors $\mathcal{K}_1(K)\otimes\mathcal{K}_1(L)$.  
This product structure of $\mathcal{K}_1$ yields  that even if the two summands have finite supersymmetric Killing structures this is not valid for the product, apart from one very special case. This is the case where one factor is   even dimensional with finite SSKS and the second factor is one dimensional.
Even in this very special case the  restriction to products where at least one factor is even dimensional was necessary. This is due to the fact that otherwise the doubling of the spinor bundle would lead to a doubling  of the amount of odd Killing fields. 
This would from the very beginning destroy the finiteness of the SSKS. The latter is similar to the discussion of the real SSKS on K\"{a}hler manifolds of quaternionic type.  

\item {\it Introducing Twistor spinors.}\  
The relation between twistor spinors and conformal vector fields with respect to the morphism $S\otimes S\to \X(\hat M)$ might suggests the extension of the supersymmetric Killing structures to these objects. This needs another ansatz of the SSKS because for our ansatz a closing of the algebra is only guaranteed if we restrict ourself to Killing vector fields. 
The  restriction of the even fields again would force us to restrict the set of odd fields to those twistor spinors which close into Killing vector fields. This, indeed, is a very strong restriction as the example of the flat space tells us. There the space of twistor spinors is of dimension $2^{2[\frac{D}{2}]}$, i.e.\ maximal, but the reduction would lead us to the subset of constant spinors so that we are left with the algebra discussed above. 

The problem to define superalgebras from twistor spinors was also recognized by \cite{Habermann1}. In that article it is shown that the set of imaginary Killing spinors does not lead to a superalgebra in general.  

\item  {\it Non vanishing field strength.}\  
The next step is  to introduce fields on the manifold $M$ which will enter in the Killing equation (motivated by gravity thories with nonvanishing field strengths). 
For the natural  extension of our ansatz the introduction of a nonvanishing right hand side of the Killing equation leads  to certain conditions on the involved fields. These properties arise, because we have to make sure the closure of the Killing structure. These conditions are, however, not uncommon and are assumend to  be related to the field equations of the involved fields.
\end{itemize}


\begin{appendix}

\section{Spin geometry}\label{secSpingeo}

\subsection{Clifford representations and transformations} 

On $V=\RRR^D$ consider the metric $g=(-1,\ldots [\text{t times}]\ldots, -1,1, \ldots [\text{s times}]\ldots ,1)$ with signature $\sigma=t-s$. The Clifford algebra associated to this metric is denoted by $\cl_{s,t}$ and we recall the defining relation
\begin{equation}
vw+wv=-2g(v,w).
\end{equation}
In  this context Euclidean signature is given by $(t=0,s=D)$ and Lorentzian signature by $(t=1,s=D-1)$.

A {\sc Clifford representation} is an irreducible representation of  $C\ell_{s,t}$. They are characterized by the signature $\sigma  \mod 8$ and are listed in Table \ref{clifford}.
In addition we list the irreducible representations of the complexification of the Clifford algebras given by $\cl^c_D=\cl_{s,t}\otimes\CCC=\cl_{D,0}\otimes\CCC$ which do only depend on the dimension $D$.
\begin{table}[htb]\caption{The Clifford representations}\label{clifford}\centering
$\displaystyle
\renewcommand{\arraystretch}{1.9}
\begin{array}{c||c|c|c|c}
\sigma &0&2&4&6\\\hline
{\cl}_{s,t}
 &\mathfrak{gl}_{2^{[\frac{D}{2}]}}\RRR
 &\mathfrak{gl}_{2^{[\frac{D}{2}]}}\RRR
 &\mathfrak{gl}_{2^{[\frac{D}{2}]-1}}\HHH
 &\mathfrak{gl}_{2^{[\frac{D}{2}]-1}}\HHH\\\hline\hline
\sigma&1&3&5&7\\\hline
\cl_{s,t} &\mathfrak{gl}_{2^{[\frac{D}{2}]}}\RRR\oplus\mathfrak{gl}_{2^{[\frac{D}{2}]}}\RRR
 &\mathfrak{gl}_{2^{[\frac{D}{2}]}}\CCC
 &\mathfrak{gl}_{2^{[\frac{D}{2}]-1}}\HHH\oplus\mathfrak{gl}_{2^{[\frac{D}{2}]-1}}\HHH
 &\mathfrak{gl}_{2^{[\frac{D}{2}]}}\CCC\\\hline\hline
D&\multicolumn{2}{c|}{\text{even}}&\multicolumn{2}{c}{\text{odd}}\\\hline
\cl_D^c&\multicolumn{2}{c|}{\mathfrak{gl}_{2^{[\frac{D}{2}]}}\CCC}&\multicolumn{2}{c}{\mathfrak{gl}_{2^{[\frac{D}{2}]}}\CCC\oplus\mathfrak{gl}_{2^{[\frac{D}{2}]}}\CCC}
\end{array}
$
\end{table}

These algebras are considered as acting on $S_\CCC=\CCC^{2^{[\frac{D}{2}]}}$ and the elements are called complex spinors.
A {\sc Spin representation} is an irreducible representation of $Spin(s,t)\subset\cl^+_{s,t} $. These representations only depend on the absolute value of the signature $|\sigma|$, because $\cl^+_{s,t}\ \simeq\ \cl^+_{t,s}.\label{2}$.
For $D$ odd the representation of the Clifford algebra is irreducible under $Spin(s,t)$.
For $D$ even  the representation splits into two irreducible representations $S_\CCC=S_\CCC^+\oplus S_\CCC^-$ of half the dimension. The elements are  called {\sc chiral} or {\sc Weyl spinors} and the projections are given by
\begin{equation}
p^\pm=\tfrac{1}{2}\big({\bf 1}\pm\,\zeta\gamma_{D+1}\big)\,,
\end{equation}
with $\zeta=1$ if $\sigma\equiv0\,{\rm mod}\,4$ and $\zeta=i$ if $\sigma\equiv2\,{\rm mod}\,4$. We look at the different cases:
\begin{itemize}[leftmargin=2em]
\item $|\sigma|=0 $. The projections $p^\pm$ are real such that the real representation of dimension $2^{[\frac{D}{2}]}$ splits into two real representations of dimension $2^{[\frac{D}{2}]-1}$.
\item $|\sigma|=1$, i.e.\  $\sigma=1,7$. We see that $\sigma=1$ gives a real representation of dimension $2^{[\frac{D}{2}]}$.
\item $|\sigma|=2$, i.e.\ $\sigma=2,6$. $\sigma=2$ has a real representation of dimension $2^{[\frac{D}{2}]}$ and $\sigma=6$ has a quaternionic representation. The projections are neither real nor quaternionic such that the representation on the splitting is complex. We get two complex representations of real dimension $2^{[\frac{D}{2}]}$.
\item $|\sigma|=3$, i.e.\ $\sigma=3,5$. $\sigma=3$ gives a quaternionic representation of real dimension $2^{[\frac{D}{2}]+1}$.
\item $|\sigma|=4$.  The projections are quaternionic such that we get a splitting of the quaternionic representation  in two quaternionic representations of real dimension $2^{[\frac{D}{2}]}$.
\end{itemize}

For the construction of real spinors we need some transformations which connect various equivalent Clifford representations. If we start with the representation given by the matrices $\{\gamma^\mu\}$, they are given by
\begin{align}
C^\dagger_\pm \gamma_\mu C_\pm &= \pm(-)^{t+1}\gamma_\mu^T\\
B^\dagger\gamma_\mu B &=  -(-)^t\gamma_\mu^\dagger\\
A_\pm^\dagger\gamma_\mu A_\pm & = \pm \gamma_\mu^*.
\end{align}
The transformations $A,B$ and the  {\sc charge conjugation} $C$ have following properties\footnote{The proofs of the various statements on the transformations in this subsection are an extension of the proof for the Lorentzian case, see \cite{Scherk1}. We start with $A_\pm$ and $B$ and the sign $\epsilon$ follows from counting the antisymmetric matrices in $\{C^\dagger\gamma^{\mu_1\cdots\mu_k}\}_{0\leq k\leq D}$.} 
 where the signs $\epsilon_\pm$ are from Table \ref{sign} and the missing of one indicates the non-existence of the transformation.
\begin{equation}
\begin{aligned}
C^\dagger_\pm  C_\pm &= A_\pm^\dagger A_\pm= B^\dagger B = {\bf 1}\,,
&  \qquad A_\pm A_\pm^* &=\epsilon_\pm{\bf 1}\,,\\
C_\pm^T &= (\pm)^t(-)^{\frac{1}{2}t(t-1)}\epsilon_\pm C_\pm\,,
&A_\pm^T &=\epsilon_\pm A_\pm \\
B =\gamma_1\cdots\gamma_t & = (-)^{\frac{1}{2}t(t-1)}B^\dagger \,, 
&B^\dagger\gamma_\mu B & =B\gamma_\mu B^\dagger\,,\\
C_\pm &=A_\pm B^*\,. &&
\end{aligned}
\label{chargecon1}
\end{equation}

Properties \eqref{chargecon1} don't depend on the signature but only on the dimension of the space, see Lemma \ref{propIndependence1} and \ref{propIndependence2}.
The charge conjugation is block diagonal, i.e.\ a map $S_\CCC^\pm\otimes S_\CCC^\pm\to \CCC$, for $D=0,4$  and off-blockdiagonal, i.e.\ $S_\CCC^\pm\otimes S_\CCC^\mp\to\CCC$, for $D=2,6$. 

\begin{table}[htb]\caption{The signs }\label{sign}\centering
$\displaystyle
\begin{array}{c||c|c|c|c|c|c|c|c}
\sigma &\ \ 0\ \ &\ \ 1\ \ &\ \ 2\ \ &\ \ 3\ \ &\ \ 4\ \ &\ \ 5\ \ &\ \ 6\ \ &\ \ 7\ \  \\\hline\hline
\epsilon_+& + & + & + & / & - & - & -&/\\ \hline
\epsilon_- & + & / & - & - &-  &/  &+& +
\end{array}
$
\end{table}

$A$ and  $B$  give rise to the  {\sc Dirac conjugate} and the {\sc charge conjugate} of a spinor $\eta\in S_\CCC$. They are given by
\begin{equation}
\bar\eta =\eta^\dagger B 
\quad \text{and}\quad  
\eta^C=A_\pm\eta^*,
\end{equation}
respectively. They are  connected via
\begin{equation}
\eta^C=C_\pm\bar\eta^T.\label{ch} 
\end{equation}

We remarked that the symmetry of the charge conjugation does not depent on the signature. This is also true for the higher $\mathfrak{spin}$-invariant morphisms $S\otimes S\to \Lambda^kV\otimes\CCC$ which are explicitly constructed as follows. Take the generators of the Clifford representation and consider the maps
\begin{equation}
\gamma^{\mu_1\ldots\mu_k}:S_\CCC\to\Lambda^kV\otimes S_\CCC \text{ with } 
\gamma^{\mu_1\ldots\mu_k}=
\gamma^{[\mu_1}\gamma^{\mu_2}\cdots\gamma^{\mu_k]}.\label{mord}
\end{equation}
Combined with the charge conjugation this gives trise to the morphism $S_\CCC\otimes S_\CCC\to \Lambda^kV\otimes\CCC$. 

\begin{lemma}\label{propIndependence1}
The symmetry $\Delta^k$ of the morphism \eqref{mord} is given by  
\begin{equation}\begin{aligned}
\Delta^0_\pm &= (\pm)^t(-)^{\frac{1}{2} t(t-1)}\epsilon_\pm  \\  
\Delta_\pm^k &= (\pm)^k(-)^{\frac{1}{2} k(k+2t+1)}\Delta^0_\pm 
\end{aligned}\end{equation}
with
\begin{equation}
\Delta_\pm^k\ =\ -\,\Delta_\pm^{k-2}.
\end{equation}
\end{lemma}
Because of the periodicity we only have to consider $k=0$ and $k=1$. If we in addition use $(-)^{\frac{1}{2}(2m)(2m-1)}=(-)^m$ we get the result which is given in Table \ref{summary}.
\begin{lemma}\label{propIndependence2}
$\Delta^k_\pm$ for $k=0,1$ is  given by\footnote{Here as well as in Table \ref{summary}, \ref{complexsusy} and \ref{realsusy} we use a slightly modified definition of the charge conjugation. 
This modification is given by $C_\pm$ goes to $C_\pm$ for $t$ odd, and to $C_\mp$ for $t$ even.\label{foot}}
\begin{center}
$\displaystyle\small{\renewcommand{\arraystretch}{1.9}
\begin{array}{c|c|c|c|c|c|c|c}
  D=0&1&2&3&4&5&6&7   \\\hline\hline
\Delta^0_\pm=+ & \Delta^0_+=+  & \Delta^0_\pm=\pm  & \Delta^0_-=- &
 \Delta^0_\pm=-  & \Delta^0_+=-   &  \Delta^0_\pm =\mp &  \Delta^0_-=+\\\hline
\Delta^1_\pm=\pm & \Delta^1_+=+  & \Delta^1_\pm=+ & \Delta^1_-=+ &
 \Delta^1_\pm=\mp  & \Delta^1_+=-   &  \Delta^1_\pm =- &  \Delta^1_-=-\\
\end{array}
}$\end{center}
\end{lemma}

\begin{table}[htb]\caption{The symmetry properties of the $\mathfrak{spin}$-invariant morphisms which are given by
$S_\CCC\otimes S_\CCC\to\Lambda^{k}V\otimes\CCC$ for $D$ odd, and by
 $S_\CCC^\pm\otimes S_\CCC^\pm\to \Lambda^{2m+1}V\otimes\CCC$
(resp.\ $S_\CCC^\pm\otimes S_\CCC^\mp\to \Lambda^{2m}V\otimes\CCC$)
in the chiral (c)
or  $S_\CCC^\pm\otimes S_\CCC^\mp\to \Lambda^{2m+1}V\otimes\CCC$
(resp.\ $S_\CCC^\pm\otimes S_\CCC^\pm\to \Lambda^{2m}V\otimes\CCC$)
in the non chiral (nc)  case for $D$ even.}\label{summary}\centering

$\displaystyle\renewcommand{\arraystretch}{1.9}
\begin{array}{c||c|c||c|c||c|c}
D & k=2m & k=0& k=2m+1 & k=1 & &\\\hline
 &  \multicolumn{2}{c||}{\Delta^{2m}} &  \multicolumn{2}{c||}{\Delta^{2m+1}}&\text{via}&
\\\hline
0  & (-)^m &+
    &\pm(-)^m&\pm&C_\pm& \text{nc}\\\hline
1 & (-)^m &+
   & (-)^m&+&C_+&\\\hline
2 &\pm(-)^m&\pm
   & (-)^m & +&C_\pm& \text{c}\\\hline
3 & -(-)^m &-
   & (-)^m& +&C_-&\\\hline
4 & -(-)^m &-
   &\mp(-)^m&\mp&C_\pm&\text{nc}\\\hline
5 & -(-)^m &-
   &-(-)^m &-&C_+&\\\hline
6 &\mp(-)^m &\mp
   &-(-)^m &-&C_\pm&\text{c}\\\hline
7 & (-)^m &+
   &-(-)^m& -&C_-
\end{array}
$
\end{table}


\subsection{Real spinors}\label{realspinorsapp}

So far we have constructed morphisms $A,B,C:S_\CCC\otimes S_\CCC\to\CCC$  based on the complex spinors. To get real supersymmetry, we need real spinors.

We have to distinguish the three cases of real, quaternionic and complex representations. For their description we need the notion of (anti-)conjugation.

\begin{remark}
A conjugation (resp.\ anti-conjugation) on a complex space $V$ is a map $\tau:V\to V$ with $\tau^2=1$ (resp.\ $\tau^2=-1$) which is  anti-linear, i.e. $\tau(az)=\bar a\tau(z)$ for all $a\in\CCC,z\in V$.

In our examples of the spinor spaces the charge conjugation yields the conjugation ($\epsilon_\pm=1$) or anti-conjugation ($\epsilon_\pm=-1$).
\end{remark}

\begin{itemize}[leftmargin=2em]
\item {{\it The Real Case:} $\sigma=1,0,7$}

The fact that we have a real representation of the spinors is reflected in the existence of a conjugation $\tau$ on $S_\CCC$ which commutes with the action of $Spin$. In this case the real part of $S_\CCC$ with respect to $\tau$ is  the space:
\begin{equation}
S=\left\{\left.z\in S_\CCC\right| z=\tau(z)\right\}=\,{\rm Re}(S_\CCC)\subset S_\CCC.
\end{equation}
Moreover in the case $\sigma=0$ the real structure is compatible with the chiral structure, i.e.\ $S=S^+\oplus S^-$.

\item {{\it The Complex Case:} $\sigma=2,6$}

In signature $\sigma=2$ we have a real representation. Because the representation only depends on the absolute value of the signature,  in both cases a conjugation $\tau$ exists which commutes with the action of $Spin$. This real structure does not respect the chiral splitting of $S_\CCC$. In fact, we have $\tau(S_\CCC^\pm)=S_\CCC^\mp$.
The space of real  spinors is given by
\begin{equation}\begin{aligned}
S&=\big\{z\in S_\CCC\ |\ \tau(z)=z\big\} 
  =\left\{\binom{z}{\tau(z)}\big|\ z\in S_\CCC^+\right\}\\
  &\subset {\rm span}_\CCC\{Q_\alpha,Q_{\dot \alpha}\}=S_\CCC
\end{aligned}\end{equation}
The action of $U(1)$ on $S_\CCC=S_\CCC^+\oplus S_\CCC^-$ given by
\begin{equation}
a(\phi^+,\phi^-)=(a\phi^+, \bar a\phi^-)\quad\text{for }a\in U(1)
\end{equation}
commutes with the conjugation $\tau$ and so yields an additional symmatry called  $R$-symmetry. 

\item {{\it The Quaternionic Case:} $\sigma=3,4,5$} 

Take a quaternionic representation $S_\CCC$ and consider it as complex as indicated. There exists an anti-conjugation $\tau_0:S_\CCC\to S_\CCC$ which commutes with the action of $Spin$.

On $\CCC^2$ we consider the anti-conjugation given by $\tau_1={\rm conj}\circ\Omega$ with $\Omega=\left(\begin{array}{cc}0&1\\-1&0\end{array}\right)$.

We define the space $\tilde S_\CCC=S_\CCC\oplus S_\CCC\simeq S_\CCC\otimes \CCC^2$ with conjugation $\tau=\tau_0\otimes\tau_1:\tilde S_\CCC\to \tilde S_\CCC$.
The real spinor space is\footnote{In fact, $S$ is a representation space for  $Spin\otimes SU(2)$, because the $R$-symmetry group 
$SU^*(2)\simeq SU(2)$ commutes with the action of $\tau_1$. 
The $R$-symmetry group can be a smaller group contained in $SU(2)$, namely $SO^*(2)\simeq U(1)$.\label{footR}}
\begin{equation}
S=\left\{\left.z\in\tilde S_\CCC\right|\,\tau(z)=z\right\}=\left\{\left.(z^1,z^2)\in\tilde S_\CCC\right|\,(-\tau_0(z^2),\tau_0(z^1))=(z^1,z^2)\right\}.\label{real}
\end{equation}
Moreover in the case $\sigma=4$ the quaternionic structure is compatible  with the chiral structure, i.e.\ $S=S^+\oplus S^-$ 
\end{itemize}


\section{Supersymmetry}\label{secSUSY}

In the introduction we drew a connection between supersymmetry and spinors which obey a Killing equation and in particular   we stressed the importance of parallel spinors. But we did this without explaining supersymmetry itself. In the  first part of this section we will shortly recall the flat super Poincar\'e algebra, before we turn to the construction of supersymmetry via spinors and extend this to curved manifolds. The results summed up in theorems \ref{whencomplex} and \ref{whenreal} are  similar to the results in \cite{dAuFerrLledVar} with a slight modification in the case of signatures 2 and 6 which have to be discussed seperately.

The famous result by Coleman and Mandula in \cite{ColMan} states  --  roughly --  that in a  physical theory the symmetry algebra is at most a direct sum of the Poincare algebra $\mathfrak{so}(3,1)\ltimes \RRR^{3,1}$ and an algebra of internal symmetries. This means that the generators of the latter do not have any spin indices but only multiplet indices. One requirement in this theorem is that the generators form an algebra with respect to the bracket. If we weaken this condition and allow  the generators to  form a superalgebra this direct sum behaviour is not longer true. It was shown in \cite{WessZumino2} that it is possible to construct a superalgebra with nontrivial mixing of internal and spinorial generators and proven in \cite{HaagLopusSohn} that the construction is the only one. The generators of this enlarged {\sc super Poincar\'e algebra} $\mathfrak{p}$ are:
\begin{itemize}[leftmargin=2em]
\item the even generators :
\begin{itemize}[leftmargin=2em]
\item Poincar\'e algebra: $E_{ij}$, $P_k$ for $1\leq i,j,k\leq 4$.
\item Internal symmetry $\mathfrak{su}(N)\oplus \mathfrak{u}(1)^M$: $X^a=(X^a)^A_B$ and  $x^{AB}=-x^{BA}=x^{AB}_o b^o$ for  $1\leq A,B\leq N$, $1\leq a\leq N^2-1$, $1\leq o\leq M$.
\end{itemize}
\item the odd generators: $Q_\alpha^A$ for $1\leq \alpha\leq 4$, $1\leq A\leq N$.
\end{itemize}
The odd generators behave as Majorana spinors under Lorentz transformations and the additional, nontrivial  brackets are given by 
\begin{equation}\begin{aligned}
{[}E^{ij},Q_\alpha^A] &= (\gamma^{ij})_{\alpha}^{\beta}Q^A_\beta\,,&
[P_k,Q_\alpha^A] &= 0 
\\
[Q_\alpha^A,Q_\beta^B] &= 2\gamma^k_{\alpha\beta}\delta^{AB}P_k+\epsilon_{\alpha\beta}x^{AB}\,, &
[X^a,Q_\alpha^A]&= (X^a)^A_BQ_\alpha^B\,. \label{classicalSUSY}
\end{aligned}\end{equation}
In this list $\gamma^k_{\alpha\beta}$ is the product of the charge conjugation matrix and the image of the standard basis under the spin representation $\rho$. 

Of course, this can easily be enlarged to arbritrary dimensions $D$ and signatures. We ``only'' have to make sure that the odd-odd commutator is indeed an anticommutator. This is to say the matrix  $\gamma^k_{\alpha\beta}$ is symmetric. Later in this appendix we will see when this construction is possible and  that this symmetry requirement can be weakened. Furthermore to form a real algebra the odd generators should be in a real representation.


\subsection{Complex supersymmetry}

Supersymmetry needs {symmetric} morphisms $S\otimes S\to V$ which will yield the supersymmetry brackets. 
In the case of complex supersymmetry the existence or non existence follows from  whether or  not the morphism for $k=1$ is symmetric, see Table  \ref{summary}. The result is given in Table \ref{complexsusy} and yields
\begin{theorem}\label{whencomplex}
Complex supersymmetry is possible in dimension $D=0,1,2,3,4\mod 8$.
\end{theorem}


\subsection{Real supersymmetry}

In the case of real supersymmetry we are able to construct supersymmetry although the charge conjugation itself has not the ``right'' symmetry property. This is due to the appearance of the additional $R$-symmetry group, in particular in the quaternionic case.

With regard to real supersymmetry we have to insist on two conditions
\begin{itemize}[leftmargin=2em]
\item The bracket has to be real.
\item The symmetry of the bracket  (which is not in all cases a priori given by the one of the charge conjugation).
\end{itemize}

\begin{proposition}\label{bilinear}
If we restrict ourself to the real spinors, the morphisms $S\otimes S\to \CCC$ and  $S\otimes S\to V\otimes \CCC$  take their values in $\RRR$ and  $V$, respectively (or, accidentally, in $i\RRR$ and $iV$),  if they are defined as follows in the different cases.
\begin{itemize}[leftmargin=2em]
\item In the real and complex case the morphisms are given by the restriction of $C$ to $S\subset S_\CCC$. The explicit form is given by $\eta\otimes\xi\mapsto \eta^TC^\dagger\gamma_\mu\xi $. We get
\begin{equation}\begin{aligned}
C(\eta,\xi)&=\bar\xi\eta \\
C(\eta,\gamma^\mu\xi)&=\bar\xi\gamma^\mu\eta.
\end{aligned}\end{equation}

\item In the quaternionic case\footnote{We denote elements in $S$ in the quaternionic case by capital Greek letters, and the entries in the pair by the corresponding small Greek letters, e.g.\ ${\rm H}=(\eta,\tau(\eta)), \Phi=(\phi,\tau(\phi))$} we take the morphism to be of the form\linebreak 
$\Phi\otimes\Psi\stackrel{C_\Omega}{\mapsto}\frac{1}{2}\Omega^{IJ}\phi_I^TC^\dagger\gamma_\mu\psi_J$
if $\Delta^1=-1$, or 
$\Phi\otimes\Psi\stackrel{C_\delta}{\mapsto}\frac{1}{2}\delta^{IJ}\phi_I^TC^\dagger\gamma_\mu\psi_J$
if $\Delta^1=+1$. 
For $\Delta^1=-1$ this yields
\begin{equation}\begin{aligned}
C_\Omega({\rm H},\Xi)&=\frac{1}{2}(\bar\xi\eta+\overline{\xi_C}\eta_C)\\
C_\Omega({\rm H},\gamma^\mu\Xi)&=\frac{\pm1}{2}(-)^{t+1}\big(\bar\eta\gamma^\mu\xi+\overline{\eta_C}\gamma^\mu\xi_C\big),
\end{aligned}\end{equation}
and for $\Delta^1=+1$
\begin{equation}\begin{aligned}
C_\delta({\rm H},\Xi)&=-\frac{1}{2}(\overline{\xi_C}\eta+\bar\xi \eta_C) \\
C_\delta({\rm H},\gamma^\mu\Xi)&=\frac{\pm1}{2}(-)^{t+1}
     \big(\overline{\eta_C} \gamma^\mu\xi+\bar\eta\gamma^\mu\xi_C\big).
\end{aligned}\end{equation}
\end{itemize}
\end{proposition}

\begin{remark}
\begin{enumerate}[leftmargin=2.5em]
\item[(1)] We recall the symmetry of the morphism $S\otimes S\to\Lambda^k$
\begin{equation}
\Delta^1_{\pm}\ =\ -\epsilon_\pm(\pm)^{t+1} (-)^{\frac{1}{2}t(t+1)}\label{decide}.
\end{equation} 
Consider for example  the quaternionic ($\epsilon_\pm=-$) \underline{and} symmetric ($\Delta^1=+$) case. \eqref{decide} is equivalent to  $ (\pm)^{t+1}=(-)^{\frac{1}{2}t(t+1)}$ which has no solution for Lorentzian space-times ($t=1$).
\item[(2)] We have the following relations between spinors and their charge conjugated
\begin{align}
\bar\xi\eta&=(\pm)^t(\overline{\xi_C}\eta_C)^\dagger\label{realityCharge}\\
\bar\xi\gamma^\mu\eta&=(\pm)^{t+1}(\overline{\xi_C}\gamma^\mu\eta_C)^\dagger\label{35}
\end{align}
\item[(3)]
We  mentioned in the last remark the accidental occurrence of different reality properties of the two morphisms when restricted to real spinors.
This in fact no accident, but due to
\[
\tau(\gamma^\mu\varphi)=A_\pm(\gamma^\mu\varphi)^*=\pm \gamma^\mu A_\pm\varphi^*
=\pm\gamma^\mu\varphi
\]
for $\varphi=\tau\varphi$. The sign reflects whether we choose $C_+$ or $C_-$ to define the charge conjugation.
Whenever we have to choose $C_-$ we consequently have to replace the usual morphism  $V\otimes S_\CCC\ni X\otimes\varphi\mapsto X\varphi\in S_\CCC$ induced by Clifford multiplication  by $V\otimes S\ni X\otimes\varphi\mapsto iX\varphi\in S$ to make sure that we have the same reality property in both morphisms.
\end{enumerate}
\end{remark}

The construction of the bracket  together with \eqref{decide} provides us the tool to decide when real supersymmetry is possible. The complete result is given in Table \ref{realsusy}.

\begin{theorem}\label{whenreal}
Real supersymmetry is possible in dimension $D=1,2,3\mod 8$ (independent of the signature) and furthermore in dimension and signature \linebreak $(D\ {\rm mod}\ 8, \sigma\ {\rm mod}\ 8)=(0,0)$, $(0,2)$, $(0,4)$, $(4,0)$, $(4,4)$, $(4,6)$, $(5,3)$, $(5,5)$, $(6,4)$, $(7,3)$ and $(7,5)$.
\end{theorem}

\begin{remark}
The pairs in the list in Theorem \ref{whenreal} can be seen as minimal supersymmetries obtained by charge conjugation, i.e.\ those which arise by considering the charge conjugation and the nature of the spinor representation (real, complex or quaternionic). If we allow extended supersymmetries $(N\geq 2)$ also in the non-quaternionic cases, or if we consider constructions of bilinears by taking into account further spin invariant automorphisms resulting from the algebraic nature of the spinors, we are able to fill some of the gaps in the list above (see \cite{AlekCor2}). 
\end{remark}

In Tables \ref{complexsusy} and \ref{realsusy} we list in which dimension and signature our construction via the charge conjugation leads to supersymmetry.
We use the following abbreviations:
\begin{itemize}[leftmargin=2em]
\item$Y_{(\alpha\beta)}$ means $Y=Y^T$, $Y_{[\alpha\beta]}$ means $Y=-Y^T$
\item Yes, No denotes whether we have supersymmetry or not.
\item Chiral, non chiral denotes the type of the morphism $S\otimes S\to V$.
\item r,c or q denote the reality type of the spin representation (real, complex or quaternionic).
\item M, MW, W or D denote the type of the spinors (Majorana, Majorana-Weyl, Weyl or Dirac).
\end{itemize}
In Table \ref{realsusy} we shaded the diagonals for Lorentzian signature ($D=2-\sigma$) and Euclidean signature which is  ($D=-\sigma$). We see that in all Lorentzian cases supersymmetry is present, whereas in Euclidean signature there is no supersymmetry in $D=5$ and $6$.
\begin{table}[htb]\caption{Complex supersymmetry from charge conjugation}\label{complexsusy}\centering
$\displaystyle
{\renewcommand{\arraystretch}{1.9}
\begin{array}{c|cl||c|cl}
D=0&C_+     &\text{Yes, non chiral}\quad&\quad D=4&C_-&\text{Yes, non chiral} \\
\hline
D=1&C_+     &{\rm Yes}                      &\quad D=5&&{\rm No}\\
\hline
D=2&C_\pm&\text{Yes, chiral }       &\quad D=6&&{\rm No}\\
\hline
D=3&C_-      &{\rm Yes}                      &\quad D=7&&{\rm No}
\end{array}}
$
\end{table}

\begin{landscape}
\begin{table}[p]\caption{Real supersymmetry from charge conjugation}\label{realsusy}
${\renewcommand{\arraystretch}{1.2}\tiny
\begin{array}{c||c|c|c|c|c|c|c|c||c}
 & \sigma=0 &\sigma=1&\sigma=2&\sigma=3&\sigma=4&\sigma=5&\sigma=6&\sigma=7& \\
\LCC
& \grau && \hellgrau  & & & & & & \\
\hline\hline
D=0
 &\epsilon_+=1 & &\epsilon_+=1 & &\epsilon_-=-1 & &\epsilon_-=1 & &\\
 &(C_+^\dagger\gamma^\mu)_{(\alpha\beta)} &
 & (C_+^\dagger\gamma^\mu)_{(\alpha\beta)} &
 &(C_-^\dagger\gamma^\mu)_{[\alpha\beta]} &
 & (C_-^\dagger\gamma^\mu)_{[\alpha\beta]} & &\\
 &{\rm Yes}& &{\rm Yes} & &{\rm Yes},\ R=SU(2) & & {\rm No} & &C_\pm^T=C_\pm\\
\cline{2-2}\cline{4-4}\cline{6-6}\cline{8-8}
 &\epsilon_-=1 & &{\epsilon_-=-1 }& &\epsilon_+=-1 & &\epsilon_+=-1 & & \text{non chiral} \\
&(C_-^\dagger\gamma^\mu)_{[\alpha\beta]} &
& (C_-^\dagger\gamma^\mu)_{[\alpha\beta]}&
&(C_+^\dagger\gamma^\mu)_{(\alpha\beta)} &
 & (C_+^\dagger\gamma^\mu)_{(\alpha\beta)} & &\\
 &{\rm No}& &{\rm No} & &{\rm Yes},\ R=U(1) & & {\rm No} & &\\
\ECC
\LCC
& &\hellgrau & & & & & &\grau & \\ 
\hline
 D=1
 &&{\epsilon_+=1 }& &\epsilon_+=-1 & &\epsilon_+=-1 & &\epsilon_+=1 &\\
 &&{(C_+^\dagger\gamma^\mu)_{(\alpha\beta)}} &
 & (C_+^\dagger\gamma^\mu)_{(\alpha\beta)} &
 &(C_+^\dagger\gamma^\mu)_{(\alpha\beta)} &
 & (C_+^\dagger\gamma^\mu)_{(\alpha\beta)} &C_+^T=C_+\\
&&{\rm Yes}& &{\rm Yes},\ R=U(1) & &{\rm Yes},\ R=U(1) & & {\rm Yes} & \\
\ECC
\LCC
& \hellgrau & & & & & &\grau & & \\ 
\hline
D=2
 &{\epsilon_+=1} & &\epsilon_-=1 & &\epsilon_-=-1 & &\epsilon_+=1 & &\\
 &(C_+^\dagger\gamma^\mu)_{(\alpha\beta)} &
& (C_-^\dagger\gamma^\mu)_{(\alpha\beta)} &
 &(C_-^\dagger\gamma^\mu)_{(\alpha\beta)} &
 & (C_+^\dagger\gamma^\mu)_{(\alpha\beta)} & &\\
 &{\rm Yes}& &{\rm Yes} & &{\rm Yes},\ R=U(1) & & {\rm Yes} & &C_\pm^T=\pm C_\pm\\
\cline{2-2}\cline{4-4}\cline{6-6}\cline{8-8}
 &{\epsilon_-=1 }& &\epsilon_+=-1 & &\epsilon_+=-1 & &\epsilon_-=-1 & & \text{chiral} \\
&{(C_-^\dagger\gamma^\mu)_{(\alpha\beta)}} &
 & (C_+^\dagger\gamma^\mu)_{(\alpha\beta)} &
 &(C_+^\dagger\gamma^\mu)_{(\alpha\beta)} &
 & (C_-^\dagger\gamma^\mu)_{(\alpha\beta)} & &\\
 &{{\rm Yes}}& &{\rm No} & &{\rm Yes},\ R=U(1) & & {\rm No} & &\\
\ECC
\LCC
& && & & &\grau & &\hellgrau & \\ 
\hline
D=3
 &&\epsilon_-=1 & &\epsilon_-=-1 & &\epsilon_-=-1 & &\epsilon_-=1 &\\
 &&(C_-^\dagger\gamma^\mu)_{(\alpha\beta)} &
 & (C_-^\dagger\gamma^\mu)_{(\alpha\beta)} &
 &(C_-^\dagger\gamma^\mu)_{(\alpha\beta)} &
 & (C_-^\dagger\gamma^\mu)_{(\alpha\beta)} &C_-^T=-C_-\\
&&{\rm Yes}& &{\rm Yes},\ R=U(1) & &{\rm Yes},\ R=U(1) & & {\rm Yes} & \\
\ECC
\LCC
& & & & &\grau & &\hellgrau & & \\ 
\hline
D=4
 &\epsilon_-=1 & &\epsilon_+=1 & &\epsilon_+=-1 & &\epsilon_-=1 & &\\
 &(C_-^\dagger\gamma^\mu)_{(\alpha\beta)} &
 & (C_+^\dagger\gamma^\mu)_{[\alpha\beta]} &
 &(C_+^\dagger\gamma^\mu)_{[\alpha\beta]} &
 & (C_-^\dagger\gamma^\mu)_{(\alpha\beta)} & &\\
 &{\rm Yes}& &{\rm No} & &{\rm Yes},\ R=SU(2) & & {\rm Yes} & &C_\pm^T=-C_\pm\\
\cline{2-2}\cline{4-4}\cline{6-6}\cline{8-8}
 &\epsilon_+=1 & &\epsilon_-=-1 & &\epsilon_-=-1 & &\epsilon_+=-1 & & \text{non chiral} \\
&(C_+^\dagger\gamma^\mu)_{[\alpha\beta]} &
 & (C_-^\dagger\gamma^\mu)_{(\alpha\beta)} &
 &(C_-^\dagger\gamma^\mu)_{(\alpha\beta)} &
 & (C_+^\dagger\gamma^\mu)_{[\alpha\beta]} & &\\
 &{\rm No}& &{\rm No} & &{\rm Yes},\ R=U(1) & & {\rm No} & &\\
\ECC
\LCC
& & & &\grau & &\hellgrau & & & \\
\hline
D=5
 &&\epsilon_+=1 & &\epsilon_+=-1 & &\epsilon_+=-1 & &\epsilon_+=1 &\\
 &&(C_+^\dagger\gamma^\mu)_{[\alpha\beta]} &
 & (C_+^\dagger\gamma^\mu)_{[\alpha\beta]} &
 &(C_+^\dagger\gamma^\mu)_{[\alpha\beta]} &
 & (C_+^\dagger\gamma^\mu)_{[\alpha\beta]} &C_+^T=-C_+\\
&&{\rm No}& &{\rm Yes},\ R=SU(2) & &{\rm Yes},\ R=SU(2) & & {\rm No} & \\
\ECC
\LCC
& &&\grau & &\hellgrau & & & & \\ 
\hline
D=6
 &\epsilon_+=1 & &\epsilon_-=1 & &\epsilon_+=-1 & &\epsilon_+=1 & &\\
 &(C_+^\dagger\gamma^\mu)_{[\alpha\beta]} &
 & (C_-^\dagger\gamma^\mu)_{[\alpha\beta]} &
 &(C_+^\dagger\gamma^\mu)_{[\alpha\beta]} &
 & (C_+^\dagger\gamma^\mu)_{[\alpha\beta]} & &\\
 &{\rm No}& &{\rm No} & &{\rm Yes},\ R=SU(2) & & {\rm No} & &C_\pm^T=\mp C_\pm\\
\cline{2-2}\cline{4-4}\cline{6-6}\cline{8-8}
 &\epsilon_-=1 & &\epsilon_+=-1 & &\epsilon_-=-1 & &\epsilon_-=-1 & & \text{chiral} \\
&(C_-^\dagger\gamma^\mu)_{[\alpha\beta]} &
 & (C_+^\dagger\gamma^\mu)_{[\alpha\beta]} &
 &(C_-^\dagger\gamma^\mu)_{[\alpha\beta]} &
 & (C_-^\dagger\gamma^\mu)_{[\alpha\beta]} & &\\
 &{\rm No}& &{\rm No} & &{\rm Yes},\ R=SU(2) & & {\rm No} & &\\
\ECC
\LCC
& &\grau & &\hellgrau & & & & & \\
\hline
D=7
 &&\epsilon_-=1 & &\epsilon_-=-1 & &\epsilon_-=-1 & &\epsilon_-=1 &\\
 &&(C_-^\dagger\gamma^\mu)_{[\alpha\beta]} &
 & (C_-^\dagger\gamma^\mu)_{[\alpha\beta]} &
 &(C_-^\dagger\gamma^\mu)_{[\alpha\beta]} &
 & (C_-^\dagger\gamma^\mu)_{[\alpha\beta]} &C_-^T=C_-\\
&&{\rm No}& &{\rm Yes},\ R=SU(2) & &{\rm Yes},\ R=SU(2) & & {\rm No} & \\ \hline\hline
\ECC
&{\rm r,\ MW} &{\rm r,\ M}&{\rm c,\ M\ o.\ W}&{\rm q,\ W}&{\rm q,\ W}&{\rm q,\ D}
&{\rm c,\ M\ o.\ W}&{\rm r,\ M}&
\end{array}
}$
\end{table}
\end{landscape}


\newpage 
\section{Conventions in dimension four}\label{SecDim4}

We consider the four dimensional Euclidean space. In this case we have a quaternionic spin representation ($\epsilon_\pm=-1$) which is compatible with the chiral structure. 
Our convention for describing the quaternions $\HHH$ is
\begin{equation}\begin{aligned}
\CCC^2\ni & (a_1+ia_3,a_2+ia_4)\\ 
	&\stackrel{\simeq}{\longleftrightarrow}(a_1+Ja_3)+I(a_2+Ja_4)=a_1+a_2I+a_3J+a_4K\in\HHH\,.
\end{aligned}\end{equation}
With
\begin{equation}
\CCC^{2n}\ni(z_1,\ldots,z_n,z_{n+1},\ldots,z_{2n})\stackrel{\simeq}{\longleftrightarrow}(z_1+Iz_{n+1},\ldots,z_{n}+Iz_{2n})\in\HHH^n
\nonumber
\end{equation}
the multiplication with $I,J,K$ are given by the following complex $2n\times2n$ matrices
\begin{align}
I&=\begin{pmatrix}&-{\bf 1}\\{\bf 1}&\end{pmatrix},&
J&=\begin{pmatrix}i{\bf 1}&\\&-i{\bf 1}\end{pmatrix},&
K&=\begin{pmatrix}&i{\bf 1}\\i{\bf 1}&\end{pmatrix}.
\end{align}
The spinors are given by
\begin{equation}
S_\CCC=\CCC^4=S^+_\CCC\oplus S_\CCC^-=\HHH\oplus\HHH
\end{equation}
with $S_\CCC^+={\rm span}_\CCC\{e_1,e_3\},S_\CCC^-={\rm span}_\CCC\{e_2,e_4\}$ by the above conventions.

The $\gamma$-matrices in chiral representation are given by
\begin{equation}\begin{aligned}
\gamma_1 &=\begin{pmatrix}&-1\\1&\end{pmatrix}
=\footnotesize\begin{pmatrix}&-1&&\\1&&&\\&&&-1\\&&1&\end{pmatrix},&
\gamma_2 &=\begin{pmatrix}&I\\I&\end{pmatrix}
=\footnotesize\begin{pmatrix}&&&-1\\&&-1&\\&1&&\\1&&&\end{pmatrix},\\
\gamma_3 &=\begin{pmatrix}&J\\J&\end{pmatrix}
=\footnotesize\begin{pmatrix}&i&&\\i&&&\\&&&-i\\&&-i&\end{pmatrix},&
\gamma_4 &=\begin{pmatrix}&K\\K&\end{pmatrix}
=\footnotesize\begin{pmatrix}&&&i\\&&i&\\&i&&\\i&&&\end{pmatrix}.
\end{aligned}\end{equation}
By using the modified convention from footnote \ref{foot} the charge conjugation is given by 
\begin{equation}
C_\pm=\footnotesize \begin{pmatrix}&&1&\\&&&\mp1\\-1&&&\\ &\pm1&&\end{pmatrix}.
\end{equation}
We modify the $\mathfrak{so}$-generators using the self duality oparation $\Lambda^2\to\Lambda^2$ with $F_{ij}\mapsto\frac{1}{2}\varepsilon_{ijkl}F^{kl}$. The results $\tilde\gamma_{ij}=\frac{1}{2}\big(\gamma_{ij}-\frac{1}{2}\varepsilon_{ijkl}\gamma^{kl}\big)$ are obviously anti-self dual, i.e.\ $\tilde\gamma_{ij}=-\frac{1}{2}\varepsilon_{ijkl}\tilde\gamma^{kl}$
\begin{gather}
\tilde\gamma_{12} =
\begin{pmatrix}-I&0\\0&0\end{pmatrix}=
\begin{pmatrix}0&0&1&0\\0&0&0&0\\-1&0&0&0\\0&0&0&0\end{pmatrix}
\quad
\tilde\gamma_{13} =
\begin{pmatrix}-J&0\\0&0\end{pmatrix}=
\begin{pmatrix}-i&0&0&0\\0&0&0&0\\0&0&i&0\\0&0&0&0\end{pmatrix}
\nonumber\\
\tilde\gamma_{14} =
\begin{pmatrix}-K&0\\0&0\end{pmatrix}=
\begin{pmatrix}0&0&-i&0\\0&0&0&0\\-i&0&0&0\\0&0&0&0\end{pmatrix}\label{conn2}
\end{gather}

\end{appendix}

\end{document}